\documentclass[12pt,reqno]{amsart}

\usepackage{amssymb,amsmath,amsthm,mathtools,wasysym,calc,verbatim,enumitem,tikz,pgfplots,hyperref,url,mathrsfs,cite,fullpage,bbm,comment}
\usepackage{comment}
\mathtoolsset{showonlyrefs}
\pgfplotsset{compat=1.18}
\newtheorem{theorem}{Theorem}

\newtheorem{lemma}[theorem]{Lemma}
\newtheorem{claim}[theorem]{Claim}
\newtheorem{corollary}[theorem]{Corollary}

\newtheorem{conj}[theorem]{Conjecture}

\newtheorem{observation}[theorem]{Observation}

\newtheorem*{claim*}{Claim}

\theoremstyle{remark}
\newtheorem*{remark*}{Remark}

\numberwithin{theorem}{section}

\renewcommand{\phi}{\varphi}

\renewcommand{\leq}{\le}
\renewcommand{\geq}{\ge}

\newcommand{\eps}{\varepsilon}

\newcommand{\cF}{\mathcal F}
\newcommand{\cE}{\mathcal E}

\newcommand{\cX}{\mathcal X}
\newcommand{\cY}{\mathcal Y}
\newcommand{\cA}{\mathcal A}
\newcommand{\cB}{\mathcal B}
\newcommand{\cD}{\mathcal{D}}
\newcommand{\cC}{\mathcal C}
\newcommand{\cL}{\mathcal L}

\newcommand{\cH}{\mathcal H}

\newcommand{\cT}{\mathcal T}

\newcommand{\cI}{\mathcal I}
\newcommand{\E}{\mathbb{E}}

\newcommand{\Var}{\operatorname{Var}}
\def\1{\mathbbm{1}}

\def\g{{\gamma}}

\renewcommand{\le}{\leqslant}
\renewcommand{\ge}{\geqslant}
\renewcommand{\P}{\mathbb{P}}

\newcommand{\EE}{\mathbb E}

\newcommand{\PP}{\mathbb P}

\title
{A new lower bound for the Ramsey numbers $R(3,k)$} 
\author{Marcelo Campos, Matthew Jenssen, Marcus Michelen, Julian Sahasrabudhe}

\address{Instituto de Matem\'{a}tica Pura e Aplicada (IMPA).}
\email{marcelo.campos@impa.br}

\address{King's College London, Department of Mathematics.}
\email{matthew.jenssen@kcl.ac.uk}
\address{University of Illinois, Chicago. Dept of Mathematics, Statistics and Computer science.}
\email{michelen.math@gmail.com}
\address{University of Cambridge. Department of Pure Mathematics and Mathematical Statistics.} 
\email{jdrs2@cam.ac.uk}

\begin{document}

\begin{abstract}
We prove a new lower bound for the off-diagonal Ramsey numbers,
\[ R(3,k) \geq \bigg( \frac{1}{3}+ o(1) \bigg) \frac{k^2}{\log k }\, , \]
thereby narrowing the gap between the upper and lower bounds to a factor of $3+o(1)$. This improves the best known lower bound of $(1/4+o(1))k^2/\log k$ due, independently, to Bohman and Keevash, and Fiz Pontiveros, Griffiths and Morris, resulting from their celebrated analysis of the triangle-free process. As a consequence, we disprove a conjecture of Fiz Pontiveros, Griffiths and Morris that the constant $1/4$ is sharp.
\end{abstract}

\maketitle

\vspace{-2em}

\section{Introduction} 
The Ramsey number $R(\ell,k)$, is defined to be the minimum $n$ such that every red/blue colouring of the edges of the complete graph $K_n$ contains either a blue $K_{\ell}$ or a red $K_{k}$. These numbers were first proved to exist by Ramsey~\cite{Ramsey1930} in 1930 and the first reasonable upper bounds were given by Erd\H{o}s and Szekeres~\cite{erdos1935combinatorial} in 1935, who showed that $R(\ell,k) \leq \binom{\ell+k-2}{\ell-1}$. In the years since, it has become a major research direction to understand the asymptotic growth of these numbers. Perhaps the best known are the \emph{diagonal} Ramsey numbers $R(k)= R(k,k)$. Here the best bounds are 
\[  2^{k/2 + o(k)} \leq R(k) \leq (4-c)^{k}, \] for some $c>0$. The lower bound was proved by Erd\H{o}s~\cite{erdos1947some} in 1947 and the upper bound is due to a recent work of Campos, Griffiths, Morris and Sahasrabudhe~\cite{campos2023exponential}, giving the first exponential improvement since the seminal work of Erd\H{o}s and Szekeres~\cite{erdos1935combinatorial}. This recently has been improved further to $c \approx 0.2$ in \cite{gupta2024optimizing}.

Another line of research has focused on the extreme off-diagonal Ramsey numbers $R(3,k)$. The study of $R(3,k)$ has a rich history (which we detail below in Section~\ref{sec:history}) and has inspired the development of a number of important combinatorial tools. Today the Ramsey numbers $R(3,k)$ represent one of the successes of the field, and have been determined up to a factor of $4+o(1)$, the best known bounds being
\begin{equation}\label{eq:R3k-bounds} \hspace{-1.5em} \bigg( \frac{1}{4} +o(1) \bigg) \frac{k^2}{\log k } \leq R(3,k) \leq \big( 1+o(1)\big) \frac{k^2}{\log k}. \end{equation}
Here the upper bound is due to Shearer~\cite{shearer1983note}, from 1983, who built on the seminal papers of Ajtai, Koml\'os and Szemer\'{e}di~\cite{ajtai1980note,ajtai1981dense}, which introduced the extremely influential ``semi-random method''. The lower bound is due, independently, to the celebrated works of Bohman and Keevash~\cite{bohman2021dynamic} and Fiz Pontiveros, Griffiths and Morris~\cite{FGM}. In both of these papers the authors establish this lower bound by studying the \emph{triangle-free process}, a random process introduced by Bollob\'{a}s and Erd\H{o}s (see \cite{FGM}) for generating a random triangle-free graph. 

In their paper \cite{FGM}, Fiz Pontiveros, Griffiths and Morris conjectured that their lower bound \eqref{eq:R3k-bounds} is sharp. In this paper we obtain a new lower bound on $R(3,k)$, disproving their conjecture, and narrowing the gap between the upper and lower bound to a factor of $3+o(1)$.
\begin{theorem}\label{thm:r3k} 
$$R(3,k)\geq \bigg( \frac{1}{3} +o(1)\bigg)\frac{k^2}{\log k}\, .$$
\end{theorem}

\vspace{2mm}

One particularly interesting consequence of Theorem~\ref{thm:r3k} is that it strongly suggests that the natural triangle-free process is not optimal for Ramsey (the authors of \cite{bohman2021dynamic,FGM} only prove an \emph{upper bound} on the independence number for the triangle-free process, however this bound is almost certainly tight) and rather suggests that the optimal graphs mix randomness with some underlying structure. 

The main idea behind Theorem~\ref{thm:r3k} is that there exist triangle-free graphs that are denser than the triangle-free process but still sufficiently pseudo-random to keep the independence number small. Our construction is easy to describe: we first sample a random graph on $n/(\log n)^2$ vertices at density $p_0 = \alpha_0 (\log n/n)^{1/2}$, clean out triangles, blow it up and then run a \emph{variant} of the triangle-free process with the resulting graph. For reasons that we will sketch in Section~\ref{sec:heuristic} this results in a  denser graph, with  smaller independence number. 

The reason we work with a variant of the triangle-free process is that the triangle-free process exhibits a complicated dynamics that requires an extremely subtle analysis. For example, in both \cite{bohman2021dynamic} and \cite{FGM}, the authors need to show the highly non-trivial interaction and ``self correction'' properties of certain martingales associated with the process. By contrast, in our process, we are able to avoid these
subtleties entirely, by manually ``steering'' the process to ensure it follows a simpler trajectory. This results in a greatly simplified analysis and more straightforward proof. In particular, if we apply our version of the triangle-free process without our initial ``seed'' step where we take a blow up of a random graph, this paper gives a much simpler proof of the lower bound~\eqref{eq:R3k-bounds} on $R(3,k)$ proved in  \cite{bohman2021dynamic,FGM}.

Interestingly, for $\ell \geq 4$ our understanding of $R(\ell,k)$ drops off considerably. In a recent breakthrough, Mattheus and Verstra\"{e}te \cite{mattheus2024asymptotics} gave an ingenious construction, combining both algebraic and probabilistic ideas, to determine 
$R(4,k)$ up to poly-logarithmic factors. For $R(\ell,k)$, where $\ell \geq 5$ is fixed and $k \rightarrow \infty$, the best known lower bounds are due to the celebrated work of Bohman and Keevash \cite{bohman2010early} on the $K_{\ell}$-free process. There they gave a poly-logarithmic improvement over the local lemma argument from the 1970s due to Spencer~\cite{spencer1977asymptotic}. However, in these cases there is still a $k^{\ell/2+o(1)}$ factor gap between the upper and lower bounds. 

In the remainder of this section, we describe the history of $R(3,k)$.

\subsection{A history of the problem} \label{sec:history}
The Ramsey numbers were introduced in 1930 in the pioneering work of Ramsey \cite{Ramsey1930}, who defined and showed the existence of the diagonal Ramsey numbers $R(k)$. A few years later the off-diagonal Ramsey numbers $R(\ell,k)$ were introduced by Erd\H{o}s and Szekeres~\cite{erdos1935combinatorial}, who went on to show that $R(\ell,k) \leq \binom{\ell+k-2}{\ell-1}$. In particular, when $\ell =3$, this implies that $R(3,k) \leq k^2$.

In 1968 this was improved by Graver and Yackel~\cite{graver1968some} to $O(k^2 \log \log k /\log k)$ and then, in 1980, by Ajtai, Koml\'os and Szemer\'{e}di~\cite{ajtai1980note,ajtai1981dense}, to 
\[ R(3,k) \leq  \frac{ck^2}{\log k},\] for some $c>0$, in two highly influential papers that introduced the ``semi-random method''. Soon after, this was improved by Shearer~\cite{shearer1983note} to 
\[ \hspace{2em} R(3,k) \leq \big( 1+o(1) \big)\frac{k^2}{\log k},\] which still stands as the best known upper bound for $R(3,k)$.

The first progress on the \emph{lower} bound is due to Erd\H{o}s in the 1950s who showed $R(3,k) = \Omega(k^{1+\eps})$, for some $\eps >0$, by giving an explicit construction~\cite{erdos1957remarks}. In 1959, Erd\H{o}s proved that $R(3,k) = \Omega(k^{3/2})$ by using a probabilistic argument before proving, in 1961, that 
\[R(3,k) \geq \frac{ck^2}{(\log k)^2},\]
for some $c>0$, by ingeniously modifying a sample of the binomial random graph~\cite{erdos1961graph}. Different proofs of this result were then given by Spencer~\cite{spencer1977asymptotic}, Bollob\'as~\cite{bollobas2001random}, Erd\H{o}s, Suen and Winkler~\cite{erdos1995size}, and Krivelevich~\cite{krivelevich1995bounding}
before the breakthrough work of Kim~\cite{kim1992ramsey}, in 1995, established
\begin{equation}\label{eq:kim} R(3,k)\geq \frac{ck^2}{\log k},\end{equation}
for some $c> 0$, thereby determining $R(3,k)$ up to constant factors. 

Kim constructed the desired graph by iteratively ``sprinkling'' random edges into an empty graph in such a way that the edges don't form triangles with edges from previous rounds. He then used martingale concentration to show that the various parameters of the graph evolved as expected, according to a certain differential equation.

In 2008, a different and influential proof of this result was given by Bohman \cite{bohman2009triangle}, who showed that the graph produced by the ``triangle-free process'' could also be used to prove Kim's lower bound on $R(3,k)$. The triangle-free process is a random graph process first defined by Bollob\'as and Erd\H{o}s in 1990 (see \cite{FGM}). To define the triangle-free process on $[n] = \{1,\ldots ,n\}$, one starts with the empty graph $G_0$ on vertex set $[n]$ and then inductively defines $G_{i+1}$ to be $G_i + e_{i+1}$ where $e_{i+1}$ is chosen uniformly among all edges $e$ for which $G_i + e$ is triangle free, until there is no such edge, in which case the process stops. 

In the difficult and influential papers \cite{bohman2021dynamic, FGM}, Bohman and Keevash and, independently, Fiz Pontiveros, Griffiths and Morris, studied the trajectory of this process all the way to its (asymptotic) end, showing that the number of edges in the terminating graph $G_{\tau}$ is 
\begin{equation}\label{eq:density-of-tfp} e(G_{\tau}) = \bigg( \frac{1}{2\sqrt{2}} + o(1) \bigg)n^{3/2}\sqrt{\log n}\, , \end{equation}
with probability $1-o(1)$. They also proved the independence number of the terminating graph satisfies
\begin{equation}\label{eq:alpha-tfp} \hspace{-2em} \alpha(G_{\tau}) \leq (\sqrt{2} +o(1) )\sqrt{n\log n} \end{equation}
(one actually expects equality in this case) thereby implying the bound 
\[\hspace{-3em}  R(3,k) \geq \bigg( \frac{1}{4} + o(1) \bigg) \frac{k^2}{\log k }\, . \]
We remark that while a large part of the motivation for studying the triangle-free process, and other related processes, is derived from the Ramsey numbers, it is far from the only motivation. 
Indeed, these subtle processes are fascinating objects of study in their own right and have been the subject of intense study in recent years~\cite{erdos1995size, bollobas2000constrained, osthus2001random, bohman2009triangle, bohman2010early, warnke2011dense, warnke2014cycle, bennett2016note, FGM, bohman2021dynamic}. More generally, the $H$-free process was studied by Bohman and Keevash \cite{bohman2010early}, who obtained lower bounds on the number of edges in the $H$-free process for all strictly $2$-balanced graphs $H$. Corresponding upper bounds, up to constants, were given by Warnke~\cite{warnke2014cycle} in the case of the $C_{\ell}$-free process.

The triangle-removal process is another natural random process that produces a triangle-free graph. Here one starts with the complete graph and then at each step, one \emph{removes} a triangle uniformly at random, until the remaining graph is triangle free. This problem was introduced by Bollob\'{a}s and Erd\H{o}s \cite{Bollobs1998} who conjectured that the remaining graph has $\Theta(n^{3/2})$ edges. The best known result towards this conjecture is due to Bohman, Frieze and Lubetzky \cite{bohman2015379}, who proved the triangle-free graph that results from the triangle-removal process has $n^{3/2+o(1)}$ edges. This has recently been extended to a much more general class of graphs and hypergraphs by Joos and  K{\"u}hn~\cite{joos2024}.

There are two main contributions of the present paper. The first is to show that one can improve the lower bounds on $R(3,k)$ by running a triangle-free-like process from a carefully chosen ``seed graph''. Second, is
to give a much simpler analysis of a modified triangle-free process, all the way to the end of its trajectory. In fact our process is similar to the one in the original work of Kim \cite{kim1992ramsey} and the subsequent work of Warnke and Guo~\cite{guo2020packing}. However, unlike these papers, we are able to follow our process (and track the independent sets) all the way to the (asymptotic) end.

\subsection{The asymptotic value of \texorpdfstring{$R(3,k)$}{R(3,k)}}

It is interesting to speculate about the asymptotic value of $R(3,k)$.  We strongly believe that $R(3,k) \geq (1/2+o(1))k^2/\log k$ and tentatively conjecture that, in fact, we have equality.

\begin{conj}\label{conj:bold}
\[ R(3,k) = \bigg(\frac{1}{2}+o(1)\bigg)\frac{k^2}{\log k} .\]
\end{conj}

In a forthcoming companion paper we support this conjecture by proposing a very different construction,
that would conjecturally imply the lower bound in Conjecture~\ref{conj:bold}, along an infinite sequence of values of $k$. The construction is based on the Cayley sum graph generated by the sum-free process on $\mathbb{F}_2^d$. It also seems possible that one could provide a better ``seed step'' in our process from this paper to obtain a constant of $1/2+o(1)$, although we were unable to find such a seed.  

As mentioned above, the best known \emph{upper bound} on $R(3,k)$ was obtained by Shearer~\cite{shearer1983note} who proved the following more general result: any triangle-free graph on $n$ vertices with average degree $d$ has an independent set of size at least  \begin{align}\label{eq:shearer}
\big( 1+o (1)\big)\frac{n \log d }{d}, \end{align}
as $d\rightarrow \infty$. This is seen to be sharp, up to a factor of $2+o(1)$, by considering a random graph of average degree $d$, which is triangle-free with probability at least $c(d)>0$, when $d$ is constant. It is a notoriously difficult open problem to improve Shearer's bound by \emph{any} constant factor and it remains a complete mystery as to what the best possible constant in \eqref{eq:shearer} is.

The challenge of understanding this factor of $2$ also reflects itself in the study of algorithms. For example, it is a major open problem to determine if there is an efficient algorithm that can find an independent set of size $\geq (1+\eps)(n/d)\log d$ in the random graph on $n$ vertices of average degree $d$, despite the fact that we know there are \emph{many} such independent sets. In fact it has been shown \cite{gamarnik2014limits,rahman2017local} that no ``local algorithm'' can find such an independent set (see \cite{wein2022optimal} for an extension to a hardness result for a broader class of algorithms). 

Work inspired by phase transitions in statistical physics \cite{achlioptas2008algorithmic,coja2015independent} gives us another perspective on why \eqref{eq:shearer} should be hard to improve upon. In the random graph of constant average degree $d$, the ``geometry'' of the space of independent sets of size $\lambda n \log d/d$ becomes significantly more complex as $\lambda$ transitions from $1-\eps$ to $1+\eps$. In particular, when $\lambda \leq 1-\eps$ the set of independent sets of this size are ``connected'' meaning that one can move between any two independent sets $I$, $J$ via a path of independent sets $I = I_1,I_2,\ldots ,I_{t} = J$, where $I_i$ have the same size and $I_i$, $I_{i+1}$ have symmetric difference $O(1)$. On the other hand, when $\lambda = 1+\eps$ the space of independent sets ``shatters'' into exponentially many components. 

To see how the bound on $R(3,k)$ follows from \eqref{eq:shearer}, note that any triangle-free graph of average degree $d$ contains an independent set of size $d$, since in a triangle-free graph neighbourhoods are independent sets. Optimizing this bound against \eqref{eq:shearer}, we deduce that any triangle-free graph on $n$ vertices has an independent set of size at least $(1/\sqrt{2} + o(1))\sqrt{n \log n }$, which implies $R(3,k) \leq (1 + o(1)) k^2 / \log k$.

The authors of \cite{FGM} speculated that the gap of $4$ between the upper and lower bounds on $R(3,k)$ in \eqref{eq:R3k-bounds} has two separate ``sources'', each of which accounts for a factor of $2$. The first comes from the gap in Shearer's bound \eqref{eq:shearer}, which can conjecturally be improved by a factor of $2$. The second, they speculate, is that for optimal constructions for $R(3,k)$, the independence number is $2 + o(1)$ times larger than the maximum degree.  
We highlight that the independence number of our construction is only $3/2+o(1)$ times larger than the maximum degree. Moreover, we believe that in an optimal construction (witnessing the bound of Conjecture~\ref{conj:bold}) the independence number will be asymptotically equal to the maximum degree. Thus Conjecture~\ref{conj:bold} reflects our belief that the only factor of two to be gained in the upper bound comes from the ``missing'' factor of two in Shearer's bound. 

\section{Heuristic discussion of the construction}

At a heuristic level, there are two properties we desire for the triangle-free graph that we seek. One is density: having many edges in our graph makes it easier to force edges into potential independent sets. Of course, this is not enough; the densest graph with no triangle, the complete bipartite graph, contains very large independent sets. The other natural property we want our graph to have is quasi-randomness. We want the edges that we do have to be well distributed throughout the graph. 

Thus, one is naturally led to study the triangle-free process, which is (in a sense) the most random-like graph containing no triangle, and study its behavior all the way to the end, when it is the densest. As we 
highlighted above, this has been a very successful program which has resulted in beautiful mathematics and the previous best
lower bound for $R(3,k)$. 

The key observation behind this paper is that the triangle-free process is actually ``too random'' for an optimal lower bound on $R(3,k)$. Instead, it is better to trade off some of the randomness of the triangle-free process to get a denser graph, which is still random enough to ensure that the independent sets behave like those in a random graph $G(n,p)$.

\subsection{A heuristic justification for the previously best known bound}\label{subsec:heuristic}

Before explaining our improvement to $R(3,k)$ we lay out a heuristic way of thinking about the previously best known bound $(1/4 +o(1))k^2/\log k$ that arises from the triangle-free process.

In what follows, we make one large (and non-rigorous) assumption: that the triangle-free process, at density $p$, ``behaves as $G(n,p)$, apart from the fact that it has no triangles''. Let $G_p$ be the triangle-free process when it has density $p$ and let $\PP_p$ denote the corresponding probability measure.
We first ask: how large could we reasonably expect $p$ to be? For this, let us say a pair $e\in K_n\setminus G_p$ is\footnote{Here and throughout the paper we identify a graph with its edge set. So a graph is a set of unordered pairs.} \emph{open} if we can add it to $G_p$ without forming a triangle. Now note that if 
\begin{equation}\label{eq:expected-number-of-open} \EE_p\, \big| \big\{  e \in K_n : e \text{ is open} \big\} \big| \ll p\binom{n}{2},  \end{equation} then the number of edges that one could possibly add after this point (since we can only add open edges) is dwarfed by the number of edges that have already been added, and therefore is negligible.

To calculate the probability a pair $e = xy$ is open (and thus the expectation in \eqref{eq:expected-number-of-open}), we note that if $e$ is open then for all $v \in [n]\setminus \{x,y\}$, either $xv$ or $yv$ is not an edge of $G_p$. Using this and our hypothesis that $G_p$ behaves like $G(n,p)$, we have
\begin{equation} \label{eq:tfp-density-heuristic} \PP_p\big(  xy \text{ is open} \big) \approx \prod_{v } \PP_p\big( \{ xv , yv \} \not\subset  G_p \big) = \big( 1 - p^2 \big)^{n-2} = e^{-(1+o(1))p^2 n} . \end{equation}
 Therefore, noting that  
\[ e^{-p^2n} = p \qquad \Longrightarrow \qquad p = (1+o(1))\sqrt{\frac{\log n}{2n}}, \]
  this gives the correct heuristic for the maximum density of the triangle-free process (cf. \eqref{eq:density-of-tfp}).
 
 So let us fix $p$ to be this maximum density and turn to derive a heuristic for the largest independent set in the triangle-free process. 
 Since we are assuming that $G_p$ behaves as $G(n,p)$,
 it is natural to guess that $\alpha(G_p) = (1+o(1))\alpha(G(n,p)) = (1+o(1))(\log n)/p$. 
 While this heuristic is in fact correct, we are brushing up against a spot where the heuristic starts to falls apart: in the triangle-free graph $G_p$ neighbourhoods are also independent sets. So a better heuristic might be
 \begin{equation}\label{eq:heur-alpha-tfp} \alpha(G_p) = (1+o(1))\max\big\{ pn , (\log n)/p \} = (\sqrt{2} + o(1))\sqrt{n\log n} \, . \end{equation}
While the maximum degree condition is actually superfluous for this density $p$, we make note of it because it will be important later,  when we move to our new process (see \eqref{eq:k-max-deg}).  

 Thus, setting this latter term to $k$ and solving for $n$ gives a heuristic for the bound $R(3,k) \geq (1/4+o(1))k^2/\log k$. Of course, \emph{proving} that this heuristic conforms to reality is an extremely difficult task and is the content of~\cite{FGM, bohman2021dynamic}. 
 
 We now turn to heuristically justify why we can do better than the triangle-free process.

\subsection{Heuristic justification for our construction} \label{sec:heuristic}

In our construction, we first sample a random graph on $n/r$ vertices, at density $p_0 = \alpha_0 (\log n/n)^{1/2}$, pass to a maximal triangle-free subgraph and then blow it up\footnote{That is, replace each vertex $i$ of the graph with a set $V_i$ of size $r$, and replace each edge $ij$ with a complete bipartite graph between $V_i, V_j$.} by a factor of $r$, where $r=(\log n)^2$ and $\alpha_0>0$ is a constant (to be specified later). We shall refer to this as the \emph{seed step} in the process. We then run a variant of the triangle-free process within the open edges of the resulting graph. 

The key idea behind this seed step is that the seed contributes a positive proportion of the final density of our graph while leaving almost all pairs open. (For the optimal case, the seed will account for half of the density of the final graph.) That is, after the initial seed step, almost any edge can be added to the seed without forming a triangle. To see this, note that if $e$ is a pair between parts of the blow up, we have
\[ \PP\big( e \text{ open after seed step} \big) \approx (1-p_0^2)^{n/r} = 1-o(1),\] since in our case, we have $p_0 = \alpha_0\sqrt{\log n /n}$ and $r = (\log n)^2$. It is worth comparing this to the probability that $e$ is open after we run the triangle-free process to the same density, which is $(1-p_0^2)^n = n^{-\alpha_0^2 +o(1)}$ (see \eqref{eq:tfp-density-heuristic}). 

We now let $p_1 = \alpha_1 (\log n/n)^{1/2}$ be the density of the graph produced in all steps following the seed step, so that $p_0 + p_1$ is the density of the final graph. We call this graph $G_{p_0,p_1}$ and call $\PP_{p_0,p_1}$ the corresponding probability measure.

To guess the independence number of $G_{p_0,p_1}$, we first identify a good ``model'' in which to argue heuristically about $G_{p_0,p_1}$. In particular, we define the random graph  
\[ G = G_0 \cup G_1, \] where $G_0$ is an $r$-blow up of $G(n/r,p_0)$ on vertex set $[n]$ and $G_1$ is an independent copy of $G(n,p_1)$ on the same vertex set. In analogy with our reasoning for the triangle-free process, we would like to think of the random graph $G_{p_0,p_1}$ as akin to the random graph $G_0 \cup G_1$, ``apart from the fact that the graph $G_{p_0,p_1}$ has no triangles''.

As in our discussion of the triangle-free process, we analyse the permissible densities $(p_0,p_1)$ by understanding the probability a pair $e$ becomes closed in $G_{p_0,p_1}$. 
Indeed if $(p_0,p_1)$ is such that $\PP_{p_0,p_1}\big(  e \text{ is open} \big) \ll p_0+p_1 $ then we have essentially reached a maximal pair of densities, since any further edges we could hope to add would have to come from open pairs. Therefore, the maximal pairs $(p_0,p_1)$ for which we could reasonably expect  $G_{p_0,p_1}$ to exist are those for which 
\[ \PP_{p_0,p_1}\big(  e \text{ is open} \big) \approx p_0+p_1.\] 

To calculate this probability, let $e$ be a pair with end-points in different parts of the partition defined by the blow up. (We ignore pairs contained within 
the parts, as there are extremely few such pairs). There are three ways such a pair can be closed. It can be closed by two edges from $G_0$, which has probability $1-(1-p_0^2)^{n/r}$; it can be closed by two edges in $G_1$, which has probability $1-(1-p_1^2)^n$; and it can become closed by one edge from $G_0$ and one edge from $G_1$, which has probability approximately $1-(1-2p_0p_1)^n$. Thus, in analogy with \eqref{eq:tfp-density-heuristic}, we have
\begin{equation*}
\PP_{p_0,p_1}\big(  e \text{ is open} \big) \approx (1-p_0^2)^{n/r}(1-2p_0p_1)^{n}(1-p_1^2)^{n} = \exp\hspace{-0.6mm}\big(\hspace{-0.5mm}-\hspace{-0.5mm}(1+o(1)) ( 2p_0p_1 + p_1^2 ) n \big).  \end{equation*}
 Recalling that $p_i = \alpha_i (\log n/n)^{1/2}$, for $i=0,1$, it follows that 
\begin{equation}\label{eq:constraint} \PP_{p_0,p_1}\big(  e \text{ is open} \big) = p_0+p_1 = n^{-1/2+o(1)}  \quad \Rightarrow \quad  (\alpha_0+\alpha_1)^2 -\alpha_0^2 = 1/2+o(1). \end{equation}
We thus obtain our first important constraint on the densities $(p_0,p_1)$. 

We now turn to understand what is most important for us, the independence number of this random graph. We shall argue heuristically that we should have 
\begin{equation}\label{eq:alpha-sqrt3/2} \alpha(G_{p_0,p_1}) \le \big( \sqrt{3/2} + o(1) \big)\sqrt{n\log n}\, ,\end{equation} for the permissible pair $(p_0,p_1)$ defined by 
\begin{equation}\label{eq:heuristic-choice-of-alpha_i} \alpha_0 = \alpha_1 = \frac{1}{\sqrt{6}} + o(1)  .\end{equation}
For this, it will be useful to fix the notation
\[ k = \kappa \sqrt{n\log n} \qquad \text{ where } \qquad \kappa = \sqrt{3/2}+o(1)\, . \]
To bound the independence number, we
let $I_k$ be the number of independent $k$-sets in $G$. Our goal is to show $\EE\, I_k = o(1)$ as $n\rightarrow \infty$. 

The subtlety in our construction comes from the fact that $I$ can interact with the partition in different ways. Indeed, if $I$ intersects many different parts of the partition then $G$ ``looks like'' a random graph of significantly larger density, from the point of view of $I$. On the other hand there are significantly fewer sets $I$ that intersect few parts and thus we have fewer sets to sum over in the expectation. Here we will show that with our choice of parameters, the sets that intersect $k$ different parts dominate the expectation $\EE\, I_k $.

To elaborate on this, we define $I_{k,\ell}$ to be the number of independent $k$-sets in $G$ that intersect $\ell$ different parts of the partition. We may bound 
\begin{equation}\label{eq:Ikell} \EE\, I_{k,\ell} \leq n^{o(k)} \binom{n}{\ell}\big(1-p_0\big)^{\binom{\ell}{2}}\big(1-p_1\big)^{\binom{k}{2}} , \end{equation} since there are $\binom{n/r}{\ell}$ ways of choosing the $\ell$ parts of the partition that $I$ intersects and at most $\binom{r\ell}{k} \leq (er)^{k} =n^{o(k)}$ ways of choosing $I$ in those parts.

To determine the maximum of $\EE\, I_{k,\ell}$ over $\ell$, we note the (approximate) recurrence
\begin{equation}\label{eq:Ikell-recurrance}  \EE\, I_{k,\ell+1} \approx n^{1/2}(1-p_0)^{\ell}  \cdot \EE\, I_{k,\ell} \end{equation}
when $\ell$ is order $k$, and observe that this sequence is increasing in $\ell$, while $p_0\ell \leq (1/2)\log n$. Thus if 
\begin{equation}\label{eq:ell_is_k_is_max_cond} \alpha_0 \kappa  \leq 1/2,  \end{equation} 
we ensure that $\EE\, I_{k,\ell}$ is maximized at $\ell = k$. Note that this holds, with equality, for our choice of $\alpha_0$ and $\kappa$.

This means that it is enough to show $\EE\, I_{k,k} = o(1)$. For this, we observe that when $\ell = k$ the right hand side of \eqref{eq:Ikell} is actually a familiar quantity (up to the unimportant $n^{o(k)}$ factor): it is the expected number of independent $k$-sets in the binomial random graph $G(n,p_0+p_1)$. Thus a standard calculation tells us that (see, e.g., \cite{bollobas2001random})
\[  k \geq (1+o(1))\frac{\log n}{p_0+p_1}    \quad \Longrightarrow  \quad \EE\, I_{k,k} = o(1) , \] where we used that $p_0+p_1 = n^{-1/2+o(1)}$. Thus, we are led to guess that if $\alpha_0$ and $\alpha_1$ satisfy
\begin{equation}\label{eq:heur-ell=k} (\alpha_0+\alpha_1) \kappa \geq 1+o(1),\end{equation} 
and \eqref{eq:ell_is_k_is_max_cond} then  $\alpha(G_{p_1,p_2}) \leq \kappa \sqrt{n\log n}$. Note that \eqref{eq:heur-ell=k} holds for our choice of $\alpha_0,\alpha_1$ and $\kappa$. 

As we saw in our heuristic for the triangle-free process, we also need to take into account the fact that neighbourhoods are independent sets in $G_{p_0,p_1}$, which gives us the further constraint  \begin{equation}\label{eq:k-max-deg} \kappa \geq (1+o(1))(\alpha_0+\alpha_1), \end{equation}
which also easily holds for our choice of $\kappa,\alpha_0,\alpha_1$. This gives some justification for our claim at \eqref{eq:alpha-sqrt3/2} and solving for $n$ gives 
\[ R(3,k) \geq \bigg( \frac{1}{3}+o(1)\bigg)\frac{k^2}{\log k} .\]
In fact, given the above constraints, one can see that this is the best possible; use \eqref{eq:heur-ell=k} $(\alpha_0+\alpha_1)^{-1} \leq \kappa$ and \eqref{eq:ell_is_k_is_max_cond} $\alpha_0 \leq 1/(2\kappa)$ in our permissibility constraint \eqref{eq:constraint} to see
\[ 1/2  \geq (\alpha_0+\alpha_1)^2- \alpha_0^2 \geq \kappa^{-2}(3/4),\]
which implies $\kappa \geq \sqrt{3/2}$. 

 However, if we reflect on our discussion above, we notice the constraint \eqref{eq:ell_is_k_is_max_cond} is actually not an essential one, designed rather to make our choices more intuitive. Thus one could easily drop this condition and try to optimize without it. If one follows this line, one will be led to obtain an optimal value of $\kappa \approx 1.07$, where $1.07$ is notably less than $\sqrt{3/2}$. 

 Interestingly, this does not represent the truth and reveals a hole in our heuristic as detailed above. The reason for this is that there is another type of independent set that is relevant for us. These are the $k$-sets $I$ that contain a neighbourhood of the seed step:
\[ I \subset [n], \quad |I| = k  \qquad \text{ such that } \qquad I \supset N_{G_0}(x). \]
Intuitively, these sets are dangerous because the edges within the neighbourhood $N_{G_0}(x)$ are excluded from the second round of the process, by the triangle-free condition. Also adding to their danger is the fact that the edges in the seed graph are much ``cheaper'' than edges added in the triangle-free process, since we get them in bunches from the blow up. 

To check that independent sets of this form don't spoil our choice of $\kappa,\alpha_0,\alpha_1$, we turn to understand the expected number of independent sets of this special form. We notice 
\begin{equation}\label{eq:funky-Is} |N_{G_0}(x)| = (1+o(1)) p_0n\end{equation}
and that $I$ intersects at most $k - p_0n + o(k)$ different parts of the underlying partition. So if we let $J_{k,\ell}$ be the number of independent sets of this form that intersect $\ell$ different parts, we have 
\begin{equation}\label{eq:sketch-Jkell} \EE\, J_{k,\ell} \leq n^{o(k)}\binom{n}{\ell}\big(1-p_0\big)^{\binom{\ell}{2}}\big(1-p_1\big)^{\binom{k}{2} - \binom{p_0n}{2}},  \end{equation}
where the exponent $\binom{k}{2} - \binom{p_0n}{2}$ reflects the fact that no edges inside of $N_{G_0}(x)$ can be included in $G_1$ without forming a triangle. 

We now argue that these special sets don't dominate the count of independent sets. In other words, we would like to show 
$\EE\, J_{k,\ell} \ll \EE\, I_{k,k}$ for all $\ell \leq k-p_0n$. For this we observe, by comparing \eqref{eq:Ikell} with \eqref{eq:sketch-Jkell}, that
\[ \EE\, J_{k,\ell} \leq  (1-p_1)^{-\binom{p_0n}{2}} \EE\, I_{k,\ell}. \]
We then iteratively apply the recurrence \eqref{eq:Ikell-recurrance}, to see that 
\begin{equation}\label{eq:EJkellvsEIkk} \EE\, J_{k,\ell} \leq n^{-(k-\ell)/2}(1-p_0)^{\binom{\ell}{2}- \binom{k}{2}} (1-p_1)^{-\binom{p_0n}{2}} \EE\, I_{k,k}.  \end{equation}
We now claim that the right hand side of \eqref{eq:EJkellvsEIkk} is $\ll \EE\, I_{k,k}$ for all $\ell \leq k-p_0n$ when 
\begin{equation}\label{eq:Jellk-constraint} 2\alpha_0\kappa - \alpha_0^2 + \alpha_0\alpha_1\leq 1 + o(1),\end{equation}
thus giving us our final constraint. Note also that this equation is satisfied and saturated for our selection at \eqref{eq:heuristic-choice-of-alpha_i}. To see why \eqref{eq:Jellk-constraint} implies that the right hand side of \eqref{eq:EJkellvsEIkk} is $\ll \EE I_{k,k}$, we use that 
 $k-\ell \geq  p_0n$, from \eqref{eq:funky-Is}, which implies $k+\ell \leq 2k - p_0n $ and therefore
\[ p_0\binom{k}{2}-p_0\binom{\ell}{2} = p_0(k+\ell)(k-\ell)/2 + O(\log n) \leq (1+o(1))(2p_0k-p_0^2n)(k-\ell)/2 . \]
Also, using that $p_0n \leq k-\ell$, we have
\[ p_1 \binom{p_0n}{2} \leq p_1p_0n(k-\ell)/2 = \alpha_0\alpha_1(\log n)(k-\ell)/2.\]
Then, factoring out a $(\log n)(k-\ell)/2$ from the exponent of each of the first three terms on the right hand side of \eqref{eq:EJkellvsEIkk}, we have
\[ \exp\big(\big(-1 + 2\alpha_0\kappa - \alpha_0^2 + \alpha_0\alpha_1 + o(1) \big)(\log n)(k-\ell)/2 \big) \cdot \EE\, I_{k,k},\] which is $\ll \EE I_{k,k} $  when \eqref{eq:Jellk-constraint} is satisfied.

What we have seen above is, of course, only a heuristic discussion and making this rigorous is a challenging task and the main technical content of this paper.

\section{The triangle-free nibble}\label{sec:definition-of-process}
In this section we formally define our modified triangle-free nibble and define many of the notations that we use throughout the paper. Our process takes as input a density $p_0 \in [0,1]$ which we express as
\begin{equation}\label{eq:def-p00} p_0 = \alpha_0 \bigg(\frac{\log n}{n}\bigg)^{1/2},  \end{equation}
for some constant $\alpha_0\geq 0$. As we saw in Section~\ref{sec:heuristic}, this corresponds to the initial step of the process.

\subsection{Definition of the process} 
Given $n$, we define (somewhat arbitrarily) the parameter 
\begin{equation}\label{eq:def-gamma} \g = (\log n)^{-10}\, , \end{equation}
and describe a random process which defines a random triangle-free graph $G$. At stage $i$ of the process we will have produced graphs $G_0,\ldots, G_i$ so that 
\[ G_{\leq i}  = G_0 \cup G_1 \cup \cdots \cup G_{i}, \]
is triangle free. We think of $G_{\leq i}$ as our desired triangle-free graph ``so far.'' We will also maintain a graph of \emph{open pairs} at stage $i$, which we denote $O_i$. Crucially $O_i$ satisfies
\[ O_i \subset \big\{ e \in K_n \setminus G_{\leq i } : G_{\leq i } + e \not\supset K_3 \big\}. \] We also define auxillary graphs $G_0',\ldots,G'_i$ on $[n]$ for which $G_i \subset G_i'$ and set  
\begin{equation}\label{eq:def-G'} G'_{\leq i} =G'_0\cup G_1' \cup \cdots \cup G'_i.
\end{equation} We will say that a pair $e \in K_n$ is \emph{closed} if $e \not\in O_i$. In fact, we maintain the stronger property 
\begin{equation} \label{eq:O_i-first} O_i \subset \big\{ e \in K_n \setminus G'_{\leq i } : e \text{ does not form a triangle with } G_{\leq i}' \big\}. \end{equation}
Our final graph $G$ is defined as $G = G_{\leq \tau}$, where the process stops at time $\tau$.  In practice, we will only run the process for a specified number of steps $T$, which will be $< \tau$ with high probability (see \eqref{eq:T_0}). 

As mentioned above, the zeroth step of this process is peculiar and will require special treatment throughout the paper. 

\subsection{Construction of the seed}\label{subsec:seed}
We now describe the zeroth step of the process which produces graphs $(G_0,G'_0,O_0)$. Set $r = (\log n)^2$, sample a random graph
\begin{equation}\label{eq:def-G_{ast}'} G_{\ast}' \sim G(n/r,p_0) , \end{equation} and then define ${G}_\ast \subset G_{\ast}'$ to be a maximal triangle-free subgraph. Here $p_0$ is the input parameter defined at \eqref{eq:def-p00} and we identify the vertex set of $G_{\ast}'$, $G$ with $[n/r]$. We define $G'_0$ and $G_0$ as $r$-blow-ups of ${G}_\ast'$ and ${G}_\ast$ respectively. 
Now define 
\begin{equation}\label{eq:def-O'_0} O'_0 = \big\{ e \in K_n \setminus G'_0 :  e \text{ does not form a triangle with } G_{0}' \big\} \end{equation} and let $[n] = V_1 \cup \cdots \cup V_{n/r}$ be the partition corresponding to the blown up vertices. Note that $O'_0$ is also a blow-up with respect to the same partition. 

We now define $O_0 \subset O'_0$ by choosing an edge $e_{i,j} \in O_0'[V_i,V_j]$ uniformly at random for each pair $i<j$ such that $O_0'[V_i,V_j]$ is non-empty\footnote{If $G$ is a graph and $U,V \subset V(G)$ we define $G[U] = \{ xy \in G : x,y \in U \}$ to be the graph \emph{induced} on $U$ and define $G[U,V] = \{ xy \in G : x \in U, y \in V\}$.}. We then define 
\[  O_0 = \big\{ e_{i,j} \,: \, O_0'[V_i,V_j] \not= \emptyset \big\}.\]
The reason for this step is somewhat technical; it will allow us to control codegrees in the graph $O_i$ more easily (see \eqref{eq:intersections}). 

\subsection{A step in the process}\label{sec:a-step-in-process} We now define the process for steps $i \geq 1$. Given graphs $G_0',\ldots, G_i'$ , $G_0, \ldots , G_i$ and $O_0, \ldots ,O_i$, we define $G'_{i+1}, G_{i+1},O_{i+1}$ as follows.
\begin{enumerate}[label=\arabic*., ref=\arabic*] 
 \item \label{step:sprinkle} (Nibble step) We first sample $G_{i+1}'$ randomly from the open pairs $O_i$. Write 
\begin{equation}\label{eq:def-theta_i} e(O_i) = \theta_i \binom{n}{2} \hspace{3em} \text{ define } \hspace{3em}  p_{i+1} =   \frac{\g}{\theta_i \sqrt{n}}, \end{equation} and then sample \[  G'_{i+1} \sim G(O_i,p_{i+1}),\] if $p_{i+1} \leq \g^3 $. Otherwise, stop and set $\tau=i$. 

\vspace{2mm}

\item \label{step:maximial-triangle-free} (Cleaning step) Let $G_{i+1}\subset G_{i+1}'$ be maximal such that  $G_{\leq i} \cup G_{i+1}$ is triangle free. 

\vspace{2mm}

\item (Regularization step) \label{step:reg-step} To define $O_{i+1}$, we first define, 
\begin{equation}\label{eq:def-Oi+1} O_{i+1}' = \big\{ e \in  (K_n  \setminus G_{\leq i+1}') \cap O_{i}: e \text{ does not form a triangle with } G_{\leq i+1}' \big\}.\end{equation}
Then for each $e \in O_i$, we define the quantity
\begin{align}\label{eq:qedef}
    q_e  =  \frac{\min_{f\in O_i}\PP\big( f \in O_{i+1}' \big)}{ \PP\big( e \in O_{i+1}' \big) } \in [0,1]\, ,
\end{align}  
where the above probabilities are with respect to the randomness of $G'_{i+1} \sim G(O_i,p_{i+1})$.
We now define $Q_{i+1} \subset O_{i}$ to be a random subset where each pair $e 
\in O_i$ is independently included with probability $q_e$. Finally we define 
\begin{equation}\label{eq:reg-step} O_{i+1} = Q_{i+1} \cap O_{i+1}'. \end{equation} This completes a step in the process. 
\end{enumerate}

The main feature of our regularization step is that the ``survival probability'' $\P(e \in O_{i+1} \,|\,e \in O_i)$ is the same for all $e$ (see \eqref{eq:prob-e-open-constant}). While this ``dampening'' seems to be a small detail, it actually has a large impact and greatly simplifies our analysis.

\subsection{Selection of parameters and main technical result}

Let $ 0< \delta < 1/8$ be fixed and arbitrarily small throughout the paper and let $\alpha_0,\alpha_1 \geq 0$ be such that 
\begin{equation}\label{eq:alpha_0alpha_1-setup} (\alpha_0 +\alpha_1)^2-\alpha_0^2 \leq 1/2 - 2\delta . \end{equation}
Then define the quantities 
\begin{equation}\label{eq:def-p0} p_0 = \alpha_0 \bigg(\frac{\log n}{n}\bigg)^{1/2} \qquad \text{ and } \qquad  p_1 = \alpha_1 \bigg(\frac{\log n}{n}\bigg)^{1/2}. \end{equation}
We will show that, with high probability, we can run the above process until time 
\begin{equation}\label{eq:T_0} T = \alpha_1 \gamma^{-1} \sqrt{\log n}, \end{equation} to obtain the graph $G_{\leq T}$ of density $p_0+p_1$. Most importantly, we show that if we set
\begin{equation}\label{eq:alpha_i-k-selection} \alpha_0 = \alpha_1 = (1-3\delta)/\sqrt{6} \qquad \text{ and } \qquad  k = (1+\delta)\sqrt{(3n/2) \log n }\end{equation} then $\alpha(G_{\leq T})< k$, with high probability. This is our main technical theorem of the paper and quickly implies Theorem~\ref{thm:r3k}.

\begin{theorem}\label{thm:main-theorem-ind-set} 
We have
\begin{equation}\label{eq:ind-set-bound}  \PP\big( \alpha( G_{\leq T} ) <  k \big) = 1- o(1).\end{equation}
\end{theorem}

In what follows we develop some of the basic terminology for working with the process and  sketch a heuristic understanding of how the process runs to time $T$.

\subsection{Central definitions and heuristic discussion of the evolution of the process}\label{sec:evo-of-process-sketch} 
For $i\geq 0$, we let $\P_{i+1},\E_{i+1}$ denote probability and expectation taken with respect to the randomness at step $i+1$ of our process; the selection of $Q_{i+1}, G'_{i+1}\subset O_{i}$, given $G_{\leq i}$ and $O_{i}$.

Throughout, we let $\theta_i$ be the random variable defined by
\begin{equation}\label{eq:def-theta_i-psi_i} e(O_i) = \theta_i \binom{n}{2}\, ,
\end{equation}
that is, $\theta_i$ is the density of open pairs at step $i$. It will be useful to keep the following heuristic in mind: $O_i$ and $G_{\leq i}'$ resemble two independent samples of a binomial random graph at the appropriate densities. 

Before discussing the evolution of $\theta_i$, we note that the density of $G_{\leq i}'$ is simple to understand. Indeed, recall that $G_0'$ is a blowup of the random graph $G(n/r,p_0)$ and that for $i\geq 0$, we defined $G'_{i+1}\sim G(O_i, p_{i+1})$, where $p_{i+1}\theta_i=\gamma/\sqrt{n}$. Thus, by standard concentration estimates we will maintain that  
\[ e(G'_{i}) =  (1+o(1))(\gamma/\sqrt{n}) \binom{n}{2}\] for each $i\geq 1$ and that the density of $G'_{0}$ is $(1+o(1))p_0$. Thus, for $i\geq 0$,
\begin{align}\label{eq:Psi-Approx}
 e(G'_{\leq i}) = (1+o(1))(i\g/\sqrt{n} + p_0 )\binom{n}{2}\, .
 \end{align}
Assuming that we can run the process to step $T = \alpha_1\g^{-1}\sqrt{\log n}$, the density of the graph $G_{\leq T}'$ is then $(1+o(1))(p_0 + p_1)$
which is in line with the heuristics in Section~\ref{sec:heuristic}.

We now turn to understand the trajectory of $\theta_i$. Here we will maintain that 
\begin{equation} \label{eq:sketch-theta_i} \theta_i \geq (1-o(1)) \exp\hspace{-0.5mm}\big(\hspace{-0.5mm} - \hspace{-0.5mm}\g^2 i^2-2\gamma i\alpha_0\sqrt{\log n}\big) \cdot \theta_0, \end{equation}
with high probability, where $\theta_0$ comes from the seed step and satisfies $\theta_0 = (1-o(1))r^{-2} = n^{-o(1)}$. (In fact, we expect \eqref{eq:sketch-theta_i} to hold with equality, however we only need the one-sided bound and so limit ourselves to proving this.)

It is a little harder to see why \eqref{eq:sketch-theta_i} holds, and for this we first make note of an important feature of our random process which we mentioned above (and the reason for the regularization step): the probability that a pair $e \in O_i$ remains open in the next round (that is, $e\in O_{i+1}$) is uniform over all $e \in O_{i}$. While this seems like a small detail, it dramatically simplifies the process and allows for a much easier analysis. Indeed, for $e \in O_i$, we have that 
\begin{equation}\label{eq:prob-e-open-constant} \PP_{i+1}\big( e \in O_{i+1} \big) = \PP_{Q_{i+1}}\big( e\in Q_{i+1}\big)\PP_{G_{i+1}'}\big( e \in O_{i+1}' \big) = \min_{f \in O_i }\, \PP_{i+1} \big( f \in O_{i+1}' \big).\end{equation}
Therefore it makes sense to define the ``survival probability'',
\begin{equation}\label{eq:def-si+1} s_{i+1} = \PP_{i+1}(e \in O_{i+1}) \qquad \text{and thus } \qquad  \EE_{i+1} \theta_{i+1} = s_{i+1} \theta_i,\end{equation} where $e$ is any $e \in O_i$. Thus we have a recurrence for the expectation of $\theta_{i+1}$ in terms of $s_{i+1}$ and so we turn to find an expression for this probability. Here we need two definitions, corresponding to the two ways that an edge could close a triangle at step $i$. 

We say that edges $e,f,g$ form an \emph{open} triangle at step $i$ if $e,f,g$ form a triangle and $e,f,g \in O_i$. For $e \in O_i$
we define the graph on open pairs
\begin{equation}\label{eq:def-Y} \cX_i(e) = \big\{ \{f,g\} \in O_i^{(2)} : e, f, g \text{ form an open triangle in } O_i  \big\}. \end{equation}
Similarly, we say that $e,f,g$ form a \emph{clopen triangle} if $e,f \in O_i$ and $g \in G'_{\leq i}$ (or in some other permutation). For each $e \in O_i$ we define 
\begin{equation}\label{eq:def-X} \cY_i(e) = \big\{ f \in O_i : e, f \text{ span a clopen triangle} \big\} \cup \{ e \},  \end{equation}
where it makes sense to add $e$ to this set since an edge $e$ is trivially closed if we add it to $G_{\leq i+1}'$ (see the definition at \eqref{eq:def-Oi+1}). It will be useful throughout the paper to define the quantities
\begin{equation}  X_i(e) = |\cX_i(e)| \qquad \text{ and } \qquad Y_i(e) = |\cY_i(e)| .  \end{equation}
Thus, using \eqref{eq:prob-e-open-constant}, we may express $s_{i+1}$ as
\begin{equation}\label{eq:si+1-formula} s_{i+1} = \min_{ e \in O_i }\, \PP_{i+1} \big( e \in O_{i+1}' \big) = \min_{e \in O_i}\, (1-p_{i+1}^2)^{X_i(e)} (1-p_{i+1})^{Y_i(e)}. \end{equation}
With \eqref{eq:si+1-formula} in view, we are naturally led to study $X_i(e)$ and $Y_i(e)$. For this we define 
\[ N_{\leq i}'(x) = N_{G_{\leq i}'}(x) \qquad \text{ and } \qquad N^{\circ}_{i}(x) = N_{O_i}(x).  \]
We now recall our heuristic above: we expect that the two graphs $G_{\leq i}'$ and $O_i'$ should look somewhat like independent copies of the binomial random graph of the appropriate densities. With this in mind, we maintain that throughout the process we have
\begin{equation}\label{eq:intersections}  |N^{\circ}_i(x) \cap N^{\circ}_i(y)| \leq (1+o(1)) \theta_i^2 n \qquad \text{ and } \qquad |N'_{\leq i}(x) \cap N^{\circ}_{i}(y)| \leq (1+o(1)) \theta_i\psi_i n, \end{equation}
for all pairs of distinct vertices $x,y\in [n]$ where $\psi_i=i\g/\sqrt{n}+p_0$. Again, we expect that these inequalities are equalities, but we only need the upper bounds. These bounds at \eqref{eq:intersections} are then easily seen to imply 
\begin{equation}\label{eq:disc-XiYi-bounds} X_i(e)\leq (1+o(1)) \theta_i^2 n \qquad \text{and} \qquad Y_i(e)\leq (1+o(1)) 2\theta_{i}\psi_i n   \, .\end{equation}
Which, using \eqref{eq:si+1-formula}, \eqref{eq:disc-XiYi-bounds} and recalling that $p_{i+1} = \g/(\theta_i\sqrt{n})$, implies
\[s_{i+1}  \geq (1-o(1))\exp\hspace{-0.5mm}\big(\hspace{-0.5mm}-\hspace{-0.5mm}(2i+1)\g^2-2\gamma\alpha_0\sqrt{\log n}\big). \]
Thus inductively applying $\EE_{i+1} \theta_{i+1}  = s_{i+1}\theta_i$ gives \eqref{eq:sketch-theta_i}, on average. We shall then need to complement these average results with concentration results to show that they also hold with high probability.  

We now turn to sketch the proof that our process runs until time $T = \alpha_1\g^{-1} \sqrt{\log n}$. In fact, this is fairly easy given the above. From 
$\eqref{eq:sketch-theta_i}$ we see that at time $T$, we have 
\[ \theta_T \geq (1-o(1)) \exp\hspace{-0.5mm}\big(\hspace{-0.5mm} - \hspace{-0.5mm}\g^2 T^2-2\gamma T\alpha_0\sqrt{\log n}\big)\cdot \theta_0 = n^{-\alpha_1^2 -2\alpha_0\alpha_1 + o(1) } \geq n^{-1/2+\delta }, \] since $\theta_0 = n^{-o(1)}$ and since we chose  $\alpha_0,\alpha_1$ to satisfy $\alpha_1^2+2\alpha_1\alpha_0 \leq 1/2 - 2\delta$ at \eqref{eq:alpha_0alpha_1-setup}. Since the process only halts when $\theta_i \leq \g^{-3}/\sqrt{n}$ (see Section~\ref{sec:a-step-in-process}, Step~\ref{step:sprinkle}) we can conclude that the process runs to step $T$. 

The reader may find it odd that in the above discussion we track the density of $G'_{\leq i}$ rather than the density of $G_{\leq i}$, since the latter is the part of the graph that we ultimately output. However, it is a (perhaps surprising) feature of our proof that we can almost entirely focus on the properties of the graph $G'_{\leq i}$. The only place where it is crucial we work with $G_{\leq i}$ rather than $G'_{\leq i}$ comes in the independent set calculation (see Lemma~\ref{lem:prob-ind}) since an independent set in $G_{\leq i}$ may fail to be independent in $G'_{\leq i}$. Other than this, one should think of $G_{\leq i}$ and $G'_{\leq i}$ as very similar: we just delete a small proportion of edges of $G'_{\leq i}$ to obtain $G_{\leq i}$ to preserve triangle-freeness.

\subsection{Tracking degrees and codegrees} \label{ss:tracking-degs-and-codegs}

As we saw above, controlling our process boils down to controlling the maximum degree and codegree of $G_{\leq i}'$ and $O_i$ in addition to the maximum ``cross-degree'' \eqref{eq:disc-XiYi-bounds}. This is in stark contrast to \cite{bohman2021dynamic} and \cite{FGM}, where the authors are forced to track a large ensemble of (interacting) random variables. 
Perhaps surprisingly, we control these degrees and codegrees by a fairly simple moment argument.

The key input into this moment argument is the following statement on the probability a small subgraph ``survives'' to the next round of the process. Assume we are at step $i$ of the process and let $K$, $H \subset O_i$ be ``small'' fixed subgraphs (of size  $n^{o(1)}$ say). Then
\begin{equation}\label{eq:sketch-small-subgraphs} \PP_{i+1}\big( H \subset O_{i+1} \text{ and } K \subset G_{i+1}' \big) \leq (1+n^{-c})^{e(H)}s_{i+1}^{e(H)}p_{i+1}^{e(K)},\end{equation} where $c>0$ is an absolute constant. Note that this is consistent with our heuristic that $G_{\leq i+1}'$ and $O_{i+1}$ behave like independent samples of binomial random graphs at the appropriate densities.

To prove \eqref{eq:sketch-small-subgraphs}, we split the two events in the conjunction using Harris's inequality: one is an increasing event in the randomness of $G_{i+1}'$ and the other is a decreasing event for each fixed choice of $Q_{i+1}$. We then bound the quantity $\PP_{i+1}( H \subset O_{i+1} )$ using Janson's inequality (see Theorem~\ref{thm:Janson}). The only work here is in checking that the quasi-randomness properties of the graphs $G_{\leq i}'$, $O_{i}$ give the correct accounting for our application of Janson. (See Section~\ref{sec:deg-codeg-2} for the precise statements and more detail.)

Interestingly, \eqref{eq:sketch-small-subgraphs} is enough to get the concentration we need. To give the reader a feel for this, let us sketch how we maintain control on the maximum degree of $O_i$ throughout the process. In particular, we sketch why we have
\begin{equation}\label{eq:sketch-Delta-Oi} \Delta(O_i) \leq (1+\eps_i) \theta_i n ,\end{equation}
for all $i\leq T$, with probability $1-o(1)$. Here $\eps_i = i\g^{3}$ is our error term for step $i$. To prove this, assume that we have maintained the inequality \eqref{eq:sketch-Delta-Oi} up to step $i$. We now investigate the probability that it fails at step $i+1$. For this, we first observe that for all vertices $x$, we have (recalling \eqref{eq:def-si+1})
\[ \EE_{i+1}\, |N^{\circ}_{i+1}(x)| = s_{i+1}|N_{i}^{\circ}(x)| \leq (1+\eps_i)s_{i+1}\theta_i n =: \mu , \] where the inequality holds by the assumption \eqref{eq:sketch-Delta-Oi} holds at step $i$. We will then show that 
\begin{equation}\label{eq:sketch-nbh-bound} \PP_{i+1}\big( |N_{i+1}^{\circ}(x)| \leq (1 + \gamma^4 )\mu  \big) \geq 1-n^{-\omega(1)},\end{equation} for all vertices $x \in [n]$. Union bounding over all $x$ gives 
\[ \Delta(O_{i+1}) \leq (1+\g^4)\mu  = (1+\g^4)(1+\eps_i)s_{i+1}\theta_in,\] 
with high probability. We will then prove the random variable $\theta_i$ is close to its mean $\EE_{i+1} \theta_{i+1} = s_{i+1}\theta_i$. That is,
\begin{equation}\label{eq:sketch-theta-bound} \PP_{i+1}\big( \theta_{i+1}  \geq (1-\g^{4})s_{i+1}\theta_i \big) \geq 1- n^{-c},\end{equation}
which allows us to conclude that that 
\[ \Delta(O_{i+1})\leq (1+\gamma^4)^2(1+\eps_i)\theta_{i+1} n \leq (1+\eps_{i+1})\theta_{i+1} n, \] with probability $1-n^{-c}$, as desired. 

The proofs of both \eqref{eq:sketch-nbh-bound} and \eqref{eq:sketch-theta-bound} use moment-type arguments. The proof of \eqref{eq:sketch-theta-bound} is the most direct.
Write $Z = e(O_{i+1}) = \theta_{i+1} \binom{n}{2}$. We have 
\[ Z = \sum_{e \in O_i} \1(e \in O_{i+1}) \]
and so we can express the second moment as 
\begin{equation}\label{eq:sketch-moment} \EE\, Z^2 =  \sum_{e, \hspace{0.3mm} f} \PP\big( e, f \in O_{i+1} \big) \leq  (1+n^{-c})^2\sum_{e \hspace{0.5mm} \not= f}\PP\big( e \in O_{i+1} \big)\PP\big( f \in O_{i+1} \big) + \EE\, Z ,  \end{equation} where the inequality holds by applying \eqref{eq:sketch-small-subgraphs} to $K = \emptyset$ and $H = \{ e, f\}$, when $e \not= f$. We now bound the sum on the right hand side of \eqref{eq:sketch-moment} by $\EE\, Z^2$ and apply Chebyshev's inequality to obtain concentration (see Lemma~\ref{lem:theta-conc} for more detail).

For \eqref{eq:sketch-nbh-bound} we need a slightly stronger concentration estimate since we need to sum over all $x\in [n]$. Luckily here we only need to control the upper tails and so we can use a slightly different method. In particular, if we write $Z_x(t)$ be the number of stars $K_{1,t}$, $t = n^{o(1)}$, contained in $O_{i+1}$ with center $x$, we can bound
\[ \EE_{i+1} Z_x(t) \geq \PP_{i+1}\big(  |N_{i+1}^{\circ}(x)| \geq (1+\eps)\mu  \big) \cdot \binom{ (1+\eps)\mu}{t} . \]
We then realise $Z_{x}(t)$ as a sum over all stars $S$ with center $x$ in $O_i$ and write  
\[ \EE_{i+1}\, Z_{x}(t) \leq \sum_S\hspace{0.5mm} \PP( S \subset O_{i+1}) \leq (1+n^{-c})^t s_{i+1}^t \cdot \binom{|N_i^{\circ}(x)|}{t} , \] comparing the upper and lower bounds completes the proof (see Lemma~\ref{lem:cA-open-deg} for details). Thus we can prove \eqref{eq:sketch-nbh-bound} and complete our analysis of the maximum degree.

Maintaining similar bounds on the  codegrees and cross-degrees in \eqref{eq:intersections} is not much harder and follows the same strategy. We point the reader to Section~\ref{sec:deg-codegs-3} for details. 

We remark that it is an extremely important feature of these bounds that they are \emph{independent} of the value of $\theta_i$, thus any errors that accumulate in our understanding of $\theta_i$ don't ``feed back'' into this part of the process. In particular, the degrees, codegrees and cross-degrees form a ``closed system of inequalities'' that we neatly ``chain'' as the process runs (as we saw above). For these parameters we maintain multiplicative errors of the size $1+\eps_i$ at step $i$ (e.g. see \eqref{eq:sketch-Delta-Oi}). From these inequalities we \emph{derive} that 
\[\theta_{i} \geq  (1-i\eps_{i}) \exp\hspace{-0.5mm}\big(\hspace{-0.5mm}-\hspace{-0.5mm}\g^2 i^2-2\gamma i\alpha_0\sqrt{\log n}\big)\theta_0  \]
where the error term here is notably larger than $\eps_i$. This is acceptable since none of our bounds above rely on this error term. This is in contrast to the papers \cite{bohman2009triangle},\cite{bohman2021dynamic} and \cite{FGM} where the errors from fluctuation in the density of open edges feed back into the rest of the process, resulting in a very subtle dynamics. 

Actually our basic analysis of the process is not much more than what we have sketched above and is the content of (the short) sections: Sections~\ref{sec:density-of-process}, ~\ref{sec:deg-codeg-2}, and \ref{sec:deg-codegs-3}. In the remaining sections we turn to prove the upper bound on the independence number of $G_{\leq T}$. In Section~\ref{sec:moderate-degrees} we control the number of vertices that have large degree to an independent set $I\in[n]^{(k)}$. In Section~\ref{sec:open-edges-in-I-redux} we define the important notation of the ``core'' of an independent set $I$ and set up our main lemma which says that independent sets contain the ``correct'' number of open edges. In Section~\ref{sec:open-edges-in-I-martingale} we use a martingale argument to finish the proof that independent sets contain the correct number of open edges. (We highlight that our martingale and analysis is significantly different from that in \cite{bohman2021dynamic}, \cite{FGM}). In Section~\ref{sec:prob-k-set-independent} we make the heuristic in Section~\ref{sec:heuristic} rigorous and prove our bound on the probability a $k$-set $I$ is independent in the final graph $G_{\leq T}$. Then, finally, in Section~\ref{sec:proof-of-main-thm} we put these pieces together to finish the proof of our main theorem.

\subsection{Universal notation}

As we will work quite a bit with the graphs $O_i,G_{\leq i}, G_i, G'_{\leq i}, G'_{ i}$ we reserve some special notation for them. As we defined above, we let $N^{\circ}_i(x)$ denote the neighbours of $x$ in the graph $O_i$. We let $N_{i}(x)$ denote the neighbourhood in the graph $G_{i}$, $N_{\leq i}(x)$ denote the neighbourhood in the graph $G_{\leq i}$ and 
$N'_{i}(x), N'_{\leq i}(x)$ denote the neighbourhoods in the graphs $G'_{i}$ and $G'_{\leq i}$, respectively. 

As we saw above, $0 < \delta < 1/8$ is a given small constant. We will need a constant $\beta>0$ for which $\beta< \delta/2^7$ (see \eqref{eq:L-def}). At \eqref{eq:def-gamma} we defined $\gamma$ universally throughout the paper, we also define an ``error function'' for the $i$th round of this process $\eps_i$. So we record
\[ 0 < \beta < \delta/2^{7}, \qquad \gamma = (\log n)^{-10},  \qquad  r = (\log n)^2  \qquad \text{ and }  \qquad \eps_i = (i+1)\g^{3}  ,\]
for $i \geq 0$. Throughout the paper we also think of $k$ as fixed and universal
\[ k = (1+\delta)\sqrt{(3n/2) \log n } . \]

Given the partition $[n] = V_1 \cup \cdots \cup V_{n/r}$ defined by the blow-up in the seed step in Section~\ref{subsec:seed} (which we can think of as fixed throughout the paper), define the ``projection'' map $\pi : [n] \rightarrow [n/r]$,
by 
\begin{equation}\label{eq:def-pi} \pi(x) = j \qquad \text{ where } \qquad x \in V_j .\end{equation}
It will also be useful to extend this definition to graphs. Indeed if \[ H \subset [n]^{(2)} \qquad \text{ we define } \qquad \pi(H) = \{ \pi(x)\pi(y) : xy \in H \}.\] 

Throughout the paper we assume that $n$ is sufficiently large so that all our statements and inequalities hold.

\section{Degrees and co-degrees I: statement of results and consequences }\label{sec:density-of-process}

In this short section we define $\cA$, a ``quasi-randomness'' event which says, roughly speaking, that the degrees and codegrees of the graphs $G'_{\leq i}, O_i$ stay on-track as the process runs. We then state that this event holds with high probability and derive a few consequencs of $\cA$ that will be useful throughout the paper. In the following two sections, we go on to prove that $\cA$ holds with high probability.

\subsection{Definition of the event \texorpdfstring{$\cA$}{A}} We define $\cA$ by defining the events $\cA_i = \cA_i(O_{i-1},G_{\leq i-1}')$ which only depend on the randomness in the $i$th step. We then define the ``global'' event 
\[\cA = \bigcap_{0\leq i\leq T} \cA_i \, .\] Recall that $\theta_i$ is the density of the graph $O_i$ at step $i$. We define $\cA_i$ to be the intersection of the following properties of $O_i, G'_{\leq i}$. First, we specify that on $\cA_0$ we have
\begin{equation}\label{eq:theta-0} \theta_0 = (1-o(1))r^{-2}.\end{equation} While we don't specify the density of $\theta_i$, we do require it is close to its mean. For all $i\geq 0$ we require that on the event $\cA_i$ we have
\begin{align}\label{eq:theta-conc-Ai}
\big| \theta_i-\EE_i\, \theta_i \big|\leq \g^4 \cdot \EE_{i}\, \theta_i \, .
\end{align} Recall $\EE_i\, \theta_i = s_i\theta_{i-1}$ from \eqref{eq:def-si+1}. 
We then include the degree and codegree conditions into $\cA_i$
\begin{equation}\label{eq:def-cE_i}  
\Delta(O_i) \leq (1 + \eps_i) \theta_i n  \qquad \text{ and } \qquad \Delta_2(O_i) \leq (1 + \eps_i)\theta_i^2 n, \end{equation}
where $\Delta(G)$ denotes the maximum degree of a graph $G$ and $\Delta_2(G)$ denotes the maximum codegree $\Delta_2(G) = \max_{x\not=y}|N_G(x)\cap N_G(y)|$. We also require that on $\cA_i$ we have 
the events \begin{equation}\label{eq:def-cF_i} 
\Delta(G'_{i})\leq (1+\eps_i) \gamma \sqrt{n}\, ,\end{equation}
for all $i \geq 1$. Defining $\psi_i= i\g/\sqrt{n} +p_0$, we also maintain, for all $i\geq 0$,
\begin{equation}\label{eq:def-cF_i2}  
\Delta(G'_{\leq i})\leq (1+\eps_i)\psi_i n \quad \, \text{ and } \, \quad \, |N_{\leq i}'(x)\cap N_{\leq i}'(y)| \leq (\log n)^4 + i(\log n)^2\, ,
\end{equation}
whenever $\pi(x) \neq \pi(y)$. Recall from~\eqref{eq:Psi-Approx} that $\psi_i$ approximates the density of $G'_{\leq i}$ and $\pi$ identifies the part of the partition a vertex is in \eqref{eq:def-pi}. Finally, we require the cross-degree condition: for all $x,y\in [n]$ we have 
    \begin{equation}\label{eq:cross-degrees}
    |N'_{\leq i }(x) \cap N^{\circ}_{i}(y)| \leq  (1+\eps_i)\theta_i\psi_i n,
\end{equation} on the event $\cA_i$

This concludes the definition of the event $\cA_i$. In the next two sections we will show that $\cA_i$ holds, at each step, with high probability.

\begin{lemma}\label{lem:cA} We have
\[
\P(\cA)=1-o(1)\, .
\]
\end{lemma}

In the remainder of this section, we limit ourselves to deriving a few useful consequences of $\cA$ that will be helpful throughout the paper.

\subsection{Consequences of the events \texorpdfstring{$\cA_i$}{Ai}}

As we mentioned in Section~\ref{sec:evo-of-process-sketch} control on the degrees and codegrees immediately imply control of $X_i(e)$ and $Y_i(e)$.

\begin{observation}\label{lem:X(e)-Y(e)-good} If $\cA_i$ holds, then for all $e \in O_{i}$ we have
\begin{equation}\label{eq:def-cG_i}
X_i(e)\leq (1+\eps_i) \theta_i^2 n  \qquad \text{and} \qquad Y_i(e)\leq (1+\eps_i) 2\theta_{i}\psi_i n \,  \end{equation}
and in particular
\[p^2_{i+1}X_i(e) + p_{i+1}Y_i(e)  \leq (1+\eps_{i}) \big((2i+1)\gamma^2 +2\gamma \alpha_0 \sqrt{\log n}\big) = O\big(\gamma\sqrt{\log n}\big).\]
\end{observation}
\begin{proof} For the first two inequalities, simply use the degree and codegree conditions afforded by $\cA_i$. For the third, use that $p_{i+1}\theta_i=\gamma/\sqrt{n}$ and $p_0=\alpha_0\sqrt{\log n/n}$. For the final equality, recall from \eqref{eq:T_0} that $i\leq T=\alpha_1\gamma^{-1}\sqrt{\log n}$.
\end{proof}

As we saw in Section~\ref{sec:evo-of-process-sketch} at \eqref{eq:si+1-formula}, control on the $X_i(e),Y_i(e)$ in turn provides us with control on $s_{i+1}$.

\begin{observation}\label{obs:si-bd}
If $\cA_i$ holds, then
    \[ s_{i+1}  \geq (1-\eps_{i})\exp\hspace{-0.5mm}\big(\hspace{-0.5mm}-\hspace{-0.5mm}(2i+1)\g^2-2\gamma\alpha_0\sqrt{\log n}\big) .\]
\end{observation}
\begin{proof}
From \eqref{eq:si+1-formula} we have
\begin{align}
s_{i+1} &= \min_{e \in O_i} \, (1-p_{i+1}^2)^{ X_i(e) }(1-p_{i+1})^{Y_i(e)}\\
&\geq \min_{e\in O_i} \exp\hspace{-0.5mm}\big(-(1+p_{i+1}) (p_{i+1}^2X_i(e) + p_{i+1}Y_i(e))\big),\label{eq:si+1-formula-bd}
\end{align}
using the inequality $(1-x)\geq e^{-x(1+x)}$, which holds for $x\in[0,1/2]$.
We now apply Observations~\ref{lem:X(e)-Y(e)-good} to bound \begin{align*}s_{i+1} &\geq \exp\left(-(1 + p_{i+1})(1 + \eps_i)\big((2i+1)\g^2+2\gamma\alpha_0\sqrt{\log n}\big)\right) \\
&\geq (1 - \eps_i)\exp\left(-(2i+1)\g^2-2\gamma\alpha_0\sqrt{\log n}\right)
\end{align*} where we used that $((2i+1)\g^2-2\gamma\alpha_0\sqrt{\log n}) = o(1)$ and $p_{i+1} \leq \eps_i$.
\end{proof}

We can now chain these inequalities together to obtain a lower bound on the density of $O_i$

\begin{observation}\label{obs:mean-of-theta} If $\cA_0 \wedge \cdots \wedge \cA_i$ holds, then 
\[\theta_{i+1} \geq  (1-(i+1)\eps_{i}) \exp\big( - \g^2 (i+1)^2-2\gamma (i+1)\alpha_0\sqrt{\log n}\big)\theta_0  \]
\end{observation}
\begin{proof} We prove the result by induction. The result holds trivially at $i=-1$. So assume holds for $i$ and we prove it for $i+1$. From \eqref{eq:def-si+1}, we have $\EE_{i+1} \theta_{i+1} = s_{i+1}\theta_i$. So by Observation~\ref{obs:si-bd} and the induction hypothesis (since $\cA_0 \wedge \cdots \wedge \cA_i$ holds) we have
\[  s_{i+1}\theta_i \geq  (1-\eps_i)(1-i\eps_{i-1})\exp\big(-(2i+1)\g^2-2\gamma\alpha_0\sqrt{\log n}\big) \exp\big( - \g^2 i^2-2\gamma i\alpha_0\sqrt{\log n}\big)\theta_0 . \]
We now use  \eqref{eq:theta-conc-Ai} along with the inequality $(1 - \eps_i)(1 - i \eps_{i-1}) - \gamma^4 \geq 1 - (i+1) \eps_i$ to finish. \end{proof}

We also record the following useful consequence of this Lemma. 

\begin{observation}\label{obs:thetai-is-prod-si} For $i < T$, on the event $\cA_0 \wedge \cdots \wedge \cA_i$ we have
\[ \theta_{i+1} = (1+O((i+1)\gamma^4))s_{i+1} \cdots s_1 \theta_0 . \]
\end{observation}
\begin{proof}
We prove the statement by induction. It is an identity for $i=-1$. So assume it holds for $i \geq -1$. By \eqref{eq:def-si+1} and induction we have  
\[ \EE_{i+1} \theta_{i+1} = s_{i+1}\theta_i = (1-O(i\g^4))s_{i+1} \cdots s_1 \theta_0. \]
By \eqref{eq:theta-conc-Ai}, on the event $\cA_i$ we have $\theta_{i+1} = (1+O(\gamma^4))\EE_{i+1} \theta_{i+1} $. The statement follows. 
\end{proof}

From the above, we note the following useful bounds on $p_{i+1}$ and $\theta_i$.
\begin{observation}\label{lem:pi-UB}
    If $i< T$ and $\cA_0 \wedge \cdots \wedge \cA_i$ holds, then
    $p_{i+1} \leq n^{-\delta}$ and $\theta_i \geq n^{-1/2+\delta}$.  
\end{observation}
\begin{proof}
From Observation~\ref{obs:mean-of-theta}, we know that on the event $\cA_0 \wedge \cdots \wedge \cA_i$ we have 
\[ \theta_i \geq (1 - o(1) )  \exp\big(\hspace{-0.5mm} -\hspace{-0.5mm}\g^2 i^2-2\gamma i\alpha_0\sqrt{\log n}\hspace{0.5mm}\big)\cdot \theta_0 \geq (1-i\eps_i) \exp\big(\hspace{-0.5mm}-\big( \alpha_1^2 + 2\alpha_0\alpha_1\big)\log n\big)\cdot \theta_0.\] 
For the second inequality, we used that $i\leq T = \alpha_1 \gamma^{-1} \sqrt{\log n}$ and the definition of $\eps_i = (i+1)\gamma^3$. Now, using that $2\alpha_0\alpha_1 + \alpha_1^2 \leq 1/2 -2\delta$ by \eqref{eq:alpha_0alpha_1-setup} and $\theta_0=(1-o(1))r^{-2}=n^{-o(1)}$ by \eqref{eq:theta-0}, we see 
$\theta_i \geq n^{-1/2+\delta}$ , as desired. The bound on $p_{i+1}$ holds since $p_{i+1} =  \gamma/(\theta_in^{1/2}) \leq n^{-\delta}$.
\end{proof}

\section{Degrees and co-degrees II: the probability of small subgraphs} \label{sec:deg-codeg-2}

The main goal of this short section is to prove the following useful lemma on the appearance of ``small'' subgraphs in $O_i \cup G_{\leq i}'$. In addition to being useful later when we bound the independence number of $G_{\leq T}$, we will apply it multiple times in our proof of Lemma~\ref{lem:cA}.

\begin{lemma}\label{lem:probability-open'} Let $\cA_i$ hold and  let $H,K \subset O_i$ be such that 
\[ \Delta(H),\, \Delta(K)\leq n^{\delta/4} \qquad \text{ and } \qquad \Delta(G'_{\leq i }[V(H)]) \leq n^{\delta/4}\, .\] Then \begin{equation}\label{eq:prob-H-survies}
\PP_{i+1}\big( H\subset O_{i+1} \, \wedge \, K \subset G_{i+1}' \, \big) \leq \big(1+n^{-\delta/4}\big)^{e(H)}s_{i+1}^{e(H)}p_{i+1}^{e(K)}\, . \end{equation}
\end{lemma}

To prove this lemma, we will appeal to Janson's inequality, proved by Janson, Łuczak, and Ruciński \cite{janson1990exponential}  (see \cite{alon2016probabilistic}).

\begin{theorem}[Janson's inequality]\label{thm:Janson} For $p\in [0,1]$, let $\cH$ be a hypergraph on vertex set $V$ and let $X \subset V$ be a $p$-random set. Then
\[ \PP( X \not\supset \sigma, \text{ for all } \sigma\in \cH) \leq e^{-\mu(\cH) + \Delta(\cH) },  \]
where \[ \mu(\cH) = \sum_{\sigma\in \cH } \PP( X \supset \sigma ) \quad 
\text{ and } \quad \Delta(\cH) = \sum_{\sigma,\tau \in \cH: \sigma\cap \tau \neq \emptyset } \PP( X \supset \sigma \cup \tau ).\]

\end{theorem}

To see how Janson's inequality connects to the problem at hand, we note the following. Define the hypergraph $\cH = \cH(H)$ on vertex set $V=O_i$ and with edge set consisting of the following singletons and pairs
\begin{align}\label{eq:cH-def}
\cH_1 = \big\{\{f\}: f \in \cY_i(e) \text{ for some } e\in H \big\} \qquad \text{ and } \qquad \cH_2 = \bigcup_{e \in H} \cX_i(e) ,
\end{align} where we recall the definitions of $\cX_i,\cY_i$ from \eqref{eq:def-Y},\eqref{eq:def-X}.
We now recall that $G'_{i+1}$ is a $p_{i+1}$-random subset of $O_i$ and observe that
\[ \PP_{i+1}( H \subset O_{i+1} ) = \PP\big( G'_{i+1} \not\supset f \text{ for all } f \in \cH \big) \cdot \prod_{e\in H}q_e ,\]
where the $q_e$ are the quantities defined at \eqref{eq:qedef}, in the definition of the process.

Before turning to the proof of Lemma~\ref{lem:probability-open'}, we collect some combinatorial observations that are key to our application of Janson's inequality.
 
\begin{observation}
    \label{lem:combi-open-lemma-1} Let $\cA_i$ hold and let $H \subset O_i$ be such that  $\Delta(H) \leq n^{\delta/4}$ and $\Delta(G'_{\leq i }[V(H)]) \leq n^{\delta/4} $. If $e \in H$ then
\[ \sum_{f \in H, f \not= e} \big| \cY_i(e) \cap  \cY_i(f) \big|  \leq 8n^{\delta/2}\, . \]
\end{observation}
\begin{proof}
We can see this by considering two cases. First, consider pairs $f$ such that $f \cap e = \emptyset$. Note that in this case, we have $|\cY_i(e)\cap \cY_i(f)| \leq 2$. If we write $e=xy$ and $f=uv$, we note that we must have 
$u,v \in N'_{\leq i}(x)\cup N'_{\leq i}(y)$ and therefore there are at most $2n^{\delta/2}$  choices for $f$, by our assumption $\Delta(G'_{\leq i}[V(H)]) \leq n^{\delta/4}$, and thus  
\[ \sum_{f\in H :\, e\cap f = \emptyset} |\cY_i(e) \cap \cY_i(f)| \leq 4n^{\delta/2}. \]

Now consider $f$ such that $e \cap f \not= \emptyset$. In this case we claim that $|\cY_i(e) \cap \cY_i(f)| \leq \g^{-1}$, by the codegree condition at~\eqref{eq:def-cF_i}. Indeed,  if we write $e=xy$ and $f=xz$ then we must have $\pi(y) \neq \pi(z)$ since for each $a \neq b$ we have $|O_i[V_a, V_b] | \leq 1$.  Further,  any $h\in \cY_i(e) \cap \cY_i(f) \backslash \{e,f\}$ must be of the form $h=xw$ where $w\in N'_{\leq i}(y)\cap N'_{\leq i}(z)$ and so there are at most $\Delta_2(G'_{\leq i})+2\leq \gamma^{-1}$ choices for $h$. 
We also note that are at most $2n^{\delta/4}$ such choices of $f$, by the maximum degree condition on $H$. Thus
\[ \sum_{f\in H :\, e\cap f \not= \emptyset} |\cY_i(e) \cap \cY_i(f)| \leq 2n^{\delta/4}\g^{-1}. \qedhere \]
\end{proof}

\begin{observation}
    \label{lem:combi-open-lemma-2} Let $H \subset O_i$ be such that  $\Delta(H) \leq n^{\delta/4}$ and let $\cH=\cH(H)$ be defined as in~\eqref{eq:cH-def}.
Then
\[
\big|\big\{(\sigma,\tau)\in \cH_1 \times\cH_2: \sigma\cap \tau\neq \emptyset\big\}\big|\leq 2n^{\delta/4}|\cH_1|
\]
and 
\[ \hspace{1.9mm}
\big|\big\{(\sigma,\tau)\in \cH_2 \times\cH_2: \sigma\cap \tau\neq \emptyset\big\}\big|\leq 6n^{\delta/4}|\cH_2|\, .
\]
\end{observation}
\begin{proof}
Fix $\sigma=\{f\}\in\cH_1$ and suppose that $f$ is contained in some $\tau\in\cX_i(e)$ for some $e\in H$. For this to occur $e$ must be incident to $f$ and $\tau=\{f,g\}$ where $efg$ form an open triangle. There are therefore at most $2\Delta(H)\leq 2n^{\delta/4}$ choices for $e$ and hence also $\tau$. The first claim follows. 

Now fix $\sigma=\{f,g\}\in\cH_2$. If $\sigma \cap \tau\neq\emptyset$ for some $\tau\in \cX_i(e)$ where $e\in H$ then $e$ must be incident to $f$ or $g$. There are therefore at most $3\Delta(H)\leq 3n^{\delta/4}$ choices for $e$ and given $e$ there are at most $2$ choices for $\tau$ ($ef$ and $eg$ may both span an open triangle).
\end{proof}

We are now ready to prove Lemma~\ref{lem:probability-open'}.

\begin{proof}[Proof of Lemma~\ref{lem:probability-open'}] 
Since 
\[ O_{i+1} = Q_{i+1} \cap \big\{ e \in  (K_n  \setminus G_{\leq i+1}') \cap O_{i}: e \text{ does not form a triangle with } G_{\leq i+1}' \big\}, \]  increasing the set  $G_{i+1}'$ cannot increase the set $O_{i+1}$.  We thus have that $\{H \subset O_{i+1}\}$ is a decreasing event in $G_{i+1}'$.  We also have that $\{K \subset G_{i+1}'\}$ is an increasing event in $G_{i+1}'$ and so Harris's inequality allows us to split the probability
\[ \PP_{i+1}\big( H\subset O_{i+1} \, \wedge \, K \subset G_{i+1}' \, \big) \leq \PP_{i+1}\big( H\subset O_{i+1}\big)\PP_{i+1}\big( K \subset G_{i+1}' \big) = \PP_{i+1}\big( H\subset O_{i+1}\big)p_{i+1}^{e(K)}.  \]
Janson's inequality (Theorem~\ref{thm:Janson}) allows us to bound 
\[  \PP_{i+1}\big( H\subset O_{i+1}\big) =  \PP\big( G'_{i+1} \not\supset f \text{ for all } f \in \cH \big) \prod_{e \in H} q_e \leq \exp(-\mu + \Delta ) \prod_{e \in H} q_e  ,\]
where $\cH=\cH(H)$ is the hypergraph defined at~\eqref{eq:cH-def} and  $\mu=\mu(\cH), \Delta=\Delta(\cH)$.
We calculate
\[ \mu = p_{i+1}|\cH_1| + p_{i+1}^2|\cH_2| \geq \sum_{e\in H} (p_{i+1}^2X_i(e) + p_{i+1}Y_i(e)) -  8p_{i+1}n^{\delta/2}e(H), \]
where the last inequality holds since 
\[|\cH_1|= \big| \bigcup_{e \in H} \cY_i(e) \big|  \geq \sum_{e \in H} Y_i(e)  - \sum_{e\not= f} |\cY_i(e) \cap \cY_i(f)| \geq \sum_{e \in H} Y_i(e)  - 8n^{\delta/2} e(H)\, , \]
where the last inequality holds by Observation~\ref{lem:combi-open-lemma-1}. On the other hand, since two edges of an open triangle determine the third, we have 
\begin{align}\label{eq:cH2} |\cH_2|=\big|\bigcup_{e \in H} \cX_i(e)\big| = \sum_{e \in H} X_i(e)\, .
\end{align}
By \eqref{eq:qedef}, \eqref{eq:prob-e-open-constant} and \eqref{eq:si+1-formula} we have $$(1 - p_{i+1}^2)^{X_i(e)}(1 - p_{i+1})^{Y_i(e)} q_e = s_{i+1}$$
for all $e \in O_i$ and so, by Observation~\ref{lem:X(e)-Y(e)-good} and Observation~\ref{lem:pi-UB} we have 
\[ \exp(-p_{i+1}^2X_i(e) - p_{i+1}Y_i(e) +8p_{i+1}n^{\delta/2})q_e \leq \big(1+O(n^{-\delta/2})\big)s_{i+1}\, . \]
Thus
\[
\exp(-\mu)\prod_{e\in H}q_e \leq  \big(1+O(n^{-\delta/2})\big)^{e(H)} s^{e(H)}_{i+1}\, .
\]
It remains to bound $\Delta$. By Observation~\ref{lem:combi-open-lemma-2} we have
\[ \Delta \leq 2n^{\delta/4}|\cH_1|p^2_{i+1}+ 6n^{\delta/4}|\cH_2|p^3_{i+1} \leq p_{i+1}n^{\delta/4}e(H) \leq n^{-\delta/2} e(H) , 
\]
where the second inequality follows from~\eqref{eq:cH2}, the union bound $|\cH_1|\leq \sum_{e\in H}Y_i(e)$, and Observation~\ref{lem:X(e)-Y(e)-good}. The final inequality holds by Observation~\eqref{lem:pi-UB}. The result follows.
\end{proof}

\section{Degrees and codegrees III: Proof of Lemma \texorpdfstring{\ref{lem:cA}}{4.1}} \label{sec:deg-codegs-3}

In this section we prove Lemma~\ref{lem:cA}. The crux is the following lemma.

\begin{lemma}\label{lem:Ai-Aiplus1}
    Let $0\leq i < T$. If $\cA_1 \wedge \cdots \wedge \cA_i$ holds then 
    \[
    \P_{i+1}(\cA_{i+1})= 1-O(n^{-\delta/8}) \, .
    \]
\end{lemma}

We prove Lemma~\ref{lem:Ai-Aiplus1} in several steps, addressing each property of $\cA$ in turn. But before we get to this, we note that the desired properties hold for the seed graph, with high probability.

\begin{lemma}\label{lem:cA-for-seed}We have $\PP( \cA_0 ) = 1-o(1)$.
\end{lemma}
Since the proof of Lemma \ref{lem:cA-for-seed} is simply a (easy) statement about the binomial random graph, we postpone the proof to Appendix~\ref{sec:seed}. We now turn to show that Lemma~\ref{lem:cA} follows immediately from the above two lemmas. 

\begin{proof}[Proof of Lemma~\ref{lem:cA}]
By Lemmas~\ref{lem:Ai-Aiplus1} and~\ref{lem:cA-for-seed},  we have
\[\P(\cA)=\prod_{i=0}^T\PP(\cA_i \mid \cA_0,\ldots,\cA_{i-1})\geq (1-o(1))\big(1-O(n^{-\delta/8})\big)^T=1-o(1),
\]
where, for the final inequality, we recall $T = O(\g^{-1}\sqrt{\log n}) = O((\log n)^{11})$.
\end{proof}

Thus we turn to our main task for this section which is the proof of Lemma~\ref{lem:Ai-Aiplus1}. We begin with the harder part of the proof, by checking the various properties of $\cA_{i+1}$ involving $O_{i+1}$. The properties involving only $G_{\leq i+1}'$ are significantly easier. We start by showing \eqref{eq:theta-conc-Ai} that $\theta_{i+1}$ (the density of $O_{i+1}$) is concentrated about its mean. 

\begin{lemma}\label{lem:theta-conc}
If $\cA_i$ holds, then 
\[ \PP_{i+1} \big( |\theta_{i+1} - \EE_{i+1}\, \theta_{i+1} | \leq \gamma^4 \cdot \E_{i+1}\, \theta_{i+1}  \big) \geq 1- n^{-\delta/8}\, . \]
\end{lemma}
\begin{proof}
We consider the random variable $X = e(O_{i+1})$ and use the second moment method. To calculate $\EE_{i+1} X^2$, we note that for any pair of distinct $e, f \in O_i$, we have 
\[ \PP_{i+1}\big( e \in O_{i+1} \wedge f \in O_{i+1} \big)  \leq (1+n^{-\delta/4})^2 \PP_{i+1}( e \in O_{i+1}) \PP_{i+1}( f \in O_{i+1}), \] by applying Lemma~\ref{lem:probability-open'} to the graph $H = \{e,f\}$. It follows that 
\[ \Var_{i+1}(X) \leq \EE_{i+1} X + O(n^{-\delta/4})(\EE_{i+1} X)^2.\] By Chebyshev's inequality, for $\eps>0$,
\[ \PP_{i+1}\big( |X - \EE_{i+1} X| \geq \eps \cdot \EE_{i+1} X \big) \leq \frac{\Var_{i+1}(X)}{\eps^2 (\EE_{i+1} X )^2 } \leq \frac{1}{\eps^2 \EE_{i+1} X}+ O\big(n^{-\delta/4}\eps^{-2}\big)\, .\] 
Setting $\eps = \gamma^4$ and noting that $\EE_{i+1} X=\binom{n}{2} s_{i+1}\theta_i\geq n^{3/2}$ (e.g.\ by Observations~\ref{lem:pi-UB} and~\ref{obs:si-bd})  yields the result.
\end{proof}

We now turn to deal with the degrees and codegrees of the graph of open edges.

\begin{lemma} \label{lem:cA-open-deg} If $\cA_i$ holds, then
\[ \PP_{i+1}\big( \Delta(O_{i+1}) \leq (1 + \eps_{i+1} ) \theta_{i+1} n  \big) \geq  1-2n^{-\delta/8} .\]
\end{lemma}
\begin{proof}
We first consider a fixed vertex $x\in [n]$ and then union bound over all vertices.
For $t \leq n^{\delta/4}$, let $Z_x(t)$ be the number of copies of $K_{1,t}$ in $O_{i+1}$ where $x$ is the vertex of degree $t$ in the $K_{1,t}$. 

Now set $\mu = s_{i+1}\Delta(O_i)$ and note that for any $\eps \geq 0$, we have 
\begin{equation} \label{eq:LB-on-Zx} \EE_{i+1} Z_x(t) \geq \PP_{i+1}\big(  |N_{i+1}^{\circ}(x)| \geq (1+\eps)\mu  \big) \cdot \binom{ (1+\eps)\mu}{t}\, , \end{equation}
by Markov's inequality.
On the other hand, we have 
\begin{equation} \label{eq:UB-on-Zx} \EE_{i+1} Z_x(t) = \sum_{S \in (N_{i}^{\circ}(x))^{(t)} } \PP_{i+1}( xy \in O_{i+1} \text{ for all } y \in S ) \leq \big(1+ n^{-\delta/4}\big)^ts_{i+1}^t\binom{|N_{i}^{\circ}(x)|}{t} , \end{equation}
where the last inequality holds by Lemma~\ref{lem:probability-open'}, since $t \leq n^{\delta/4}.$ 
Now compare the upper and lower bounds for $\EE_{i+1} Z_x(t)$ with $\eps = \gamma^4$ and $t = 1/\gamma^5$ to get 
\[ \PP_{i+1}\big(  |N_{i+1}^{\circ}(x)| \geq (1+\eps) \mu  \big)  \leq  \bigg(\frac{1+ n^{-\delta/4} }{1+\eps}\bigg)^t \leq e^{-\eps t/2 } \leq n^{-2}\, . \]
Thus union bounding over all $x\in[n]$ yields $\PP_{i+1}\big(\Delta(O_{i+1}) \leq (1+\eps)\mu  \big) \geq 1-n^{-1}$.  To conclude the proof, we show that $(1+\eps)\mu\leq (1 + \eps_{i+1}) \theta_{i+1} n$ with probability at least $1-n^{-\delta/8}$. Indeed, this follows by noting that $\Delta(O_i)\leq (1+\eps_i)\theta_i n$ by our assumption that $\cA_i$ holds (see~\eqref{eq:def-cE_i}); and that $\theta_{i+1}\geq (1-\g^4)s_{i+1}\theta_i$ with probability at least $1-n^{-\delta/8}$ by Lemma~\ref{lem:theta-conc}.
\end{proof}

We deal with codegrees in $O_i$ in a very similar way. 

\begin{lemma}\label{lem:cA-open-codegs} For If $\cA_i$ holds, then
\[ \PP_{i+1}\big( \Delta_2(O_{i+1}) \leq (1 + \eps_{i+1} )\theta_{i+1}^2 n   \big) \geq  1-2n^{-\delta/8} .\]
\end{lemma}
\begin{proof} 
Fix distinct $x,y\in [n]$ and for $t\leq n^{\delta/4}$ let $Z_{x,y}(t)$ be the number of copies of $K_{2,t}$ in $O_{i+1}$ where $x$ and $y$ are the vertices of degree $t$ in the $K_{2,t}$. As above, we bound $\E_{i+1}Z_{x,y}(t)$ in two ways.
For $\eps \geq 0$ and $\mu = s^2_{i+1}\Delta_2(O_i)$, we have that 
\[ \EE_{i+1} Z_{x,y}(t) \geq \PP_{i+1}\big(  |N_{i+1}^{\circ}(x)\cap N_{i+1}^{\circ}(y)| \geq (1+\eps)\mu  \big) \cdot \binom{ (1+\eps)\mu}{t}. \]
On the other hand, by Lemma~\ref{lem:probability-open'}, we have 
\[ \EE_{i+1} Z_{x,y}(t)  \leq \big(1+ n^{-\delta/4}\big)^{2t}\binom{|N_{i}^{\circ}(x)\cap N_{i}^{\circ}(y)|}{t}s_{i+1}^{2t} .\]
Comparing these bounds with $\eps = \gamma^4$ and $t = 1/\gamma^5$ yields 
\[ \PP_{i+1}\big(  |N_{i+1}^{\circ}(x)\cap N_{i+1}^{\circ}(y)|  \geq (1+\eps) \mu  \big)  \leq  \bigg(\frac{(1+ n^{-\delta/4})^2 }{1+\eps}\bigg)^t \leq e^{-\eps t/2 } \leq n^{-3}\, . \]
Thus union bounding over all $x,y\in[n]$ and applying Lemma~\ref{lem:theta-conc} and~\eqref{eq:def-cE_i} as above completes the proof.
\end{proof}

Skipping over the (considerably easier) conditions on the neighbourhoods of $G'_{\leq i}$ for the moment, we now deal with the conditions on the intersections $N_{\leq i+1}' \cap N^{\circ}_{i+1}(x)$ at \eqref{eq:cross-degrees}. Again we proceed in a manner similar to the above. 

\begin{lemma}\label{lem:mixed-deg}
If $\cA_{i}$ holds, then
\[ \PP_{i+1}\big( |N'_{\leq i+1 }(x) \cap N^{\circ}_{i+1}(y)| \leq  (1+\eps_{i+1} )\theta_{i+1}\psi_{i+1} n \text{\,  for all\,  }x,y\big) \geq 1-4n^{-\delta/8}.\]
\end{lemma}
\begin{proof} For $x,y\in[n]$, we can write our set of interest as a disjoint union
\begin{align} \label{eq:cross-deg-du}
N'_{\leq i+1 }(x) \cap N^{\circ}_{i+1}(y) = \big( N'_{\leq {i} }(x) \cap N^{\circ}_{i+1}(y) \big) \cup \big(  N_{i+1}'(x) \cap N_{i+1}^{\circ}(y) \big).\end{align}
Here we use are using the notation $N'_{i+1}(x) = N_{G_{i}'}(x)$. One can show, using a similar argument to that of Lemmas~\ref{lem:cA-open-deg} and \ref{lem:cA-open-codegs} and using the assumption that~\eqref{eq:cross-degrees} holds, that
\begin{align}\label{eq:cross-deg1}
\PP_{i+1}\big( | N'_{\leq {i} }(x) \cap N^{\circ}_{i+1}(y) | \leq (1+ \eps_{i+1} )\psi_i\theta_{i+1}n \text{\,  for all\,  }x,y \big) \geq 1-2n^{-\delta/8}.
\end{align}

Thus we focus on the second set in the disjoint union. Here a vertex $z$ is in this set if $z \in N_i^{\circ}(x) \cap N_{i}^{\circ}(y)$ and $xz \in G_{i+1}'$ and $zy \in O_{i+1}$. Let 
\[ Z = Z_{x,y}(t) = \big| \big\{ S \in (N_i^{\circ}(x) \cap N_{i}^{\circ}(y))^{(t)} :  zx \in G_{i+1}' \text{ and } zy \in O_{i+1} \text{ for all } z \in S \big\} \big| .\]

Setting $\mu = s_{i+1}p_{i+1}|N_i^{\circ}(x) \cap N_{i}^{\circ}(y)|$, we have
\[ \EE_{i+1} Z \geq \PP\big(  | N'_{i+1}(x) \cap N_{i+1}^{\circ}(y) | \geq (1+\eps)\mu \big) \cdot \binom{(1+\eps)\mu }{t}\,. \]
We will use Lemma~\ref{lem:probability-open'} to bound $\E_{i+1} Z$.  Writing $K_{x,S}$ for the star from $x$ to a set $S$ and defining $K_{y,S}$ similarly, Lemma~\ref{lem:probability-open'} implies 
\begin{equation*} \EE_{i+1} Z = \sum_{S}\P_{i+1}\left( K_{x,S} \subset G_{i+1}',K_{y,S} \subset O_{i+1} \right)  \leq  (1+n^{-\delta/4})^t \binom{|N_i^{\circ}(x) \cap N_{i}^{\circ}(y)|}{t}p_{i+1}^ts_{i+1}^t  
\end{equation*} 
where the sum is over $S \in (N_i^{\circ}(x) \cap N_{i}^{\circ}(y))^{(t)}$. 
Thus, setting $\eps = \g^4$ and $t = 1/\g^5$,
\[ \PP\big(  | N'_{i+1}(x) \cap N_{i+1}^{\circ}(y) | \geq (1+\eps)\mu \big) \leq   \bigg(\frac{(1+ n^{-\delta/4}) }{1+\eps}\bigg)^t \leq e^{-\eps t/2 } \leq n^{-3} .\]
Applying Lemma~\ref{lem:theta-conc} and \eqref{eq:def-cE_i} as above, and recalling that $p_{i+1}\theta_i=\gamma/\sqrt{n}$, it follows that $(1+\eps)\mu \leq (1+\eps_{i+1})(\g/\sqrt{n})\theta_{i+1}n$ with probability at least $1-n^{-\delta/8}$. We conclude via a union  bound that 
\[ \PP\big(  | N'_{i+1}(x) \cap N_{i+1}^{\circ}(y) | \leq (1+\eps_{i+1})(\g/\sqrt{n})\theta_{i+1}n \text{\,  for all\,  }x,y\big) \geq 1-2n^{-\delta/8} .\]
The result follows by combining this with~\eqref{eq:cross-deg1}, \eqref{eq:cross-deg-du} and recalling that $\psi_{i}=i\g/\sqrt{n}+p_0$.
\end{proof}

We now easily deduce Lemma~\ref{lem:Ai-Aiplus1}.
\begin{proof}[Proof of Lemma~\ref{lem:Ai-Aiplus1}]
For $i\geq 0$ assume that $\cA_0 \wedge \cdots \wedge \cA_i$ holds. Lemma~\ref{lem:theta-conc} tells us that \eqref{eq:theta-conc-Ai} holds. Lemma~\ref{lem:cA-open-deg} and Lemma~\ref{lem:cA-open-codegs} take care of \eqref{eq:def-cE_i}. We take care of \eqref{eq:def-cF_i} and \eqref{eq:def-cF_i2} by an easy application of Chernoff's inequality. Finally \eqref{eq:cross-degrees} is taken care of by Lemma~\ref{lem:mixed-deg}.
\end{proof}

Thus we have shown that our process runs as expected up to time $T$. In the remainder of the proof we turn to the more complicated task of bounding the independence number.

\section{Vertices with large degree into \texorpdfstring{$I$}{I}}\label{sec:moderate-degrees}

We now begin a series of steps towards bounding the independence number of our graph. In this section we control the number of vertices that have large degree to a potential independent set $I \in [n]^{(k)}$. 
To be more precise, let $I \in [n]^{(k)}$ and define
\begin{equation} \label{eq:L-def}
L(I)=\big\{  x \in [n] :\, |I \cap N'_{ \leq T }(x) |\geq n^{2\beta}\big\}, \end{equation}
where $\beta>0$ is chosen so that $\beta<\delta/2^7$. For $i\leq T$, we define the event 
\begin{align}\label{eq:cC-def}
\cC = \big\{ e(G_{\leq T}'[L(I),I]) < n^\beta k  \text{ for all } I \in [n]^{(k)} \big\}\, . 
\end{align}
The main objective of this section is to prove the following. 

\begin{lemma}\label{lem:moderate-degrees}
We have
\[ \PP\big( \cC \big) = 1-o(1).\]
\end{lemma}

We prove this lemma by showing that if $G_{\leq T}[I \cup L(I)]$ contains too many edges then it contains a suitably dense ``fingerprint'' graph $H \subset G_{\leq T}[I \cup L(I)]$. We then union bound over all such $H$ and use our ``small subgraphs lemma,'' Lemma~\ref{lem:probability-open'}, to show that the appearance of such a fingerprint is unlikely. 

\subsection{Identification of a dense subgraph \texorpdfstring{$H$}{H}}
To choose $H$ we use the following simple deterministic lemma.  It will essentially allow us to consider only approximately regular graphs $H$. We note that a similar idea is used in \cite[Section 7]{FGM}. 

\begin{lemma}\label{lem:ind-subgraph}
 For $d\geq (\log n)^4$, let $G$ be a graph on $n$ vertices with average degree $d$. Then $G$ has an induced subgraph with average degree at least $\frac{c\log d}{\log n}\cdot d$ and maximum degree at most $2d^{3/2}$, where $c>0$ is an absolute constant. 
\end{lemma}
\begin{proof}
Choose $G'$ to be an induced subgraph of $G$ which is minimal, subject to having average degree $\geq d/2$. Let $V = V(G')$, let $D\geq d/2$ be the average degree of $G'$ and set $R = D^{1/4}$. We note that by the minimality of $G'$ we have $D\leq d/2+1$.

We define a partition of $V$ by
\[ V_0 = \big\{ x \in V : \, d_{G'}(x) \leq RD \big\}  \quad \text{ and } \quad    V_i = \big\{ x \in V :  R^iD < d_{G'}(x) \leq R^{i+1}D \big\}, \]
for $i= 1, \ldots, \ell := (\log n)/(\log R)$. We also write $V_{<i} = V_1 \cup \cdots \cup V_{i-1}$. 

We focus on $G'$ across a bipartition $(V_{<j},V_j)$. For this, note there exists $j$ such that\footnote{For a graph $G$ and $A,B\subset V(G)$ we let $e_G(A,B)$ denote $e(G[A,B])$.}  
\begin{align}\label{eq:avg-Vj}
e_G(V_{<j},V_j) + e(G[V_j])  \geq e(G')/\ell, 
\end{align}
by averaging.
Now if $j=0$, $G[V_0]$ is the graph we are looking for. Indeed, 
\[ 2e(G[V_0])/|V_0| \geq 2e(G')/(\ell|V|) \geq D/\ell \quad \text{ and } \quad \Delta(G[V_0]) \leq RD \leq D^{3/2}.\] 
So assume $j > 0$. First note that $|V_j| \leq |V|/R^{j}$ by Markov's inequality and therefore \[ e(G[V_j]) \leq d|V_j|/4 \leq D|V|/(2R^j) \leq e(G')/(2\ell), \] 
where for the first inequality follows from the minimality of $G'$ and the last inequality uses that $R \geq 2\ell$ since $D\geq d/2\geq (\log n)^4/2$. Thus $e_G(V_{<j},V_j) \geq e(G')/(2\ell)$ by~\eqref{eq:avg-Vj}.

We now remove all the vertices that have high degree inside of $V_j$. Define 
\[ V'_j = \big\{x \in V_j : |N_{G}(x) \cap V_j |\leq R^2D \big\} .\] We have to ensure we don't lose too many edges when passing to $V'_j$. So note, by definition of $V'_j$ and minimality of $G'$ that we have 
\[ (R^2D/2)|V_j\setminus V'_j| \leq e(G[V_j]) \leq d|V_j|/4 \quad  \text{ and thus } \quad  |V_j\setminus V'_j| \leq |V_j|/R^2 .\]
Thus the number of edges incident to $V_j\setminus V_j'$ in $G'$ is at most 
\[ R^{j+1}D |V_j\setminus V_j'| \leq R^{j-1}D |V_j| \leq D|V|/R\] 
and so 
\[
e_G(V_{<j},V_j')\geq e_G(V_{<j},V_j)-D|V|/R \geq e(G')/(4\ell)
\]
where for the final inequality we recall that $e_G(V_{<j},V_j)\geq e(G')/(2\ell)$ and that $R\geq \log n\geq  8\ell$ since $D\geq (\log n)^4$.

Now define $V_{<j}' \subset V_{<j}$ to be a $p$-random subset, where $p = R^{-j}$. Then define our final graph to be $G[S]$ where
\[   S = (V'_{<j} \cup V'_{j})\setminus L ,\quad
\text{ and } \quad L = \big\{ x \in V'_{<j} \cup V'_{j} : |(V'_{<j} \cup V'_{j}) \cap N_G(x)| \geq 2R^2 D \big\}.\] Note $\Delta(G[S]) \leq 2R^2D = 2 D^{3/2}\leq 2d^{3/2}$ by construction; so it is enough to show that \begin{equation}\label{eq:expecation} \EE\, e(G[S]) - (cD/\ell)  \EE |S|  \geq 0 ,\end{equation}
for some absolute constant $c>0$. 
Now note that 
\[ e(G[S]) \geq e_G(V'_{<j},V'_j) - R^{j}D|L\cap V_{<j}| - R^{j+1}D|L\cap V_j|,   \quad \EE\, |S| \leq \EE\, |V'_{<j}| + |V'_{j}| \leq 2p|V|, \] and that 
\[ \EE\, e_G(V'_{<j},V'_j) = pe_G(V_{<j},V'_j) \geq pe(G')/(4\ell) = pD|V|/(8\ell) . \] We now bound $\EE |L|$. Note $x \in L$ only if $|N_G(x) \cap V'_{<j}| \geq R^2D $ and that $|N_G(x) \cap V'_{<j}|$ is a binomial random variable with mean $\EE |N_G(x) \cap V'_{<j}| \leq pR^{j+1}D = RD $. By the Chernoff bound, we have that $\PP( x \in L ) \leq p \exp(-RD)$. Thus 
\[ R^{j}D \cdot \EE |L \cap V_{<j}| \leq D|V|\exp(-RD) \quad \text{ and } \quad R^{j+1}D \cdot \EE |L \cap V_{<j}| \leq RD|V|\exp(-RD). \]
Noting that $R\exp(-RD) \leq p/(32\ell)$ since $D\geq (\log n)^4/2$ we conclude that $\E e(G[S])\geq p D|V|/(16 \ell)$. It follows that~\eqref{eq:expecation} holds with $c=1/32$. Therefore there exists $S$ with $e(G[S])-(cD/\ell)|S|\geq 0$, as desired.
\end{proof}

\subsection{The probability of containing a subgraph \texorpdfstring{$H$}{H}}

We now prove the following lemma about the probability a fixed subgraph is contained in our graph process. In fact this lemma is easy to prove with our Lemma~\ref{lem:probability-open'}. 

\begin{lemma}\label{lem:prob-subgraph}
Let $H\subset K_n$ satisfy $\Delta(H) \leq n^{\delta/4}$ and suppose that $\pi(x)\neq \pi(y)$ for all distinct $x,y\in V(H)$.
If we write $V(H) = V$ then 
$$\PP \big( H\subset G'_{\leq T} ~ \wedge ~ \Delta(G'_{\leq T}[V] \big)\leq n^{\delta/4} ~\wedge~ \cA \big)\leq n^{-(1+o(1))e(H)/2} \, .$$
\end{lemma}

\begin{proof}
Define the event $\cE = \{ \Delta(G'_{\leq T}[V] \big)\leq n^{\delta/4}\} \wedge \cA$. We union bound over all partitions of the edges $H = H_0 \cup \cdots \cup H_T$ of $H$, noting there are only $(T+1)^{e(H)} = n^{o(e(H))}$ many such partitions.  We bound
\[ \PP \big( H\subset G'_{\leq T} \wedge \cE \big) \leq \sum_{H_{0},\ldots,H_T}  \PP\big(\, \forall\, 0 \leq i \leq T, \,  H_i \subset G'_{i} \, \wedge \, \cE \big) .\]
Now, if we let $H_{>i} = H_{i+1} \cup \cdots \cup H_T$, we bound $\PP\big( \forall i, \,  H_i \subset G'_{ i}  \wedge \cE \big) $ above by
\begin{equation}\label{eq:prob-subgraph-1}  \PP \big( H_0 \subset G'_0 \wedge H_{>0} \subset O_0 \big) \prod_{0 \leq i\leq T-1 } \max
\, \PP_{i+1}\big( H_{i+1} \subset G_{i+1}'  \wedge H_{>i+1} \subset O_{i+1} \wedge   \cE \big) ,\end{equation}
where the maximum is over all outcomes in steps $0,1,\ldots,i$ satisfying $\cA_0 \wedge \cdots \wedge \cA_i$. 

We bound each term in the product \eqref{eq:prob-subgraph-1} individually. For the zeroth term we use that 
\[ \PP\big( H_0 \subset G'_0 \wedge H_{>0} \subset O_0 \big) \leq p_0^{e(H_0)}(1/r^2)^{e(H_{>0})} = ((1+o(1))\theta_0)^{e(H_{>0})}n^{-(1+o(1))e(H_0)/2}, \]
where the first inequality holds by the property that each vertex of $H$ is in a different part of the blowup partition by assumption and thus the probability each edge is included in the respective graph is independent. The second inequality holds since $\theta_0 = (1-o(1))r^{-2}$ on the event $\cA$. 

For all the other terms in the product \eqref{eq:prob-subgraph-1}, we apply Lemma~\ref{lem:probability-open'} to obtain
\[\PP_{i+1}\big( H_{i+1} \subset G_{i+1}'  \wedge H_{>i+1} \subset O_{i+1} \wedge   \cE \big) \leq \big(1+n^{-\delta/4}\big)^{e(H_{i+1})}s_{i+1}^{e(H_{>i+1})}p_{i+1}^{e(H_{i+1})}. \]
Multiplying these together and using Observation~\ref{obs:thetai-is-prod-si} gives
\[ \theta_i=(1+o(1))s_i\cdots s_1\theta_0\, ,
\]
on the event $\cA$ and that
$p_{i+1}\theta_i = \g/n^{1/2}$, gives the desired result. 
\end{proof}

\subsection{Proof of Lemma~\ref{lem:moderate-degrees}}

We now prove Lemma~\ref{lem:moderate-degrees} by applying Lemma~\ref{lem:ind-subgraph} and Lemma~\ref{lem:prob-subgraph}.

\begin{proof}[Proof of Lemma~\ref{lem:moderate-degrees}] 
It is enough to show $\PP( \cC^c \cap \cA ) = o(1) $, by Lemma~\ref{lem:cA}. For this, let $I\in [n]^{(k)}$ and let $G' \in \cA$ be an instance of $G'_{\leq T}$ for which $e_{G'}(L(I),I) \geq n^\beta k$. Our aim is to find a subgraph $H(I) \subset G'[ L(I)\cup I]$ whose appearance is unlikely.

{Let $J\subseteq L(I) \cup I$ be a subset of  minimal size such that $e(G'[J]) \geq n^\beta |J|/2$. First we justify why such a $J$ exists. Indeed if $|L(I)|\leq k$, then $e(G'[ L(I)\cup I])\geq e_{G'}(L(I),I) \geq n^\beta k\geq n^\beta| L(I)\cup I|/2$. If instead $|L(I)|>k$, then let $L'\subseteq L(I)$ where $|L'|=k$ and note that each vertex of $L'$ has degree at least $n^{2\beta}$ in $G'[L'\cup I]$ (by the definition of $L(I)$) and so $e(G'[L'\cup I])\geq |L'|n^{2\beta}/2 \geq n^\beta |L' \cup I|/2$. We may therefore select $J$ as above and moreover $|J|\leq 2k$. Note that by the minimality of $J$ we have $e(G'[J])\leq n^\beta |J|/2+|J|\leq n^\beta |J|$. 

Now let $J_\ast \subset J$ be maximal such that $J$ intersects each part of the blowup partition in at most one vertex. Note that 
 \[
 r^{-2}n^{\beta}|J|/2\leq e(G'[J_\ast]) \leq n^{\beta}|J|
 \]
 and $|J|/r\leq |J_\ast|\leq |J|$ and so the average degree of $G'[J_\ast]$ is $\tilde\Theta(n^{\beta})$.}

 Now apply Lemma~\ref{lem:ind-subgraph} to the graph $G'[J_\ast]$, to find  an  induced subgraph $H=H(I)$ with vertex set $V\subset [n]$ such that
\begin{equation}\label{eq:H'-props} |V| \leq 2k, \qquad    e(H) =  \tilde\Omega(n^{\beta})|V|, \quad \text{ and } \quad \Delta(H) = \tilde O(n^{3\beta/2})\, . \end{equation} 
 By Lemma~\ref{lem:prob-subgraph}, recalling that $\beta<\delta/2^7$, we have 
\[ \PP\big( H = G_{\leq T}'[V] \big) \leq n^{-(1+o(1))e(H)/2}. \]
We now simply union bound over all such $H(I) \subset K_n$ satisfying the conditions at \eqref{eq:H'-props}. If we write $h = |V(H)|$ and $m = e(H)$, we see there are $\binom{n}{h}$ ways of choosing $V(H) \subset [n]$ and at most $\binom{h^2}{m}$
ways of choosing $H$ on this vertex set. Putting this together gives
\[ \PP( \cC^c \cap \cA ) \leq \sum_{h,m} n^h
\bigg( \frac{eh^2}{m}\bigg)^m n^{-(1+o(1))m/2} \leq  \sum_{h,m} n^{(-\beta/2 +o(1))m} = o(1), \]
since $h=\tilde O(n^{-\beta}m)$, $m=\tilde O(kn^{3\beta/2})= \tilde O(n^{1/2+3\beta/2})$ and so $h^2/m= \tilde O(mn^{-2\beta})= \tilde O(n^{1/2-\beta/2})$.
\end{proof}

\section{Open edges in \texorpdfstring{$I$}{I}: defining and derandomizing the core }\label{sec:open-edges-in-I-redux}
One way for a set $I \in [n]^{(k)}$ to be at risk of becoming independent is if $I^{(2)} \cap O_i$ becomes much smaller than it ``should be''. In the next two sections, we show that it is extremely unlikely for there to be such an $I$. The challenge here is that it is not obvious how large $I^{(2)} \cap O_i$ ``should'' be. Indeed if $I$ has large intersection with a neighbourhood $N_{\leq i}(x)$, all of the pairs $(N_{\leq i}(x))^{(2)}$ must be closed by the triangle-free property. Thus to correctly account for the number of open pairs that we expect in $I$, it makes sense to exclude the pairs that will be covered by some vertex whose neighbourhood has large intersection with $I$. 

This motivates the definition of the \emph{core} of $I$. If we recall the definition of $L(I)$ from \eqref{eq:L-def}, as the set of vertices $x$ with $|I \cap N_{\leq T}'(x)| \geq n^{2\beta}$, we define the \emph{core} of $I$, at stage $i$, to be set of pairs in $I^{(2)}$ that are not covered by the neighbourhoods of vertices in $L(I)$. That is,
 \begin{align}\label{eq:core-def}
 C_{i}(I) = I^{(2)} \setminus \bigcup_{ x \in L(I) } N_{\leq i}'(x)^{(2)}. 
 \end{align}
 Now define $\cB_i(I)$ to be the event that the number of open edges in the core is large,
\begin{equation}\label{eq:def-B_i} \cB_{i}(I) = \big\{ |C_i(I) \cap O_{i} | \geq (1-\eps_i) \theta_{i} |C_{i}(I)| \big\} . \end{equation}

One wrinkle in working with the event $\cB_i(I)$, is that we will not be able to ensure that $\cB_i(I)$ holds if $|C_i(I)|$ becomes very small, in particular if $|C_i(I)| \leq \gamma^2 k^{2}$, which is an event that we deal with in the context of its application (see Lemma~\ref{lem:prob-ind}).  Thus it makes sense to define the event
\[ \cD_i(I) = \big\{ |C_i(I)|\leq \gamma^2 k^{2} \big\},\]
so that we can work on $\cD_i^c$ in what follows. 
We also require that $|\pi(C_0(I) \cap O_0)|$ is not small and so we also define the analogous event $$\cB_0^\pi(I) = \big\{|\pi(C_0(I) \cap O_0)| \geq (1 - \eps_0) r^2 \theta_0 |\pi(C_0(I))| \big\}\,, $$
for the pre-blown up graph in the seed step. Now define the event
\[ \cF =  \bigwedge_{I \in [n]^{(k)}}  \bigwedge_{0 \leq i\leq T} \big( \cB_i(I) \vee \cD_i(I)\big) \wedge \cB_0^\pi(I)\,. \]

The main objective of this and the next section is to prove the following.
\begin{lemma}\label{lem:cB-failure} We have
\begin{equation}\label{eq:cB-failure} \PP\big( \cF \big)  = 1-o(1).\end{equation}
\end{lemma}

One challenge in proving Lemma \ref{lem:cB-failure} is that the definition of $C_i(I)$ depends on steps in the process later than $i$; in particular, $L(I)$ depends on the final graph $G_{\leq T}$. The purpose of this section is to reduce Lemma \ref{lem:cB-failure} to two statements that only depend on a fixed step of the process.  To do this we show that can remove this ``dependence on the future'' by union bounding over all possible ``candidates'' for $C_i(I)$ and $L(I)$.  

\subsection{Derandomizing \texorpdfstring{$C_i(I)$}{C(I)}}
For $I \in [n]^{(k)}$, note that on $\cC$ we have $|L(I)|\leq 2kn^{-\beta}$, since  
\begin{equation} \label{eq:use-of-C} G_{\leq T}' \in \cC \qquad \Longrightarrow \qquad  |L(I)|n^{2\beta}/2 \leq e(G_{\leq T}[I,L(I)]) \leq kn^{\beta}. \end{equation}
Thus we will union bound over all possible sets $L = L(I)$ for which $|L|\leq 2kn^{-\beta}$. We now define the collection of ``potential cores'' that we will union bound over;
\[ \Lambda(I) = \bigg\{ I^{(2)} - \bigcup_{i=1}^{\ell} S_i^{(2)} :  S_i \subset I,\, |S_i|\geq n^{2\beta} \,\text{ and } \, \sum_i |S_i| \leq 2kn^{\beta} \bigg\}. \] 
Using \eqref{eq:use-of-C}  we note  that $C_i(I) \in \Lambda(I)$, on the event $\cC$.

We now define two events that ensure that $L$ plays well with $C$ and $I$. Indeed, for $C \subset I^{(2)}$, and $L \subset [n]$ we define the event 
\begin{equation}\label{eq:def-T(C,I)} \cT_{i}(L,C) = \big\{  (N'_{\leq i}(x))^{(2)} \cap C = \emptyset \, , \, \forall x \in L \big\}.\end{equation}
Informally this says that $C$ ``looks like a core of $I$ with respect to the set $L$, at step $i$.'' Indeed, $\cT_{i}(L(I),C_{i}(I))$ always holds by definition. We also define the event
\[ \hspace{1.5em} \cL_{i}(I,L) = \big\{  \, | I \cap N'_{\leq i}(x)| \leq n^{2\beta} \, , \, \forall x \in [n]\setminus L  \big\}\] 
and note that $\cL_{i}(I,L(I))$ always holds. 

We are now ready to state our two main lemmas that (as we will see) imply Lemma \ref{lem:cB-failure}. Each of these statements give good bounds on the probability that the number of open edges in a set $C \in \Lambda(I)$ decreases more than is expected in a fixed round $i$.

\begin{lemma}\label{lem:single-step-B-failure} For $0 \leq i< T$ and $I\in [n]^{(k)}$, let $C \in \Lambda(I)$, $|C| \geq \gamma^2 k^{2}$, and $L \subset [n]$. If $\cA_i$ holds and $|C\cap O_i| \geq \theta_i \gamma^2 k^{2} $ then 
\begin{equation}\label{eq:cB-proof-single-step-statement} \PP_{i+1}\big( |C \cap O_{i+1}| < (1-\gamma^4)s_{i+1}|C \cap O_{i}| \wedge \cT_{i}(L,C)   \wedge \cL_{i}(I,L) \big) \leq \exp(-k^{1+\delta/4}).\end{equation}
\end{lemma}

The following lemma is the corresponding result for the ``seed step.'' While this is a simpler statement only about the binomial random graph, we deduce it from the same techniques we used to prove Lemma~\ref{lem:single-step-B-failure}, with a few minor adjustments.

\begin{lemma}\label{lem:single-step-B-failure-0} Let $I \in [n]^{(k)}$, $L \subset [n]$ and let $C \in \Lambda(I)$ satisfy $|C| \geq \gamma^2 k^2$. We have that
\begin{equation}\label{eq:single-step-B-failure-0}\PP\big(  |C \cap O_{0}| \leq (1-\gamma^4) \theta_0 |C| \, \wedge \cT_{0}(L,C) \wedge \cL_{0}(I,L) \big) \leq \exp( - k^{3/2 } ),\end{equation}
and 
\begin{equation}\label{eq:single-step-B-failure-01}
\PP\big(  |\pi(C \cap O_{0})| \leq (1-\gamma^4) r^2\theta_0 |\pi(C)| \, \wedge \cT_{0}(L,C) \wedge \cL_{0}(I,L) \big) \leq \exp( - k^{3/2 } ).
\end{equation}
\end{lemma}
Recall that for a graph $H \subset K_n$, we defined \eqref{eq:def-pi} the map $\pi(H) = \{ \pi(x)\pi(y) : xy \in H \}$, where $\pi : [n] \rightarrow [n/r]$ is the map defined by $\pi(x) = j$, whenever $x \in V_j$.

\subsection{Proof of Lemma \ref{lem:cB-failure}}
Now to prove Lemma~\ref{lem:cB-failure}, assuming Lemma~\ref{lem:single-step-B-failure} and Lemma~\ref{lem:single-step-B-failure-0}, we first note that there are not many core-like sets; and thus our union bound is not too large. 
  
    \begin{observation}
        \label{obs:bound-on-Lambda}
    We have $|\Lambda(I)| \leq 2^{kn^{2\beta}} $.
    \end{observation}
    \begin{proof} We bound $|\Lambda(I)|$ by counting the number of choices of tuples $\{S_1,\ldots,S_{\ell}\}$ that satisfy the conditions in the definition. Note that by the condition $|S_1| + \cdots + |S_{\ell}| \leq 2kn^{\beta}$, we have 
    \[ \big|\big\{ i : 2^{j}n^{2\beta} \leq |S_i| \leq 2^{j+1}n^{2\beta} \big\}\big| \leq 2k/(2^jn^{\beta}), \]
    for all $0\leq j \leq \log n$. Thus the number of choices is at most
    \[
            \prod_{j = 0}^{\log n} \binom{k}{2^{j+1}n^{2\beta}}^{2kn^{-\beta} 2^{-j}} \leq \exp(4 \log k \cdot \log n \cdot k n^{\beta}) \leq \exp\left(k n^{2\beta}\right) \,.\qedhere \]\end{proof}
     
\noindent We now prove Lemma~\ref{lem:cB-failure}. 

\begin{proof}[Proof of Lemma \ref{lem:cB-failure} (assuming Lemmas~\ref{lem:single-step-B-failure} and \ref{lem:single-step-B-failure-0})]
    We note that it is enough to show the two statements 
    \begin{equation}\label{eq:cBfailure-ub1} \PP\big( |C_{i}(I) \cap O_i| < (1-\eps_i)\theta_i|C_i(I)|\, \wedge\, |C_{i}(I)| \geq \gamma^2 k^2 \, \wedge \, \cA \wedge  \cC \big)  \leq e^{-kn^{3\beta}} , \end{equation}
    for all $0 \leq i < T$ and all $I \in [n]^{(k)}$; \emph{ and }  \begin{equation} \label{eq:cBpi-failure-ub}
        \PP\big( |\pi(C_{0}(I) \cap O_0)| < (1- \eps_0) r^2 \theta_0 |\pi(C_0(I))|\, \wedge\, |C_{0}(I)| \geq \gamma^2 k^2 \, \wedge\,  \cA \wedge \cC \big)  \leq e^{-kn^{3\beta}} \,,
    \end{equation} for all $I \in [n]^{(k)}$. This is because if \eqref{eq:cBfailure-ub1} and \eqref{eq:cBpi-failure-ub} hold, we can union bound over all $I \in [n]^{(k)}$, $0\leq i< T$, to obtain \eqref{eq:cB-failure}.  
    
    We first prove \eqref{eq:cBfailure-ub1}, which is the more complicated of the two, and then note that \eqref{eq:cBpi-failure-ub} is similar. Using \eqref{eq:use-of-C}, the probability on the left hand side at \eqref{eq:cBfailure-ub1} is 
    \begin{equation}
    \leq \sum_{L} \, \sum_{C}\, \PP\big( |C \cap O_i| < (1-\eps_i)\theta_i|C|  \wedge \cT_{i}(L,C)   \wedge \cL_{i}(I,L) \wedge \cA \big),  \label{eq:cBfailure-ub2} 
    \end{equation}
    where the first sum is over all $L$ with $|L| \leq 2kn^{-\beta}$ and the second sum is over all $C \in \Lambda(I)$ with $|C| \geq \gamma^2 k^2$.
    We now need to check that this union bound is not too large. There are at most $\binom{n}{kn^{-\beta}} \leq 2^{k}$ such choices for $L$, which is small relative to our final bound. Observation~\ref{obs:bound-on-Lambda} tells us that $|\Lambda(I)|$ is relatively small. Thus we may bound \eqref{eq:cBfailure-ub2} above by \begin{equation}\label{eq:cBfailure-ub3} \leq 2^{kn^{2\beta + o(1)}} \max_{L,C}\, \PP\big( |C \cap O_i| < (1-\eps_i)\theta_i|C|  \wedge \cT_{i}(L,C)   \wedge \cL_{i}(I,L) \wedge \cA \big),\end{equation} where the maximum is over $L$ with $|L| \leq k n^{-\beta}$ and $ C \in \Lambda(I)$ with $|C| \geq \gamma^2 k^2$.
    
    We now reduce to a single step of the process.  
    Indeed note that that if $|C \cap O_i| < (1-\eps_i)\theta_i|C|$
    then there exists $j$ with $1\leq j \leq i$ for which 
    \begin{equation}\label{eq:single-step-split} |C \cap O_j| < (1-\gamma^4) s_j |C \cap O_{j-1}| \qquad \text{ or } \qquad  |C \cap O_0| \leq (1- \gamma^4) \theta_0 |C|\, , \end{equation}
    where we recall that $\theta_i\geq (1-O(i\g^4))s_i\cdots s_1\theta_0$ on the event $\cA$, by Observation~\ref{obs:thetai-is-prod-si}.
    So we define 
    \[ A_{j+1} =  \max_{L,C}\,\PP_{j+1}\big( |C \cap O_{j+1}| < (1-\gamma^4)s_{j+1}|C \cap O_{j}| \wedge \cT_{j}(L,C)   \wedge \cL_{j}(I,L) \big), \] for $0\leq j< T$ and define
    \[ B = \max_{L, C}\, \PP\big( |C \cap O_{0}| < (1-\gamma^4)\theta_0|C| \wedge \cT_{0}(L,C)   \wedge \cL_{0}(I,L) \big), \] where the two maximums are over $L\subset [n]$, $|L| \leq kn^{-\beta}$ and $C \in \Lambda(I)$ with $|C|\geq \g^2 k^2$. Union bounding over $j$ for which \eqref{eq:single-step-split} holds, tells us that the right-hand-side of \eqref{eq:cBfailure-ub3} is 
    \[ \leq  2^{kn^{2\beta+o(1)}}( A_1 + \cdots + A_{T}  +B )  \leq \exp(-k^{1+\delta/4 +o(1)} ),\]
    where the inequality holds by Lemmas~\ref{lem:single-step-B-failure} and \ref{lem:single-step-B-failure-0} and since $\beta<\delta/2^7$. This proves \eqref{eq:cBfailure-ub1}. 

    To prove \eqref{eq:cBpi-failure-ub}, we again union bound over $L$ and $C$ and use \eqref{eq:single-step-B-failure-01} in  Lemma~\ref{lem:single-step-B-failure-0} to complete the proof.
\end{proof}

\section{Open edges in \texorpdfstring{$I$}{I}: Martingale concentration}\label{sec:open-edges-in-I-martingale}
In this section we prove 
Lemma~\ref{lem:single-step-B-failure} and Lemma~\ref{lem:single-step-B-failure-0}. For this, we first compare the iid measure on edges with an associated ``capped'' measure $\widehat{\PP}$, which ``caps'' the degrees of vertices into $I$. This allows us control the increments in an associated martingale which tracks the number of open edges in the core. This, in turn, allows us to control the large deviations of the number of open edges in the capped measure. This is the main technical use of the definition of the core: since we have removed the contribution of large degree---and safely controlled their effect in Section \ref{sec:moderate-degrees}---the increments of our associated martingale will now be sufficiently controlled to obtain the concentration we require.

\subsection{The capped measure \texorpdfstring{$\widehat{\PP}$}{P}} Given $I,L \subset [n]$, define the ``capped'' measure $\widehat{\PP} = \widehat{\PP}_{i,I,L} $ as follows. First randomly order the edges of $O_i$. Then, for each vertex $x \in [n] \setminus (L \cup I)$, explore pairs in $O_i$ incident with $x$, in order, including each into a graph $\widehat{G}_{i+1}$ independently with probability $p_{i+1}$ \emph{unless} the current pair is between $x$ and $I$ and the degree from $x$ into $I$ is $\geq  n^{2\beta} $, in which case we skip over this pair, not including it into $\widehat{G}_{i+1}$. Finally, we then run the regularization step as in the usual process: recall that we defined $Q_{i+1} \subset O_{i}$ to be a random subset where each pair $e 
\in O_i$ is independently included with probability $q_e$ (where $q_e$ is defined at~\eqref{eq:qedef}). We define 
\begin{align}\label{label:hat-O-def}
\widehat O_{i+1} = Q_{i+1} \cap \big\{ e \in K_n \setminus (G'_{\leq i} \cup \widehat G_{i+1})  : e \text{ does not form a triangle with }G'_{\leq i} \cup \widehat G_{i+1}\big\}. \end{align}

The point of defining this measure is we can control the lower tails of $|C\cap O_{i+1}|$,
on the event $\cT_{i+1}(L,C)  \wedge \cL_{i+1}(L,C)$, using $\widehat{\PP}$. In particular we show the following. 

\begin{lemma}\label{lem:P-P-tilde-domination} Let $I \in [n]^{(k)}$, $L \subset [n]$ and $C \subset I^{(2)}$. For all $t\geq 0$, 
\begin{equation}\label{eq:coupling} \PP_{i+1}\big(  |C \cap O_{i+1}| \leq t \, \wedge \cT_{i+1}(L,C) \wedge \cL_{i+1}(I,L) \big) \leq \widehat{\PP}_{I,L}\big(  |C \cap \widehat{O}_{i+1}| \leq t + k^{1+8\beta} \big)\, . \end{equation}
\end{lemma}

Before we prove this a few remarks are in order. First note the measure $\widehat{\PP}$ does not expose any  edges between $L$ and $I$ or edges inside of $I$. Informally this is fine for us, since edges between $L$ and $I$ cannot witness the closure of any edges inside of the core, precisely by the definition of the core. Moreover, vertices that are in $I \setminus L(I)$ have low degree into $I$ and therefore cannot close many edges, deterministically. Indeed this accounts for the $k^{1+8\beta}$ term in the above.

For the remainder of this section, we fix $I,L$ and $C$ as in Lemma~\ref{lem:P-P-tilde-domination} and we let $\widehat{\EE}$ denote expectation with respect to the measure $\widehat{\P}$.

It is important that we can control the expectation of the one measure in terms of the other.

\begin{observation}\label{obs:P-hatP-means} 
We have
\[\widehat{\EE}\, |C \cap \widehat{O}_{i+1}| \geq \EE_{i+1} \, |C \cap O_{i+1} | = s_{i+1}|C \cap O_i| .\]
\end{observation}

We prove Lemma~\ref{lem:P-P-tilde-domination} and Observation~\ref{obs:P-hatP-means} using a natural coupling between the two measures. To describe this, first note that we can sample the iid measure by simply randomly ordering the edges of $O_i$ and then including each with probability $p_{i+1}$ into $G'_{i+1}$.

Thus, to couple the measures, simply use the same random order and the same choices on each edge, unless the degree constraint in the definition of $\widehat{G}_{i+1}$ is violated (in which case the edge is included in $G'_{i+1}$ but not in $\widehat{G}_{i+1}$). We also use the same choice of $Q_{i+1}$.  Thus we obtain a probability measure $\widetilde \PP$ on quadruples $(G'_{i+1},O_{i+1},\widehat{G}_{i+1},\widehat{O}_{i+1})$ whose marginals $(G'_{i+1},O_{i+1})$ and $(\widehat{G}_{i+1},\widehat{O}_{i+1})$ are $\PP_{i+1}$ and $\widehat{\PP}$, respectively. Moreover we have that
\begin{equation}\label{eq:P-Phat-dominance} \widehat{G}_{i+1} \subset G'_{i+1}\qquad \text{ and therefore } \qquad \widehat{O}_{i+1} \supset O_{i+1}, \end{equation}
which immediately implies Observation~\ref{obs:P-hatP-means}.
We are also in a position to prove Lemma~\ref{lem:P-P-tilde-domination}.

\begin{proof}[Proof of Lemma~\ref{lem:P-P-tilde-domination}]
We sample $(G'_{i+1},O_{i+1},\widehat{G}_{i+1},\widehat{O}_{i+1})$  according to the coupling $\widetilde  \P$ described above. We note that if $(G'_{i+1},O_{i+1})$ satisfies the event $\cL_{i+1}(I,L)$ then each vertex of $[n]\setminus ( I \cup L)$ has $\leq n^{2\beta}$ neighbours in $I$. Therefore the graphs $G'_{i+1},\widehat{G}_{i+1}$ are identical apart from possibly between $L$ and $I$
and within $I$. However no pair $e\in C$ can be closed by a vertex of $L$ without violating $\cT(L,C)$, thus we can disregard the edges between $L$ and $I$ in $G'_{i+1}$. Now note that each vertex $x \in I \setminus L$ can close at most $n^{4\beta}$ edges in $C$. Since there are at most $k$ such vertices we have 
\[ |C \cap \widehat{O}_{i+1}| \leq  |C \cap O_{i+1}| + k^{1+8\beta}.\]

Thus the indicator of the event on the left hand side of \eqref{eq:coupling} is at most the indicator of the event on the right hand side of \eqref{eq:coupling}. Applying the expectation $\widetilde{\EE}$ (i.e. with respect to the measure $\widetilde{\P}$) to both indicators completes the proof. 
\end{proof}

\subsection{Definition of the martingale} We can now prove the following concentration result about this capped measure $\widehat{\PP}$ using an appropriate martingale. It is worth mentioning that our martingale differs significantly from those in the works of Bohman~\cite{bohman2009triangle}, Bohman Keevash~\cite{bohman2010early,bohman2021dynamic} and Fiz Pontiveros, Griffiths,  Morris~\cite{FGM}.  

\begin{lemma}\label{lem:martinglae-application} If $\eps >0 $ and $\cA_i$ holds then 
\[ \widehat{\PP}\big(  |C \cap \widehat{O}_{i+1}| \leq (1-\eps) s_{i+1}|C \cap O_{i}|  \big) \leq \exp\big( - \eps^2|C \cap O_i | n^{-5\beta}  \big)\,.\]
\end{lemma}

We now define our martingale. Enumerate the vertices $[n] \setminus (L\cup I)=\{v_1,v_2,\ldots, v_\ell\}$. Define $\cF_j$ to be the 
$\sigma$-algebra defined by the \emph{relative} ordering of all of the edges between $\{v_1,\ldots,v_j\}$ and $I$, along with the edges between $\{v_1,\ldots,v_j\}$ and $I$, that is $\widehat{G}_{i+1}[\{v_1,\ldots,v_j\},I]$. Let $Z =  |C \cap \widehat{O}_{i+1}| $ and define the martingale $M_j$ with respect to this filtration by
\begin{equation}\label{eq:def-martingale} M_j = \widehat{\EE}[ Z \, | \cF_j ] \qquad \text{ which has increments }\qquad  \xi_j = M_j - M_{j-1}. \end{equation}
Note that 
\[ \xi_j = \sum_{e \in C\cap O_i } \widehat{\PP}\big( e \in \widehat O_{i+1}\, |\, \cF_j \big) - \widehat{\PP}\big( e \in \widehat O_{i+1} \, | \, \cF_{j-1} \big)  = \sum_{e \in C\cap O_i } I_e .\]
We now make a few simple observations about these increments which will feed naturally into Freedman's inequality (see Theorem~\ref{thm:Freedman}), a standard inequality for the concentration of martingales.  

For what follows, it is natural to slightly extend the terminology from \eqref{eq:def-Y}, \eqref{eq:def-X} of open and clopen triangles. We say an edge $e$ and a vertex $v$ form a clopen triangle if $e,v$ defines a triangle with two edges in $O_i$ and one in $G'_{\leq i}$. Say that $e,v$ form an open triangle if they define a triangle with all edges in $O_i$. 

For an open pair $e = xy\in O_i$, where $x\in I$, $y \in [n] \setminus (L \cup I)$, we define the quantity
\begin{equation}\label{eq:def-p(e)} p(e) = \widehat{\PP}\big( e \in \widehat{G}_{i+1} \big) \quad \text{ and note that } \quad p(e) \leq \min\big\{ p_{i+1} , n^{2\beta}/{|N^{\circ}_{i}(y) \cap I|}\}. \end{equation}
The first part of the inequality holds since $\widehat{G}_{i+1} \subset G_{i+1}'$. The second holds since each open edge {incident to $y$} is included with equal probability and  the maximum degree of $y$ is at most $n^{2\beta}$.

To understand the increments we note the following. 

\begin{observation}\label{obs:P(einO)-movement-0} Let $e = xy\in C \cap O_i$ be such that $e,v_j$ forms a clopen triangle with $xv_j \in O_i$ then 
\[ \widehat{\PP}\big( e \in \widehat{O}_{i+1}\, |\, \cF_j \big) = \begin{cases}  (1-p(xv_j))^{-1} \PP\big( e \in \widehat{O}_{i+1} | \cF_{j-1} \big) \qquad \text{ if  } xv_j \notin  \widehat{G}_{i+1} \\   0 \hspace{15em} \text{ if } xv_j \in  \widehat{G}_{i+1}.\end{cases}\]
\end{observation}
\begin{proof}
If $xv_j \in  \widehat{G}_{i+1}$ then clearly $e \not\in \widehat{O}_{i+1}$. Now let $\cE$ be the event $xv_j \not\in  \widehat{G}_{i+1}$ and let $z$ denote $\widehat{\PP}\big( e \in \widehat{O}_{i+1}\, |\, \cF_j \big)$ on this event, i.e.\ $z = \widehat{\P}(e \in \widehat{O}_{i+1} \,|\, \cF_{j-1}, \cE)$. By the tower property of conditional expectation 
\[(1-p(x v_j))z + p(xv_j) \cdot 0 = \widehat{\EE}\big[ \, \widehat{\PP}\big[ e\in \widehat{O}_{i+1} | \cF_j \big] \big| \cF_{j-1} \big] = \widehat{\PP}(e \in \widehat{O}_i | \cF_{j-1} ).\]
Solving for $z$ completes the proof.  
\end{proof}

Similarly for two open edges $e,f \in O_i$ we define 
\begin{equation}\label{eq:def-p(e,f)} p(e,f) = \widehat{\PP}\big( e,f \in \widehat{G}_{i+1} \big)\quad \text{ and note that } \quad p(e,f) \leq \min\big\{p_{i+1}^2 , n^{4\beta}/|N^{\circ}_i(x) \cap I|^2 \big\} . \end{equation}
Similar to the above, we now observe the following. 
\begin{observation}\label{obs:P(einO)-movement-01}
If $e,v_j$ forms an open triangle, where $e = xy\in C\cap O_i$ we have
\[ \widehat{\PP}\big( e \in {\widehat{O}}_{i+1}\, \big|\, \cF_j \big) = \begin{cases} (1-p(xv_j,yv_j))^{-1} \PP( e \in {\widehat{O}}_{i+1} | \cF_{j-1} ) \qquad \text{ if  } xv_j \text{ or }yv_j \notin \widehat G_{i+1} \\   0 \hspace{16.6em} \text{ if } xv_j \text{ and } yv_j \in \widehat G_{i+1}.\end{cases}\]
\end{observation}

We can now quickly deduce the following bounds on the increments $I_e$. 

\begin{observation}\label{obs:P(einO)-movement} For each $e = xy \in C\cap O_i$. If $ev_j$ forms a clopen triangle with $xv_j \in O_i$
\begin{equation}\label{eq:type1-increment} - \1(xv_j \in \widehat{G}_{i+1} ) \leq  I_e \leq {2p( xv_j )}. \end{equation}
If $ev_j$ forms an open triangle, we have 
\begin{equation}\label{eq:type2-increment} - \1(xv_j \text{ and } yv_j \in \widehat{G}_{i+1} ) \leq  I_e \leq {2p(xv_j,yv_j)}. \end{equation}
\end{observation}
\begin{proof}
If $ev_j$ forms a clopen triangle with $xv_j$ then by Observation~\ref{obs:P(einO)-movement-0}
\[
I_e= \widehat{\PP}\big( e \in \widehat{O}_{i+1}\, |\, \cF_j \big)\left(\frac{1-\1(xv_j \in \widehat{G}_{i+1} )}{1-p(xv_j)} -1\right)\, .
\]
The bounds on $I_e$ now follow, where for the upper bound we note that $p(x v_j)=o(1)$ by~\eqref{eq:def-p(e)}.  The second statement follows in an identical way from Observation~\ref{obs:P(einO)-movement-01} and~\eqref{eq:def-p(e,f)}.
\end{proof}

We now make two observations that will allow us to bound the squares of the increments in Freedman's inequality.

\begin{lemma}\label{obs:max-xi} For all $j$, we have $ |\xi_j| \leq  5n^{4\beta}$. 
\end{lemma}
\begin{proof}
Write $\xi_j = \xi_j^+ -  \xi_j^-$, where $\xi_j^+ = \max\{\xi_j, 0 \}$ and $\xi_j^- = \max\{-\xi_j, 0\}$. We prove the lemma in two parts. For each $e \in C \cap O_i$ for which either $ev_j$ forms an open or clopen triangle, write $e = x_ey_e$ and we will assume that $x_ev_j \in O_i$, without loss.

Using Observation~\ref{obs:P(einO)-movement}, we have 
\[ \xi_j^+ \leq  \sum_e {2p(x_ev_j)} \1\big( ev_j \text{ clopen}  \big)  + \sum_e {2p(x_ev_j, y_ev_j)}\1\big( ev_j \text{ open}  \big), \]  where each sum is over all $e \in C \cap O_i$.  Since $v_j$ is in at most $|N'_{\leq i}(v_j)\cap I ||N_{ i}^{\circ}(v_j)\cap I |$ clopen triangles with edges in $I^{(2)}$ and in at most $|N_{i}^{\circ}(v_j)\cap I |^2$ open triangles, we have that the above is at most
\[ \max_e\big\{  2p(x_ev_j)n^{2\beta}|N_{i}^{\circ}(v_j) \cap I | + 2p(x_ev_j, y_ev_j) |N_{i}^{\circ}(v_j) \cap I |^2 \big\} \leq 4n^{4\beta}\, ,  \]
where the max is over $e\in C \cap O_i$. For the inequality we used that $p(x_e v_j)=0$ if $|N'_{\leq i}(v_j)\cap I|>n^{2\beta}$ by definition of the measure $\widehat \P=\widehat{\PP}_{I,L}$. The final inequality follows from the bounds at \eqref{eq:def-p(e)} and \eqref{eq:def-p(e,f)}.

To bound $\xi_j^{-}$ we let $\widehat{N}_{\leq i+1}(v)$ denote the neighbourhood of $v$ in the graph $G'_{\leq i}\cup \widehat G_{i+1}$ and notice that
    \[ \xi_j^- \leq |\widehat{N}_{\leq i+1}(v_j) \cap I|^2 \leq n^{4\beta} ,\] 
where the second inequality holds by the definition of the measure $\widehat{\PP}$. \end{proof}

We now bound the sum of the $|\xi_i|$ using our bounds on $X_i(e)$ and $Y_i(e)$. 

\begin{lemma}\label{obs:sum|x_i|} On the event $\cA_i$, we have  
\begin{equation}\label{eq:sum|xi|} \sum_{j=1}^n \widehat{\EE}\big[ \, |\xi_j| \, | \, \cF_{j-1} \big]  \leq |C \cap O_i|.\end{equation}
\end{lemma}
\begin{proof} As above, write $\xi = \xi^+ - \xi^-$. It is enough to show $\widehat{\EE}[ \xi^+_j | \cF_{j-1} ] \leq  |C \cap O_i|/2 $, since $|\xi| = \xi^+ + \xi^-$ and, by the martingale property, we have 
$\widehat{\EE}[ \xi_j^+ | \cF_{j-1}] = \widehat{\EE}[ \xi_j^- | \cF_{j-1}]$. 

We simply bound the expectation by its maximum possible value.   Observation~\ref{obs:P(einO)-movement} along with the inequalities at \eqref{eq:def-p(e)} and \eqref{eq:def-p(e,f)} shows $$\sum_{j=1}^n\, \widehat{\EE}[ \xi_j^{+} | \cF_{j-1}] \leq {2p_{i+1}}\sum_{j}\sum_{e\in C\cap O_i} \1( ev_j \text{ clopen} ) +  {2p_{i+1}^2}\sum_{j}\sum_{e\in C\cap O_i} \1( ev_j \text{ open} )\,.$$ 
This is at most 
\[ {2\sum_{e\in C\cap O_i}} p_{i+1}^2X_i(e) + p_{i+1}Y_i(e)\, \leq O( |C \cap O_i| \cdot \gamma \sqrt{\log n}) \leq |C \cap O_i|/2  \]
where the first inequality is via Observation \ref{lem:X(e)-Y(e)-good} using that $\cA_i$ holds.
\end{proof}

We now state Freedman's inequality \cite{freedman1975tail}.

\begin{theorem}\label{thm:Freedman}
	Let $(M_j)_{j=1}^{\tau}$  be a martingale with respect to filtration $(\mathcal{F}_j)_{j}$ with increments $(\xi_j)_{j=1}^{\tau}$ which satisfy $|\xi_j| \leq R$ and $\E[ \xi_j^2 \,|\, \cF_{j-1}] \leq \sigma_j^2$. Then for all $t \geq 0$ we have 
	\[
	\P(| M_{\tau}-M_0| \geq t ) \leq 2\exp\left(-\frac{t^2}{2 (b + Rt)} \right)\, ,
	\] where $b = \sum_{j} \sigma_j^2$.
\end{theorem}

We are now in a position to prove Lemma~\ref{lem:martinglae-application} on the concentration of the capped process. 

\begin{proof}[Proof of Lemma~\ref{lem:martinglae-application}]
We use the martingale defined at \eqref{eq:def-martingale}. However this martingale does not quite determine $Z = |C \cap \widehat{O}_{i+1}|$. For this we need to take the ``regularization step'' into account, i.e. we need to account for the set $Q_{i+1}$ in~\eqref{label:hat-O-def}. So after exposing $\bigcup_i \cF_i$, we order the edges of $C \cap O_i$, $e(1),e(2),\ldots $ and in a $j$th additional step we determine if  $e(j) \in Q_{i+1}$. This defines $\leq |C \cap O_{i}|$ additional steps in the martingale after which $Z$ is determined. We let $\xi_{e(j)}$ denote the increments of these steps and note we have $|\xi_{e(j)}| \leq 1$ and (trivially) $\EE[|\xi_{e(j)}|^2 | \cF_{e(j-1)} ]\leq 1$.

Thus, using this and Lemma~\ref{obs:max-xi} tells us that {$|\xi_j| \leq 5n^{4\beta}$} for \emph{all} steps $j$. By Lemma~\ref{obs:max-xi} and Lemma~\ref{obs:sum|x_i|} we can control the sum of the variances 
\[ \sum \widehat{\EE}\big[ \xi_j^2 | \cF_{j-1} \big]   + \sum \EE\big[ \xi_{e(j)}^2 | \cF_{e(j-1)} \big]  \leq 5n^{4\beta} \cdot \sum \widehat{\EE}\big[ |\xi_j| | \cF_{j-1} \big] + |C \cap O_i| \leq {6n^{4\beta}}|C \cap O_i | .\]
We now apply Freedman's inequality, Theorem~\ref{thm:Freedman}, to obtain
\[\widehat{\PP}\big(  |C \cap \widehat{O}_{i+1}|  \leq  \widehat{\EE}\, |C \cap \widehat{O}_{i+1}| - t \big)  \leq 2\exp\left(-\frac{t^2}{{12 n^{4\beta}}(|C \cap O_i| + t)} \right),\]
We now set $t =  \eps s_{i+1}|C \cap O_{i}|$ and use Observation~\ref{obs:P-hatP-means} to see
\[ \widehat{\PP}\big(  |C \cap \widehat{O}_{i+1}|  \leq (1-\eps)  s_{i+1}|C \cap O_{i}|\big) \leq \exp\big( - \eps^2|C \cap O_i| {n^{-5\beta }} \big),\]
as desired. \end{proof}

\vspace{2mm}

\subsection{Proof of Lemmas~\ref{lem:single-step-B-failure} and \ref{lem:single-step-B-failure-0}}

We first prove Lemma~\ref{lem:single-step-B-failure} which gives us concentration for open edges for a general step $1\leq i \leq T$. We then sketch the proof Lemma~\ref{lem:single-step-B-failure-0}, which is the statement for the first step of the process $i = 0$.

\begin{proof}[Proof of Lemma~\ref{lem:single-step-B-failure}] 
For $0\leq i\leq T-1$, we use Lemma~\ref{lem:P-P-tilde-domination} to bound the quantity
\[ \PP_{i+1}\big( |C \cap O_{i+1}| < (1-\gamma^4)s_{i+1}|C \cap O_{i}| \wedge \cT_{i}(L,C)   \wedge \cL_{i}(I,L) \big) \] above by
\[\leq \widehat{\PP}\big(  |C \cap \widehat{O}_{i+1}| \leq (1-\gamma^4)s_{i+1}|C \cap O_{i}| + k^{1+8\beta} \big) \leq \widehat{\PP}\big(  |C \cap \widehat{O}_{i+1}| \leq (1-\gamma^5 )s_{i+1}|C \cap O_{i}| \big),\]
where the second inequality holds by the assumption $|C \cap O_{i}| \geq \theta_i\gamma^2 k^2\geq k^{1+\delta/2}$ (recall Observation~\ref{lem:pi-UB} and the fact that $\beta<\delta/2^7$). Now apply Lemma~\ref{lem:martinglae-application} to the right hand side to obtain
\[ \widehat{\PP}\big(  |C \cap \widehat{O}_{i+1}| \leq (1-\gamma^5 )s_{i+1}|C \cap O_{i+1}| \big) \leq \exp(- |C \cap O_i| {n^{-5\beta+o(1)}} ) \leq \exp(- k^{1+\delta/4}),\]
where we again used $\beta<\delta/2^7$, and $|C \cap O_i|  \geq k^{1+\delta/2}$, when $i \leq T$.
\end{proof}

In the case $i=0$, that is for the seed graph, there is a very slight additional complication in that we are working with a blow up. 

\begin{proof}[Proof of Lemma~\ref{lem:single-step-B-failure-0}]
From Section~\ref{subsec:seed}, recall that $G_{\ast}' \sim G(n/r,p_0)$. We define 
\[ O_\ast = \{e \in K_{n/r}\setminus G_{\ast}' : e \text{ does not form a triangle with } G_\ast' \} \quad \text{ so that } \quad \pi(O_0) = \pi(O_0') = O_\ast,\] where $O_0,O_0'$ were defined at \eqref{eq:def-O'_0}. 
We first prove \eqref{eq:single-step-B-failure-01} in a very similar way to the above. Indeed, we bound 
\[ \PP\big( |\pi(C) \cap O_\ast| < (1-\gamma^4)r^2 \theta_0 |\pi(C)| \wedge \cT_{0}(L,C)   \wedge \cL_{0}(I,L) \big), \] above by
\[ \leq \widehat{\PP}\big(  |\pi(C) \cap \widehat{O}_\ast| \leq (1-\gamma^5 )r^2 \theta_0 |\pi(C)| \big)  \leq \exp\left(-  |\pi(C)| {n^{-5\beta + o(1)}} \right) \leq \exp(-k^{3/2}), \] where the first inequality holds by Lemma~\ref{lem:P-P-tilde-domination} and the second holds by Lemma~\ref{lem:martinglae-application} and the final inequality uses that $|\pi(C)|\geq r^{-2}|C| \geq r^{-2}\g^2 k^2$. This proves \eqref{eq:single-step-B-failure-01}. 

To prove \eqref{eq:single-step-B-failure-0}, we repeat this argument but we now weight the edges based on $C$. In particular, set $w(e) = |C[V_x,V_y]|$, for each $e \in \pi(I)^{(2)}$ and note that $0 \leq w(e)\leq r^2$. 

We now follow the proof of Lemma~\ref{lem:single-step-B-failure} while allowing for (bounded) weights on the edges: indeed we are looking to prove that the random variable 
\[ Z = |C \cap O_0'| = \sum_{e \in \pi(I) } w(e)\1( e \in O_\ast ), \] is concentrated. Here we proceed as above 
by passing to the capped measure 
\[ \widehat{\PP}\big(  |C \cap \widehat{O}_{0}| \leq (1-\gamma^4)r^2\theta_0|C| + k^{1+8\beta} \big) \leq \widehat{\PP}\big(  |C \cap \widehat{O}_{0}| \leq (1-\gamma^5)r^2\theta_0|C| \big), \]
and then applying the same analysis of the martingale, as in the proof of Lemma~\ref{lem:martinglae-application}, inside the graph $G_{\ast}'$. The only difference is in the analysis of the increments where we have
\[ \xi_j = \sum_{e \in \pi(I)} w(e) \bigg( \widehat{\PP}\big( e \in \widehat{O}_{\ast}\, |\, \cF_j \big) - \widehat{\PP}\big( e \in \widehat{O}_{\ast} \, | \, \cF_{j-1} \big) \bigg), \]
where $\widehat{O}_{\ast}$ is $O_{\ast}$ relative to the capped measure. We can now bound $|\xi_j| \leq r^2n^{2\beta} = n^{2\beta+o(1)}$ and can obtain concentration about the mean for the number of edges in $|C \cap O_0'|$ as in the proof of Lemma~\ref{lem:martinglae-application}.

There is one more step to pass from $O_0'$ to $O_0$ where we select a random edge between each open pair of parts $V_{i},V_{j}$. In particular, we condition on $O_0'$ which satisfies $|C \cap O_0'| \geq (1-\gamma^5) r^2\theta_0|C|$ and express {$$|C \cap O_0| = \sum_{e \in \pi(I) } \zeta(e)\1( e \in O_\ast ), $$
where $\zeta(e)$ are mutually independent Bernoulli random variables with $\E\, \zeta(e) = w(e)/r^2$}.  In particular, $\E[|C \cap O_0| \,|\,O_0'] = r^{-2} |C \cap O_0'|\,.$  By the Chernoff inequality, we have \begin{equation*}
    \P\Big( \big| |C \cap O_0| -  r^{-2} |C \cap O_0'| \big| \geq  \gamma^5\theta_0 |C|   \,\Big|\, O_0'\Big) \leq \exp\left( - n^{-o(1)}\theta_0 |C| \right) \leq \exp(-k^{3/2}),
\end{equation*}
where we used $\theta_0|C|\geq r^{-2}\gamma^2 k^2$. Chaining these two concentration inequalities together, completes the proof.
\end{proof}

\section{The Probability a \texorpdfstring{$k$}{k}-set is independent}\label{sec:prob-k-set-independent}
In this section, we fix $I \in [n]^{(k)}$ and bound the probability that $I$ is independent in the final graph $G_{\leq T}$, on a quasi-randomness event $\cA'$. Here we define $\cA'$ to be the event 
\begin{equation}\label{eq:def-A'} \cA' = \cA \wedge \cC \wedge \cF,\end{equation}
where $\cA, \cC,\cF$ are the events that we showed hold with high probability in Sections~\ref{sec:density-of-process}, \ref{sec:moderate-degrees} and \ref{sec:open-edges-in-I-redux}, respectively. The following is the main objective of this section.

\begin{theorem}\label{thm:prob-ind-final}
If $I \in [n]^{(k)}$ and $|\pi(I)|=\ell$, then 
$$\PP(I\in \cI(G_{\leq T}) \wedge \cA' )\leq \hspace{-0.5mm} \exp \hspace{-1mm} \left[- p_1\binom{k}{2}-p_0\binom{\ell}{2}+\frac{1}{2}(k-\ell)p_1\min\{k-\ell, p_0n\} +10\beta k \log n\right] \, .$$ 
\end{theorem}

The reader should understand this bound in the context of our discussion in the heuristics section, Section~\ref{sec:heuristic}. In particular, Theorem~\ref{thm:prob-ind-final} parallels our bound at \eqref{eq:Ikell}. The third term in the exponent quantifies the loss from edges which are closed in the seed step of the process by a vertex of large degree to $I$, as we discussed at \eqref{eq:funky-Is}. The fourth term in the exponent should be seen as a negligible error term;  recall $\beta < \delta/2^7$.

Henceforth we fix $I \in [n]^{(k)}$, let  $\ell = |\pi(I)|$, where $k/r\leq \ell\leq k$, and turn to the proof of Theorem~\ref{thm:prob-ind-final}. Here, the main challenge is to understand how the cores $C_i(I)$ evolve over the course of the process. 
To this end we use an analogue of $L(I)$ for vertices with very high degree to $I$. Define $L'(I) \subset L(I)$ by 
\[ L'(I):= \big\{  x \in [n] :\, |N'_{ \leq T }(x)\cap I |\geq k n^{-2\beta}\big\}\, . \]
 We also define 
\[ H_i(I) =G'_{i}[L'(I),I], \] 
for all $i\in [0,T]$. We will see that the sets $H_i(I)$ determine the size of the core $C_i(I)$ (defined at~\eqref{eq:core-def}), up to lower order terms. While we are not able to union bound over all choices of core, as we did in Section~\ref{sec:open-edges-in-I-redux}, we remove much of the randomness by union bounding over the choice of $H_0(I),H_1(I),\ldots H_T(I)$ and, importantly, the choice of $L'(I),L(I)$. To this end, let us define the event 
\[ \cH(L,L',H_0,\ldots ,H_{T}) = \big\{ L(I) = L, L'(I)=L', \text{ and } H_i(I) = H_i \text{ for all } i   \big\}. \]

We record the following ``union bound'' over all possible choices of $L,L'$, and the $H_i$. Since $H_0$ is determined by the first step of our process we treat it differently to $H_1, \ldots, H_T$ in what follows. We let
\[
f(H_0,\ldots, H_T)=e(\pi(H_0))+e(H_1)+\ldots+e(H_T).
\]

\begin{lemma}\label{obs:ind-set-worst-case} 
If $I \in [n]^{(k)}$, then 
\begin{equation}\label{eq:fix-H_i}\PP\big(I\in \cI(G_{\leq T} ) \wedge \cA' \big)  \leq  \max_{L,L', R}\, n^{4\beta R +o(k)}\max_{H_i :\,  f((H_i)_i) = R } \PP\big( I \in \cI(G_{\leq T}) \wedge \cH(L,L', (H_i)_i) \, \wedge \cA' \, \big)  .\end{equation}
where the first maximum is over all sets $L'\subset L \subset [n]$, with $|L| \leq 2kn^{-\beta}$, $|L'|\leq 2n^{3\beta}$ and integers $0\leq R \leq kn^{\beta}$, and the second maximum is over all tuples $(H_0,\ldots,H_T)$ of mutually edge-disjoint graphs for which $f(H_0,\ldots,H_T) = R$.
\end{lemma}
\begin{proof} 
First note that since $\cC \supset \cA'$ holds, we have $e(G'_{\leq T}[I,L(I)]) \leq kn^{\beta}$. Since 
\[ |L(I)|n^{2\beta}\leq 2e(G'_{\leq T}[I,L(I)]) \qquad \text{ and } \qquad |L'(I)|kn^{-2\beta}\leq 2e(G'_{\leq T}[I,L(I)]) \] we have $|L(I)|\leq 2kn^{-\beta}$ and $|L'(I)|\leq 2n^{3\beta}$, on the event $\cA'$. Note also that \[ H_0(I)\cup \ldots \cup H_T(I)= G'_{\leq T}[L'(I),I] \] and so $f(H_0(I),\ldots,H_T(I))\leq kn^{\beta}$. We can therefore bound $\PP\big(I\in \cI(G_{\leq T} ) \wedge \cA' \big)$ above by
\begin{equation} \label{eq:obs-ind-set-worst}
\leq \sum_{L : |L|\leq 2kn^{-\beta}}\sum_{ L':|L'|\leq2n^{3\beta}}\sum_{R=0}^{kn^{\beta}} \sum_{H_i: f((H_i)_i)=R}  \PP\big( I \in \cI(G_{\leq T}) \wedge \cH(L,L', H_0,\ldots ,H_{T}) \, \wedge \cA' \, \big), \end{equation}
where the inner sum is over all  tuples $(H_0,\ldots,H_T)$ of mutually edge-disjoint graphs for which $f(H_0,\ldots,H_T) = R$.

We now count the number of terms in the sum. There are at most  $n^{2k n^{-\beta}+2n^{3\beta}} = n^{o(k)}$ choices for $L,L'$. We now fix $s\leq R\leq kn^{\beta}$ and and count tuples $(H_0, \ldots, H_T)$ such that $e(\pi(H_0))=s$ and $e(H)=R-s$ where $H=H_1\cup\ldots\cup H_T$. We first observe that $H_0$ is determined by $\pi(H_0)$. Since each edge of $\pi(H_0)$ has one endpoint in $\pi(I)$ and the other in $\pi(L')$, we conclude that there are  at most 
\[
    \binom{2n^{3\beta}k}{s}\leq \left(\frac{2en^{3\beta}k}{s} \right)^{s}\leq n^{3\beta s+o(k)}\, 
    \]
    choices for $\pi(H_0)$ where for the last inequality we used that $(2ek/x)^x$ is maximized at $x=2k$. Similarly the number of choices for $H=H_1\cup\ldots \cup H_T$ is at most
  \[
 \binom{k|L'|}{R-s} \leq \binom{2 n^{3\beta}k}{R-s} \leq n^{3 \beta (R-s)+o(k)}\, .
  \]
  We then note there are $T^{(R-s)} \leq n^{\beta R}$ ways to partition $H = H_0\cup \cdots \cup H_T$ (recall that $T=O(\gamma^{-1} \sqrt{\log n})$). Combining these bounds and summing over $s$ (for which there is at most $kn^{\beta}=n^{o(k)}$ choices) we conclude that there are at most $n^{4\beta R+o(k)}$ choices for $(H_0,\ldots, H_T)$ such that $f(H_0,\ldots, H_T)=R$.
  We now use this to bound the sum \eqref{eq:obs-ind-set-worst} and complete the proof. 
\end{proof}

In what follows, we fix $L,L',R$, and $H_0,\ldots, H_T$ satisfying the conditions of Lemma~\ref{obs:ind-set-worst-case} and let $\cH = \cH(L,L',H_0,\ldots,H_T)$. We now write the probability on the right-hand-side of Lemma~\ref{obs:ind-set-worst-case} 
in terms of the product of contributions that come from each step. For this, we need to introduce several useful definitions and notations. These definitions are similar to ones we have seen before and should be thought of as versions of these definitions where $H_i$, $L,L'$ are determined. 

In particular, we define the ``core with respect to $L$'',
\[
C_{i,L}(I) = I^{(2)} - \bigcup_{x \in L} (N_{\leq i}'(x))^{(2)}\, ,
\]
and the events
\[
 \cB_{i,L}(I) = \big\{ |C_{i,L}(I) \cap O_i| \geq (1-\eps_i) \theta_i |C_{i,L}(I)| \big\}
 \qquad \text{ and } \qquad
 \cD_{i,L}(I) = \big\{ |C_{i,L}(I)| \leq  \gamma^2 k^2 \big\}
 \] and the events
 \begin{equation}\label{eq:cCLL'-def}  \cC_{i,L,L'}(I)=\bigg\{e(G'_{\leq i}[L,I])\leq kn^\beta \text{ and }|N'_{ \leq i }(x)\cap I |\leq k n^{-2\beta} \text{ for all }x\notin L'\bigg\}\, .\end{equation} We also define 
\[
 \cB^\pi_{0,L}(I) = \big\{ |\pi(C_{0,L}(I) \cap O_0)| \geq (1-\gamma^4) r^2\theta_0 |\pi(C_{0,L}(I))| \big\}\, .
\]

We now define an amalgamation of these events: let
\[\cE_0 = \big\{ I\in \cI(G_{0})\big\}\wedge\cA_0 \wedge (\cB^\pi_{0,L}(I) \vee \cD_{0,L}(I) )\wedge \cC_{0,L,L'}(I) \wedge  \{ H_0 \subset G_0' \}\, ,  \]
and for $1\leq i \leq T$ define the event
\begin{equation}\label{eq:def-E_iL(H)}\cE_i = \big\{ I\in \cI(G_{\leq i})\big\} \wedge \cA_i\wedge (\cB_{i,L}(I) \vee \cD_{i,L}(I) ) \wedge \cC_{i,L,L'}(I) \wedge   \{ H_i \subset G_i' \}.  \end{equation}
Now, we break up the probability on the right hand side of Lemma~\ref{obs:ind-set-worst-case}
into a product over steps. 
\begin{lemma}\label{obs:ind-set-product-over-steps} We have 
\begin{equation}\label{eq:ind-set-prod} \PP\big( I \in \cI(G_{\leq T}) \wedge \, \cH \, \wedge \cA' \big) \leq \PP\big(\cE_{0}\big) \cdot \prod_{i=0}^{T-1} \max\, \PP_{i+1}\big(\cE_{i+1}\big)\, ,\end{equation}
where the maximum is over all realizations of $G_{\leq i}, G_{\leq i}', O_i$. 
\end{lemma}
\begin{proof} We first note the implications \[
I \in \cI(G_{\leq T})\, \Longrightarrow \, \bigwedge_{i\leq T} \{I \in \cI(G_{\leq i})\}\] and
\[\cH \wedge \cA'\, \Longrightarrow \, \big(\cB^\pi_{0,L}(I) \vee \cD_{0,L}(I)\big) \, \wedge \, \cC_{0,L,L'}(I)\, \wedge \, \bigwedge_{1\leq i\leq T}\big(\cB_{i,L}(I) \vee \cD_{i,L}(I)\big)\, \wedge\,  \cC_{i,L,L'}(I)\, .\] The first implication holds since $G_{\leq i}\subset G_{\leq T}$ for all $i$. The second implication holds since 1) $\cA'$ implies $(\cB^\pi_{0}(I) \vee \cD_{0}(I))\wedge \cC_{0}(I)$ and $(\cB_i(I) \vee \cD_i(I))\wedge \cC_i(I)$ for all $1\leq i\leq T$ and 2) $\cH$ implies that $L'(I)=L'$ and $L(I) = L$ so in particular $C_{i,L}(I) = C_i(I)$. Finally note that $\cH$ implies $H_i \subset G_i'$ for all $i\leq T$ and $\cA'$ implies $\cA_i$ for all $i\leq T$. 

Putting these implications together gives the inequality
\[ \PP\big( I \in \cI(G_{\leq T}) \wedge \, \cH \, \wedge \cA' \big) \leq 
\PP\big( \cE_0 \wedge \cdots \wedge \cE_T \big)\, .\]
Finally note that 
\[ \PP\big( \cE_0 \wedge \cdots \wedge \cE_T \big)= \prod_{i=0}^T \PP_i\big( \cE_i\, \big|\, \cE_1 \wedge \cdots \wedge \cE_{i-1} \big) \leq \P(\cE_0) \cdot \prod_{i=0}^{T-1} \max \PP_{i+1}\big(\cE_{i+1} \big),\]
as desired. 
\end{proof}

We now turn to bound the probabilities on the right hand side of \eqref{eq:ind-set-prod}. If we were to compute this probability \emph{directly} we would have to know $C_{i,L}(I)$. One could potentially union bound over ``possible'' such sets (as we did in Section~\ref{sec:open-edges-in-I-redux}) but in this case there are far too many such sets. As we prepare for the crucial lemma of this section, Lemma~\ref{lem:prob-ind}, we show that it is enough to \emph{approximate} $C_{i,L}(I)$ with $C_{i,L'}(I)$, which are determined by the $H_i$, which we are thinking of as fixed.

Indeed, the next simple observation says that $|C_{i,L}(I)|$ is essentially determined by $H_0,\ldots, H_T$. Define $\nu_0,\ldots, \nu_T\in [0,1]$ by
\begin{equation}\label{eq:def-nui} |\pi(C_{0,L'}(I))|=(1-\nu_0)\binom{\ell}{2} \qquad \text{ and } \qquad |C_{i,L'}(I)|=(1-\nu_i)\binom{k}{2}, \end{equation} for  $i\geq 1$. Here $C_{i,L'}(I) = I^{(2)} - \bigcup_{x \in L'} (N_{\leq i}'(x))^{(2)}$. Recall that on the event $\cH$, $H_i=G'_{\leq i}[L',I]$ and so $\nu_i$ is determined by the graphs $H_0,\ldots, H_i$. Also recall the definition of $\cC_{i,L,L'}(I)$  \eqref{eq:cCLL'-def}.

\begin{lemma}\label{lem:C_L-vs-L'}
On the event $\cC_{0,L,L'}(I)$, we have 
\[ |\pi(C_{0,L}(I))| \geq \big(1-\nu_0 - n^{-\beta/2}\big)\binom{\ell}{2} .
\]
For $1\leq i \leq T$, on the event $\cC_{i,L,L'}(I)$, we have 
\[ |C_{i,L}(I)| \geq \big(1-\nu_i - n^{-\beta/2}\big)\binom{k}{2}\, .
\]
\end{lemma}
\begin{proof} We prove the lemma for $1\leq i\leq T$ and remark the case $i=0$ is very similar. For $i\geq 1$, first observe that we can sandwich $|C_{i,L}(I)|$ in terms of $|C_{i,L'}(I)|$, $$|C_{i,L'}(I)| \geq |C_{i,L}(I)| \geq |C_{i,L'}(I)| - \frac{1}{2}\sum_{x \in L \setminus L'} |N'_{\leq i}(x)\cap I|^2\,.$$
Now,  $|N'_{\leq i}(x) \cap I|\leq   kn^{-2\beta}$ for all $x\in L\backslash L'$, by definition of the event $\cC_{i,L,L'}(I)$. Moreover, on the event $\cC_{i,L,L'}(I)$, we have $\sum_{x\in L}|N'_{\leq i}(x) \cap I|\leq 2e(G'_{\leq i}[L,I]) \leq 2kn^\beta$ and so \begin{equation*}\sum_{x \in L\backslash L'} |N'_{\leq i}(x) \cap I|^2 \leq kn^{-2\beta} \sum_{x \in L \setminus L'} |N_{\leq i}'(x) \cap I| \leq k n^{-2\beta} \cdot 2 k n^\beta \leq 2k^2 n^{-\beta}\,. \end{equation*} 
    Combining the previous two displayed equations completes the proof in the case $i\geq 1$. For the case $i=0$ the proof is the same except we apply $\pi$ to all of the above sets in the proof (noting that $\pi$ can only reduce the size of a set). We also recall that $\ell\geq k/r$ so that $2k^2n^{-\beta}\leq n^{-\beta/2}\binom{\ell}{2}$.
\end{proof}

We now prove the following important lemma on the probability that the set $I$ remains independent in step $i+1$. Interestingly, this is actually the only place in the paper where we work directly with $G_{\leq i}$ rather than with $G_{\leq i}'$. For this we recall from Section~\ref{sec:definition-of-process} that we defined $G_{\leq i+1}$ by first sampling $G'_{i+1} \sim G(O_i,p_{i+1})$ and then choosing $G_{i+1} \subset G'_{i+1}$ to be a maximal subgraph such that $ G_{\leq i} \cup G_{i+1} $ is triangle-free. In particular, $G_{\leq i+1}$ has the property that $G_{\leq i+1} + e \supset K_3$ for any $e \in G_{i+1}'\setminus G_{i+1}$.

\begin{lemma}\label{lem:prob-ind} For $i\geq 0$ and any realization of $G_{\leq i}, O_i$, we have
\begin{equation}\label{eq:lem-prob-ind} \PP_{i+1}\big( \cE_{i+1}  \big)\leq p_{i+1}^{e(H_{i+1})}\exp\left(-(1-\nu_{i+1}-\gamma)\frac{\gamma}{\sqrt{n}}\binom{k}{2}\right)\, .\end{equation}
Moreover, 
\begin{equation}\label{eq:lem-prob-ind-0}\PP\big( \cE_{0}  \big)\leq p_{0}^{e(\pi(H_{0}))}\exp\left(-(1-\nu_{0}-\gamma)p_0r^2\theta_0\binom{\ell}{2}\right).\end{equation} 
\end{lemma}
\begin{proof} We first focus on proving the inequality \eqref{eq:lem-prob-ind} and after we make a few remarks on how to adapt the proof to \eqref{eq:lem-prob-ind-0}. Note that if $\nu_{i+1} \geq 1-\g$, we only need to show that
\[\PP_{i+1}\big( \cE_{i+1}  \big)\leq p_{i+1}^{e(H_{i+1})}, \]
which holds since $\cE_{i+1} \subset \{ H_{i+1} \subset G_{i+1}' \}$. So we can assume $\nu_{i+1} < 1-\g$. This means that $\cD_{i+1,L}(I)$ fails and therefore we may assume $\cB_{i+1,L}(I)$ holds.

We now reveal the edges of $G_{i+1}'$ in two rounds; first we reveal the edges outside of $I$,
\[G_{I^c}' = G'_{i+1} \cap  ( O_i - I^{(2)} )  \qquad \text{ and then } \qquad G_I' = 
G'_{i+1}\cap O_i\cap I^{(2)} \, ,
\] the edges inside of $I$.  Let $O_I$ denote the set of all pairs in $O_i \cap I^{(2)}$ that are still `open' after the first round. More precisely 
  \[
 O_I:=\big\{e\in O_i \cap I^{(2)} : e \text{ does not form a triangle with } (G'_{i+1} \cup G_{\leq i})\setminus I^{(2)} \big\}\, .
 \]
 We first observe that we must have $G_{i+1}' \cap O_I = \emptyset$ if $I$ remains an independent set in $G_{\leq i+1}$.

\begin{claim}\label{cl:G cap O_i empty}
    If $G'_I\cap O_I\neq \emptyset$ then $I\not\in \cI(G_{\leq i+1})$.
\end{claim}
\noindent\textit{Proof of Claim \ref{cl:G cap O_i empty}}.  
    Let $e\in G'_{I}\cap O_I$. If $e \in G_{\leq i+1}$ we are done, so assume $e\in G'_{i+1}\backslash G_{\leq i+1}$. Since $G_{\leq i+1} \subset G_{\leq i} \cup G'_{i+1}$ is a maximal triangle-free, $G_{\leq i+1} + e \supset K_3$. Since $e \in G_I'\cap O_I$ this triangle cannot contain a vertex outside of $I$. Thus $G_{\leq i+1} + e$ must contain a triangle inside of $I$ and so $I \notin \cI(G_{\leq i+1})$. \hfill $\blacksquare$

\vspace{1em}

We next observe that $O_I$ must be large, otherwise there will not be enough open edges in the next round.

\begin{claim}\label{cl:O_i large}
    On the event $\cB_{i+1,L}(I)$ we have $|O_I| \geq (1 - \nu_{i+1}-\gamma)\theta_{i+1}\binom{k}{2}$.
\end{claim}
\noindent\textit{Proof of Claim \ref{cl:O_i large}}.
   Note $O_I\supset O_{i+1}\cap I^{(2)}$.
   Moreover, on the event $\cB_{i+1,L}(I)$, we have
   \[|C_{i+1,L}(I) \cap O_{i+1}| \geq (1-\eps_{i+1}) \theta_{i+1} |C_{i+1,L}(I)| \geq (1 - \nu_{i+1}-\gamma)\theta_{i+1}\binom{k}{2}, \]
    where the second inequality holds by Lemma \ref{lem:C_L-vs-L'}, the definition of $\nu_{i+1}$ and $\eps_{i+1}=o(\gamma)$. \hfill $\blacksquare$

\vspace{1em}

\noindent Using these two claims and writing $K = (1- \nu_{i+1}-\gamma)\theta_{i+1}\binom{k}{2}$ allows us to see that
\[ \PP_{i+1}\big( \cE_{i+1}\big)  \leq \P_{i+1}\big(H_{i+1} \subset G_{i+1}'\, \wedge \, |O_I| \geq K \, \wedge\, O_I \cap G_{I}' = \emptyset \big). \]
Now write $H_{i+1}$ as a disjoint union $H_{i+1} =F_1\cup F_2$, where $F_1 = H_{i+1} \cap (O_i - I^{(2)})$ and $F_2=H_{i+1}\cap I^{(2)}$. Now realizing the measure $\PP_{i+1}$ as $\EE_{G'_{I^c}}\PP_{G'_I}$, we have the above is at most
\[ p_{i+1}^{|F_1|}\max_{O_I : |O_I| \geq K} \P_{G_{I}}\big(F_2 \subset G_{I}'\, \wedge\, O_I \cap G_{I}' = \emptyset \big) \leq  p_{i+1}^{|F_1|}p_{i+1}^{|F_2|}(1-p_{i+1})^K=p_{i+1}^{e(H_{i+1})}(1-p_{i+1})^K\, .  
\]
Using $1-x\leq e^{-x}$ and recalling the definition of $K$, completes the proof of \eqref{eq:lem-prob-ind}.

The proof of \eqref{eq:lem-prob-ind-0} is very similar. As before we may assume $\nu_{0} < 1-\g$ and, in particular, 
\[
|C_{0,L}(I)|\geq |\pi(C_{0,L}(I))|\geq \gamma^2k^2 
\]
(recall that $\ell\geq k/r$ and $\gamma=o(1/r^2)$)
and so $\cD_{0,L}(I)$ fails. Therefore we may assume $\cB^\pi_{0,L}(I)$ holds. The remainder of the proof is the same as for \eqref{eq:lem-prob-ind} only now we apply $\pi$ to all of the sets and graphs in the argument and we use $\cB_0^\pi(I)$ in place of $\cB_{i+1,L}(I)$.
\end{proof}

We now apply Lemma~\ref{lem:prob-ind} to each step to obtain the following bound.

\begin{corollary}\label{cor:prob-ind-bound-1} We have 
\[   \PP( I \in \cI(G_{\leq T}) \wedge\,  \cH \, \wedge \cA' )  \leq  \exp\left[-(1-o(1))\left( p_1\binom{k}{2}+p_0\binom{\ell}{2}\right)\right]\exp\left(  F\right) ,\]
where we define $F = F(H_0,\ldots, H_T) $ to be the quantity
\begin{equation}\label{eq:def-F}  F =  \nu_0p_0\binom{\ell}{2}+e(\pi(H_0))\log p_0+  \frac{\g}{\sqrt{n}} \binom{k}{2} \sum_{i=1}^{T}\nu_i-  \sum_{i=1}^T  e(H_i) \log (\theta_i \sqrt{n})    .\end{equation}
\end{corollary}
\begin{proof}
Recall that $p_{i} = \g/(\theta_i \sqrt{n}) \leq (\theta_i \sqrt{n})^{-1}$ for $i\geq 1$, $r^2\theta_0=1-o(1)$ by \eqref{eq:theta-0} and that $p_1=\gamma T/\sqrt{n}$. The result follows by applying Lemma~\ref{lem:prob-ind} to each term in the product in Lemma~\ref{obs:ind-set-product-over-steps}.\end{proof}

To complete the proof of Theorem~\ref{thm:prob-ind-final} we need only to control $F = F(H_0,\ldots,H_T)$. This amounts to deriving some basic relationships between the quantities $\nu_i$, $e(H_i)$ and $\theta_i,p_i$. In what follows, we define 
\begin{equation}\label{eq:SH0-def}
S(H_0)=\frac{1}{2}\sum_{x\in L_{\ast}}|N_{H_0}(x)\cap I|^2\,  \qquad \text{ where } \qquad L_{\ast}\subset L' \end{equation} is a set which contains exactly one element of each $V_j$ for which $V_j\cap L'\neq \emptyset$. Recall that $[n] = V_1 \cup \cdots \cup V_{n/r}$ is the partition defined by the blow up in the seed step (see discussion above \eqref{eq:def-pi}).

The following observation is a step towards quantifying the cost of shrinking the core. We control the proportion of pairs removed from $I$ to form the core, that is $\nu_i$ (which we defined at \eqref{eq:def-nui}), in terms of $e(H_i)$.  

\begin{observation}\label{obs:betai-bound}
 For all $i\in[1, T]$, on the event $\cA' \wedge \cH$,
\begin{equation}\label{eq:obs-betai-bound} \nu_i \cdot \binom{k}{2} \leq  (1 + \gamma)\big( i\gamma \sqrt{n}/2 + p_0n \big)\big( e(H_1) + \cdots + e(H_{i}) \big) +  S(H_0)\, . \end{equation}
Moreover,
\begin{equation}\label{eq:obs-betai-bound0}
\nu_0 \cdot \binom{\ell}{2} \leq (1+\gamma)e(\pi(H_0))p_0n/r\,.
\end{equation}
\end{observation}
\begin{proof}
We first prove~\eqref{eq:obs-betai-bound}. Recall that $H_0 \cup \cdots \cup H_{T}$ is a disjoint union and that $H_i \subset G_{i}'$, on the event $\cH$. Let $i\geq 1$ and write $H_{[1,i]} = H_1 \cup \cdots \cup H_i$. Note that
\[ |C_{i,L'}(I)| \geq \binom{k}{2}- \frac{1}{2}\hspace{-0.5mm}\sum_{x \in L'} |N_{H_{[1,i]}}(x)\cap I|^2- \sum_{x \in L'} |N_{H_{[1,i]}}(x)\cap I||N_{H_0}(x)\cap I| -\Big|\hspace{-1mm}  \bigcup_{x\in L'}\hspace{-1mm} N_{H_0}(x)^{(2)}\cap I  \Big|\, . \]
Note that $\pi(x) = \pi(y) $ implies $N_{H_0}(x)=N_{H_0}(y)$ and so $\big |\bigcup_{x\in L'}N_{H_0}(x)^{(2)}\cap I \big|\leq S(H_0)\,$. Recalling that $|C_{i,L'}(I)|=(1-\nu_i)\binom{k}{2}$ by definition, we conclude that
\[ \nu_i \cdot \binom{k}{2} \leq   \frac{1}{2}\sum_{x\in L'} |N_{H_{[1,i]}}(x)\cap I|^2 +  \sum_{x\in L'}|N_{H_{[1,i]}}(x)\cap I||N_{H_0}(x) \cap I| + S(H_0)  .\]
Using the maximum degree bound $\Delta(G_{i}') \leq (1+\eps_i)\gamma \sqrt{n} $, which holds on $\cA$, and the fact that $H_i \subset G_i'$, which holds on $\cH$, we have $\Delta(H_{[1,i]})\leq (1+\eps_i) i\gamma \sqrt{n}$ and thus
\[ \sum_{x\in L'} |N_{H_{[1,i]}}(x)\cap I|^2 \leq  (1+\eps_i)(i\gamma \sqrt{n} )\sum_{x \in L'} |N_{H_{[1,i]}}(x) \cap I|.  \]
This is at most 
\[  \leq (1+\eps_i)(i\gamma \sqrt{n} )( e(H_{[1,i]}) + |L'|^2)  \leq (1+\g)(i\gamma \sqrt{n} )e(H_{[1,i]}), \]
where the first inequality holds since the sum over $L'$ only double counts edges contained in $L'$ and the last equality holds since $|L'|^2 \leq n^{6\beta}$ and $e(H_{[1,i]}) \geq k n^{-\beta}$, when it is non-zero.  Similarly, using that $\Delta(H_0) \leq (1+\eps_0)p_0n$, we have 
\[ \sum_{x\in L'}|N_{H_{[1,i]}}(x)\cap I||N_{H_0}(x)\cap I| \leq (1+\g) p_0n \cdot  e(H_{[1,i]})\, .  \]
Putting these inequalities together proves \eqref{eq:obs-betai-bound}. To prove \eqref{eq:obs-betai-bound0}, we note that
\[
|\pi(C_{0,L'})|\geq \binom{\ell}{2}-\frac{1}{2}\sum_{x\in\pi(L')}|N_{\pi(H_0)}(x)\cap \pi(I)|^2\geq \binom{\ell}{2}-\Delta(\pi(H_0))(e(\pi(H_0)) + n^{6\beta}).
\]
The result follows by noting that $\Delta(\pi(H_0))\leq (1+\eps_0)p_0n/r$.
\end{proof}

We now simply sum over the first inequality from Observation~\ref{obs:betai-bound} to obtain the following. As we are working towards controlling the $\exp(F)$ term in Corollary~\ref{cor:prob-ind-bound-1}, we note the left hand side of the below is a sum in the definition of $F$ in \eqref{eq:def-F}.

\begin{observation}\label{obs:betai-sum-bound}
On the event $\cA' \wedge \cH$ we have 
\[ \frac{\g}{\sqrt{n}} \binom{k}{2} \sum_{i=1}^T \nu_i \leq \frac{\log n }{8}\sum_{i=1}^{T} \big( 1-\big(i/T\big)^2\big)e(H_i) + S(H_0)p_1 \,. \]
\end{observation}
\begin{proof}
Applying Observation~\ref{obs:betai-bound}, we write
\[ \binom{k}{2} \sum_{i=1}^T \nu_i\\
    \leq (1+o(1)) \sum_{i=1}^T \sum_{j = 1}^i \left(\frac{i \gamma \sqrt{n}}{2} + p_0 n\right) e(H_j)  +  S(H_0)T. \]
This is at most 
\[ (1 + o(1))\left(\frac{p_0n T}{2}\sum_{j=1}^T (1-j /T)e(H_i) + \frac{\g \sqrt{n} T^2}{4} \sum_{j=1}^T \big(1- (j/T)^2 \big)e(H_j)\right) +  S(H_0) T\, . \]

Now recall $T = \alpha_1 \gamma^{-1} \sqrt{\log n}$ and $p_0 = \alpha_0 \sqrt{(\log n)/n}$ from \eqref{eq:def-p0} and \eqref{eq:T_0}. We use that $(1-(j/T)) \leq (1-(j/T)^2)$ to see the above is at most
\[ \leq (1 +o(1)) \frac{\sqrt{n} \log n}{4\g} \big( 2\alpha_0\alpha_1  + \alpha_1^2 \big) \sum_{j=1}^T \big(1- (j/T)^2 \big)e(H_j)  +  S(H_0)T  .\]
Now recall that the choice of $\alpha_0,\alpha_1$ satisfies $2\alpha_0\alpha_1 + \alpha_1^2 \leq 1/2 - 2\delta$ by \eqref{eq:alpha_0alpha_1-setup} to conclude the proof of the observation.\end{proof}

We now apply Observations~\ref{obs:betai-bound} and~\ref{obs:betai-sum-bound} to Corollary~\ref{cor:prob-ind-bound-1} to prove the following Corollary, which is very close in form to Theorem~\ref{thm:prob-ind-final}, our main goal of this section. Recall that $R=e(\pi(H_0))+e(H_{[1,T]})$. 

\begin{corollary}\label{cor:prob-ind-bound-2} The quantity $n^{4\beta R}\cdot\PP( I \in \cI(G_{\leq T}) \wedge\,  \cH \, \wedge \cA' )$ is at most  
\[   \exp\bigg[-(1-o(1))\left( p_1\binom{k}{2}+p_0\binom{\ell}{2}\right)-(1/2-8\beta)e(\pi(H_0))\log n+S(H_0)p_1\, \bigg]\, .\]
\end{corollary}

\begin{proof}
We only have to bound $F$ in Observation~\ref{cor:prob-ind-bound-1}. Write $F=F_0+F_1$, where 
\begin{equation}\label{eq:E-in-pf} F_0 =\nu_0p_0\binom{\ell}{2}+e(\pi(H_0))\log p_0 \quad \text{ and } \quad F_1 = \frac{\g}{\sqrt{n}} \binom{k}{2} \sum_{i=1}^{T}   \nu_i-  \sum_{i=1}^T e(H_i) \log (\theta_i \sqrt{n})  . \end{equation}
Note that by Observation~\ref{obs:betai-bound}, we have 
\[
F_0\leq (1+\gamma)e(\pi(H_0))p^2_0n/r+e(\pi(H_0))\log p_0=-(1/2-o(1))e(\pi(H_0))\log n\, .
\]
We now show  
\begin{equation}\label{eq:F_1-bound}
F_1 \leq   S(H_0)p_1-4\beta\log n \cdot 
 e(H_{[1,T]})\, ,
 \end{equation}
from which the result then follows. 
 
So to prove \eqref{eq:F_1-bound}, write $x_i = (i/T)\alpha_1$,  and observe that, by Observation~\ref{obs:betai-sum-bound}, we can bound the first sum in the definition of $F_1$ by
\[ \frac{\g}{\sqrt{n}} \binom{k}{2}\sum_{i=1}^T \nu_i \leq   \frac{\log n }{8\alpha_1^2}\sum_{i=1}^{T} \big( \alpha_1^2 -x_i^2\big)e(H_i)  + S(H_0)p_1 .\]
Now turning to the second sum in the definition of $F_1$, recall that on the event $\cA'$ we have
\[ \theta_i \geq (1 - (i{+1})\eps_i  )  \exp\hspace{-0.5mm}\big(\hspace{-0.5mm} -\hspace{-0.5mm} \g^2 i^2-2\gamma i\alpha_0\sqrt{\log n}\big)\theta_0,  \qquad T =  \alpha_1 \gamma^{-1} \sqrt{\log n}, \]
by Observation \ref{obs:mean-of-theta}
and our choice at \eqref{eq:T_0}. Since $i\eps_i = o(1)$  and $\theta_0 = n^{-o(1)}$ (by \eqref{eq:theta-0}) we have 
\[\log (\theta_i\sqrt{n}) \geq \big( 1/2-(x_i^2 + 2x_i\alpha_0)- o(1)\big)\log n\] and so we have
\begin{equation} \label{eq:F-bound} 
F_1 \leq   S(H_0)p_1  + \frac{\log n }{8\alpha_1^2}\sum_{i=1}^{T} \big(\hspace{-0.5mm}-\hspace{-0.5mm} 3\alpha_1^2-x_i^2  + 8x_i^2\alpha_1^2 + 16\alpha_1^2x_i\alpha_0 + o(1)\big)e(H_i). \end{equation}
We now claim that \begin{equation}\label{eq:numerical-inequality} -3\alpha_1^2 -x_i^2  + 8x_i^2\alpha_1^2 + 16\alpha_1^2x_i\alpha_0 + o(1) \leq -8\delta,\end{equation} which will compete the proof of the \eqref{eq:F_1-bound} (recalling that $\beta<\delta/2^7$) and thus the proof of the lemma. Indeed, if we write $x_i = x\alpha_1$, for $x = i/T \in [0,1]$, we have 
\begin{align} \label{eq:alpha-x-delta}
-3\alpha_1^2 -x_i^2  + 8x_i^2\alpha_1^2 + 16\alpha_1^2x_i\alpha_0 \leq \alpha_1^2(-3 -x^2 + (4 - 16\delta) x) \leq -16\alpha_1^2\delta,
\end{align} for $x \in [0,1]$ and $\delta \in [0,1/8]$. Here we used $\alpha_1^2 + 2\alpha_0 \alpha_1 \leq 1/2 - 2\delta $ for the first inequality. For the second, we used that the quadratic $-3 - x^2 + (4 - 16\delta )x $ in $x$ is increasing for $x \in [0,1]$ and $\delta \in [0,1/8]$.

Inequality \eqref{eq:numerical-inequality} now follows from~\eqref{eq:alpha-x-delta} since $\alpha_1^2<1/2$.  Combining \eqref{eq:numerical-inequality} with \eqref{eq:F-bound} shows
$ F_1 \leq   S(H_0)p_1-\delta\log n \cdot 
 e(H_{[1,T]})$, as desired. This completes the proof.
\end{proof}

To prove Theorem~\ref{thm:prob-ind-final}, all that remains is to determine the maximum of the term
\[ (1/2-8\beta)e(\pi(H_0))\log n+S(H_0)p_1, \]
in the exponent of the bound in Corollary~\ref{cor:prob-ind-bound-2}. We do this with the next two easy lemmas.

 \begin{lemma}\label{lem:lem:Sprime-bd}
Write $\Delta=\max_{x\in L'}|N_{H_0}(x)|$. On the event $\cA_0$ we have
 \[
  \Delta^2/2 \leq S(H_0)\leq (\Delta/2) \big(e(\pi(H_0)) + k - \ell +o(k)\big) \leq (1+o(1))\Delta k/2 \, .
 \]
 \end{lemma}
 \begin{proof}
 The first inequality is trivial so we turn to the second. Let $I_{\ast} \subset I$ be a set which contains exactly one element of each $V_j$ for which $V_j\cap I\neq \emptyset$.
 Observe that $S(H_0)$ is at most
 \begin{equation}\label{eq:lem:Sprime-bd1}
 \leq (\Delta/2)\sum_{x\in L_{\ast} }|N_{H_0}(x) \cap  I|= (\Delta/2)\sum_{x\in L_{\ast}}|N_{H_0}(x) \cap I_{\ast}| + (\Delta/2)\sum_{x\in L_{\ast}} |N_{H_0}(x) \cap  (I \setminus I_{\ast})|,
 \end{equation}
 where we recall the definition of $L_{\ast} \subset L'$ from \eqref{eq:SH0-def}. Call $A$ and $B$ the first and second sum on the right hand side of \eqref{eq:lem:Sprime-bd1}, respectively. To bound $A$, observe that 
 \begin{equation}\label{eq:lem:Sprime-bdA}
 e(\pi(H_0))=\sum_{x\in L_{\ast}}|N_{H_0}(x) \cap I_{\ast}| - O(|L_{\ast} \cap I_{\ast}|^2 ) = \sum_{x\in L_{\ast}}|N_{H_0}(x) \cap I_{\ast}| - o(k) = A - o(k),
 \end{equation}
where the $O(|L_{\ast} \cap I_{\ast}|)$ term takes care of the over-counting of edges within $L_{\ast} \cap I_{\ast}$ and the second equality holds because $|L_{\ast}| \leq |L'| = O(n^{3\beta})$. To bound $B$, we use inclusion-exclusion
 \begin{equation}\label{eq:lem:Sprime-bdB}
 |I \setminus I_{\ast} |\geq \big | \bigcup_{x\in L_{\ast}} N_{H_0}(x) \cap (I\setminus I_{\ast} ) \big|
 \geq 
 \sum_{x\in L_{\ast}}|N_{H_0}(x) \cap  (I \setminus I_{\ast}) | -\sum_{xy\in L_{\ast}^{(2)}}|N_{H_0}(x) \cap N_{H_0}(y)|\, .
 \end{equation}
So, noting that $B$ is the first sum on the right-hand-side of \eqref{eq:lem:Sprime-bdB}, we have 
\begin{align}\label{eq:B-Ubd}
B \leq  |I\setminus I_{\ast}| + \sum_{xy\in L_{\ast}^{(2)}}|N_{H_0}(x) \cap N_{H_0}(y)| \leq (k-\ell) + (\log n)^4 n^{6\beta} = (k-\ell) + o(k), 
\end{align}
where the second to last inequality holds since $|N_{H_0}(x) \cap N_{H_0}(y)| \leq (\log n)^4$ on $\cA_0$ (by \eqref{eq:def-cF_i2}) and $|L_{\ast}| \leq n^{3\beta}$. Putting \eqref{eq:lem:Sprime-bdA} and \eqref{eq:lem:Sprime-bdB} together gives
\[ S(H_0) \leq (\Delta/2)( A + B ) \leq (\Delta/2)\big( e(\pi(H_0)) + k-\ell + o(k) \big),\] as desired.

For the final inequality we note that an identical argument to the one used for~\eqref{eq:B-Ubd} shows that $A\leq \ell+o(k)$. The result then follows from~\eqref{eq:lem:Sprime-bdA}.
\end{proof}

With Corollary~\ref{cor:prob-ind-bound-2} in mind we now prove the following.

\begin{lemma}\label{lem:new-opt}
On the event $\cA_0$, we have 
\[
p_1S(H_0)-\frac{1}{2}e(\pi(H_0))\log n\leq \frac{1}{2}(k-\ell)p_1\min\{k-\ell, \Delta(H_0)\}+ o(k\log n)\, .
\]
\end{lemma}
\begin{proof}
Set $s=S(H_0)$, $t=e(\pi(H_0))$ and $\Delta=\max_{x\in L'}|N_{H_0}(x)|$. Recall $\Delta\leq \Delta(G_0')\leq (1+\gamma)p_0n$ on the event $\cA_0$ (by \eqref{eq:def-cF_i2}) and apply Lemma~\ref{lem:lem:Sprime-bd} to see \begin{equation}\label{eq:lem:new-opt} 2p_1s-t\log n
\leq t(p_1\Delta-\log n)+p_1\Delta(k-\ell)+o(k\log n)\,. \end{equation}
Now we use $\Delta\leq (1+\gamma)p_0n$ along with $\alpha_0 \alpha_1 \leq 1/6$ to see $p_1\Delta-\log n<0$.
 Now by combining the first and second inequalities in Lemma~\ref{lem:lem:Sprime-bd}, we have 
 \[ t\geq \max\{ \Delta-(k-\ell)-o(k), 0\} . \]
 This suggests two cases. Suppose first that $\Delta\geq k-\ell$ and substitute for $t$ to see \eqref{eq:lem:new-opt} is 
\begin{equation}\label{eq:lem:new-opt-1} \leq  p_1\Delta^2- (\Delta-(k-\ell))\log n +o(k \log n)\, .
\end{equation}
Now observe that the quadratic, $p_1\Delta^2-\Delta\log n$ is decreasing for $\Delta \in [0,(1+\gamma)p_0n]$ and therefore, using our assumption that $k-\ell\leq \Delta \leq (1+\gamma)p_0n$, we have
\[
2p_1s-t\log n\leq (k-\ell)^2p_1+o(k \log n)\, \leq (k-\ell)p_1\min\{k-\ell, \Delta\}+o(k \log n)\, .
\]
In the case $\Delta<k-\ell$, we substitute $t = 0 $ in \eqref{eq:lem:new-opt} to obtain
\[ 2p_1s-t\log n \leq p_1 \Delta (k-\ell )+o(k \log n) \leq (k-\ell)p_1\min\{k-\ell, \Delta\}+o(k \log n)\, . \] 
The result follows since $\Delta\leq \Delta(H_0)$.
\end{proof}

Theorem~\ref{thm:prob-ind-final} now follows quickly. 

\begin{proof}[Proof of Theorem~\ref{thm:prob-ind-final}]
Combining Lemma~\ref{obs:ind-set-worst-case}, Corollary~\ref{cor:prob-ind-bound-2} and Lemma~\ref{lem:new-opt} 
 and recalling that $\Delta(H_0)\leq (1+\gamma)p_0n$ allows us to see $\PP( I \in \cI(G_{\leq T}) \wedge  \cA' )$ is at most 
\[
  \exp\bigg[-(1-o(1))\left( p_1\binom{k}{2}+p_0\binom{\ell}{2}\right)+\frac{1}{2}(k-\ell)p_1\min\{k-\ell, p_0n\} + 8\beta e(\pi(H_0))\log n\, \bigg]\, .\]
  The result follows by noting that $e(\pi(H_0))\leq (1+o(1))k$
  by Lemma~\ref{lem:lem:Sprime-bd} and that $p_1\binom{k}{2}+p_0\binom{\ell}{2}=O(k\log n)$.
  \end{proof}

\section{Bounding the independence number and the Proof of Theorem~\ref{thm:r3k}}\label{sec:proof-of-main-thm}

Theorem~\ref{thm:r3k} will be an immediate consequence of Theorem~\ref{thm:main-theorem-ind-set} bounding the independence number of $G_{\leq T}$. Theorem~\ref{thm:main-theorem-ind-set} is itself an easy consequence of Theorem~\ref{thm:prob-ind-final}. Recall that $k =(1+\delta)\sqrt{3 n \log n /2}$.

\begin{proof}[Proof of Theorem~\ref{thm:main-theorem-ind-set}]
We first apply Lemmas~\ref{lem:cA}, \ref{lem:moderate-degrees} and \ref{lem:cB-failure}  to see that 
\[ \PP \big( \alpha( G_{\leq T} ) \geq  k \big)  = \PP \big( \alpha( G_{\leq T} ) \geq  k  \, \wedge \, \cA' \big)  + o(1), \]
where we recall that we defined $\cA'$ by $\cA'= \cA \wedge \cC \wedge \cF$ at \eqref{eq:def-A'}, where 
where $\cA, \cC,\cF$ are the events that we showed hold with high probability in Sections~\ref{sec:density-of-process}, \ref{sec:moderate-degrees} and \ref{sec:open-edges-in-I-redux}, respectively. We now union bound over all $I \in [n]^{(k)}$,
\begin{equation} \label{eq:ind-set-bound--union-bound} \PP \big( \alpha( G_{\leq T} ) \geq  k  \, \wedge \, \cA' \big) \leq \sum_{\ell}\sum_{I}  \PP(I\in \cI(G_{\leq T}) \wedge \cA' ). \end{equation}Here we are summing over $k/r \leq \ell\leq k$ and $I \in [n]^{(k)}$ with $|\pi(I)| = \ell$.
Now note that the number of $I \in [n]^{(k)}$ for which $|\pi(I)| = \ell $ is at most
\[ \binom{n}{\ell}\binom{r\ell}{k-\ell} \leq \binom{n}{\ell} \bigg( \frac{e r\ell}{k} \bigg)^k = \binom{n}{\ell} n^{o(k)}.\] Using this and applying Lemma~\ref{lem:prob-ind}, the right hand side of \eqref{eq:ind-set-bound--union-bound} is at most 
\[ n^{o(k)} \cdot \max_{\ell}\binom{n}{\ell} \cdot \PP(I\in \cI(G_{\leq T}) \wedge \cA') \leq \max_{\ell}\, \exp(f(\ell,k,n)+11\beta k\log n),\]
where we define
\[
f(\ell,k,n)=\frac{1}{2}\ell \log n- p_1\binom{k}{2}-p_0\binom{\ell}{2}+\frac{1}{2}(k-\ell)p_1\min\{k-\ell, p_0n\}\, .
\]
To finish, we now show $f(\ell,k,n) \leq - (\delta/4) k \log n$, for $k/r \leq  \ell \leq k$ and recall that $\beta<\delta/2^7$. We consider two cases. If $\ell\geq k/2$, then using $\min\{k-\ell,p_0n\}\leq k-\ell$ and $p_0=p_1$, by our choice at \eqref{eq:alpha_i-k-selection}, we have
\[
f(\ell,k,n)\leq (\ell \log n)/2-k\ell p_1+o(k\log n) = \ell\big( (\log n)/2 - p_1 k\big) + o(k\log n) \leq - (\delta/4)k\log n\, 
\] where we used $k = (1+\delta)\sqrt{(3n/2)\log n}$, $p_1 = (1-3\delta)\sqrt{\log n / 6n }$. 

On the other hand, if $\ell< k/2$ then using $\min\{k-\ell,p_0n\}\leq p_0n$, we have
\begin{align*}
f(\ell,k,n)
&=\ell\left(\frac{1}{2} \log n-k p_1\right)+\frac{p_1}{2}(k-\ell)(p_0n-(k-\ell))+o(k\log n)\\
&\leq 
\frac{p_1}{2}(k-\ell)(p_0n-(k-\ell))+o(k\log n)\\
&\leq \frac{p_1}{2}(k/2)(p_0n-k/2)+o(k\log n) \leq -\frac{k \log n}{100}\, ,
\end{align*}
where for the first inequality we used $(\log n)/2-p_1n<0$ and in the second inequality we used our assumption that $\ell<k/2$.  
  \end{proof}

We can now easily prove Theorem~\ref{thm:r3k}.

\begin{proof}[Proof of Theorem~\ref{thm:r3k}]
By Theorem~\ref{thm:main-theorem-ind-set}, there exists a triangle-free graph on $n$ vertices with no independent set of size $k=(1+\delta)\sqrt{(3n/2) \log n}$. The result follows by noting that $n= \frac{2k^2}{3(1+\delta)^2\log n}=(1+o(1))\frac{k^2}{3(1+\delta)^2\log k}$ and recalling that $\delta$ is arbitrarily small. 
\end{proof}

\section*{Acknowledgements}
The authors thank Rob Morris and Will Perkins for comments on a previous draft. Matthew Jenssen is supported by a UK Research and Innovation Future Leaders Fellowship MR/W007320/2.  Marcus Michelen is supported in part by NSF CAREER grant DMS-2336788 as well as DMS-2246624. Julian Sahasrabudhe is supported by European Research Council (ERC) Starting Grant ``High Dimensional Probability and Combinatorics'', grant No. 101165900. Part of this work was carried out while JS was visiting the Institute for Advanced Study in Princeton, New Jersey, USA. JS is grateful for the stimulating working environment provided there. The authors would also like to thank the village of Stangvik, Norway for their hospitality.

\appendix

\section{Degrees and codegrees for the seed}\label{sec:seed}

Recall that $G_\ast' \sim G(n/r,p_0)$ and define 
\[ O_\ast = \pi(O_0) = \{ e \in K_{n/r} \setminus G_\ast' :  e \text{ does not form a triangle with } G_{\ast}' \}.\]  
Recall that $O'_0$ is the $r$-blowup of $O_\ast$ (with parts $V_1,\ldots, V_{n/r}$) and $O_0$ is defined by taking a single pair from each subgraph $O_0'[V_i,V_j]$, $i< j$, whenever it is non-empty.

In analogy with our definition of $s_{i}$, define $s_\ast = \P(e \in O_\ast)$ and note that $$s_\ast = (1-p_0)(1 - p_0^2)^{n/r-2} = (1 + n^{-1/2 +o(1)}) \exp(- n p_0^2 / r)\,.$$

We first record the analogue of Lemma~\ref{lem:probability-open'} whose proof we omit as the proof of Lemma~\ref{lem:probability-open'} works verbatim for this setting. 

\begin{lemma}\label{lem:prob-open'-seed}
    Let $H, K \subset K_{n/r}$ be such that $\Delta(H), \Delta(K) \leq n^{\delta/4}$.  Then $$\P(H \subset O_\ast \wedge K \subset G_\ast') \leq (1 + n^{-\delta/4})^{e(H)} s_\ast^{e(H)} p_0^{e(K)}\,.$$
\end{lemma}

\begin{proof}[Proof of Lemma \ref{lem:cA-for-seed}]

 For a given $x\in [n/r]$ we note that $$\P\big( |N_{G_\ast'}(x)| \geq (1 + \eps_0) p_0 (n/r)\big) \leq \exp\left( - \frac{\eps_0^2 p_0 n}{3r}\right) \leq n^{-2}$$
by the Chernoff bound.  Union bounding over $x$ and noting that $\Delta(G_0') = r \Delta(G_\ast')$ establishes {the first part of \eqref{eq:def-cF_i2}} for $i=0$. An analogous argument for degrees establishes the second part of \eqref{eq:def-cF_i2}. 

Turning to the set $O_0$, we first note that $$\E |O_\ast| = (1 + n^{-1/2 + o(1)}) \frac{n^2}{2 r^2} s_\ast\,, \quad \text{ and } \quad \Var(|O_\ast|) \leq n^3$$
and so $|O_\ast - \E |O_\ast| | \leq n^{-\delta} \E |O_\ast|$ with high probability, thus showing \eqref{eq:theta-conc-Ai}.  In particular, we have  $|O_\ast| \geq (1 - n^{-\delta}) s_\ast \frac{n^2}{2 r^2}$ with high probability.  Since $|O_\ast|=|O_0|=\theta_0\binom{n}{2}$,  this shows \begin{equation}\label{eq:theta_0-lower-bound}
\theta_0 \geq (1 - n^{-\delta}) s_\ast r^{-2}
\end{equation}
with high probability, proving \eqref{eq:theta-0}.  Arguing as in Lemmas \ref{lem:cA-open-deg} and \ref{lem:cA-open-codegs}  using Lemma \ref{lem:prob-open'-seed} in place of Lemma \ref{lem:probability-open'} shows that for each $x \neq y \in [n/r]$ we have \begin{equation}\label{eq:Oast-bounds}
    \P(|N_{O_\ast}(x)| \geq (1 + \eps_0/4) s_\ast (n/r) ) \leq n^{-2}\,, \quad  \P(|N_{O_\ast}(x) \cap N_{O_\ast}(y)| \geq (1 + \eps_0/4) s_\ast^2 (n/r) ) \leq n^{-3}\,.
\end{equation}
Note that by the definition of $O_0'$,  $|N_{O_0'}(u)|=r|N_{O^\ast}(x)|$ for all $u\in V_x$ where $V_x$ is the blowup part corresponding to the vertex $x\in [n/r]$. Moreover, by Chernoff's inequality and a union bound we have that $|N_{O_0}(u)|\leq (1+\eps_0/4)|N_{O'_0}(u)|/r^2$ for all $u\in [n]$ with high probability. These observations, combined with \eqref{eq:theta_0-lower-bound} and \eqref{eq:Oast-bounds} establish \eqref{eq:def-cE_i}.  Arguing as in Lemma \ref{lem:mixed-deg} then shows \eqref{eq:cross-degrees}, completing the proof.
\end{proof}

\bibliography{Bib.bib}
\bibliographystyle{abbrv}

\end{document}